\documentclass{amsart}
\usepackage[utf8]{inputenc}

\usepackage{hyperref}

\usepackage{a4wide}
\usepackage{amsmath}
\usepackage{amsfonts}
\usepackage{amssymb}
\usepackage{mathtools}
\usepackage{bbm}
\usepackage{mathrsfs}
\usepackage{amsthm}
\usepackage{verbatim}
\usepackage{upgreek}
\usepackage{relsize}
\usepackage{enumitem}
\usepackage{dsfont}
\usepackage{graphicx}
\usepackage[all]{xy}
\usepackage{url}

\newcommand{\eps}{\varepsilon}
\newcommand{\ov}{\overline}
\newcommand{\id}{\textnormal{id}}
\newcommand{\mc}{\mathcal}

\newcommand{\mrm}{\mathrm}
\newcommand{\mscr}{\mathscr}
\newcommand{\mf}{\mathfrak}
\newcommand{\msf}{\mathsf}
\newcommand{\I}{\mathbbm{1}}

\newcommand{\vp}{\varphi}
\newcommand{\Lvp}{\Lambda_{\vp}}

\newcommand{\md}{\operatorname{d}\!}
\newcommand{\cst}{\ifmmode \mathrm{C}^* \else $\mathrm{C}^*$\fi}

\newcommand{\la}{\langle}
\newcommand{\ra}{\rangle}

\newcommand{\bbGamma}{{\mathpalette\makebbGamma\relax}}
\newcommand{\makebbGamma}[2]{%
  \raisebox{\depth}{\scalebox{1}[-1]{$\mathsurround=0pt#1\mathbb{L}$}}%
}

\DeclareSymbolFont{bbold}{U}{bbold}{m}{n}
\DeclareSymbolFontAlphabet{\mathbbold}{bbold}
\newcommand{\GGamma}{\mathbbold{\Gamma}}
\newcommand{\LLambda}{\mathbbold{\Lambda}}

\newcommand{\NN}{\mathbb{N}}
\newcommand{\RR}{\mathbb{R}}

\newcommand{\CC}{\mathbb{C}}
\newcommand{\ZZ}{\mathbb{Z}}
\newcommand{\GG}{\mathbb{G}}
\newcommand{\HH}{\mathbb{H}}

\newcommand{\EE}{\mathbb{E}}

\newcommand{\vv}{\mathrm{V}}
\newcommand{\ww}{\mathrm{W}}
\newcommand{\WW}{{\mathds{V}\!\!\text{\reflectbox{$\mathds{V}$}}}}

\newcommand*{\D}{\mathcal{D}}

\newcommand{\ismaa}[2]{\langle#1\,|\,#2\rangle}

\newcommand{\wot}{\ifmmode \textsc{wot} \else \textsc{wot}\fi}
\newcommand{\sot}{\ifmmode \textsc{sot} \else \textsc{sot}\fi}
\newcommand{\sots}{\ifmmode \textsc{sot}^* \else \textsc{sot}$^*$\fi}
\newcommand{\ssot}{\ifmmode \sigma\textsc{-sot} \else $\sigma$-\textsc{sot }\fi}
\newcommand{\ssots}{\ifmmode \sigma\textsc{-sot}^* \else $\sigma$-\textsc{sot }$^*$\fi}
\newcommand{\swot}{\ifmmode \sigma\textsc{-wot} \else $\sigma$-\textsc{wot}\fi}

\newcommand{\Linf}{\operatorname{L}^{\infty}(\GG)}
\newcommand{\Linfd}{\operatorname{L}^{\infty}(\whG)}
\newcommand{\Lj}{\operatorname{L}^{1}(\GG)}
\newcommand{\Ljd}{\operatorname{L}^{1}(\whG)}

\newcommand{\CZG}{\mathrm{C}_0(\GG)}
\newcommand{\CZGD}{\mathrm{C}_0(\whG)}

\newcommand*{\Corep}{\mathsf{Corep}}
\newcommand{\one}{1\!\!1}
\newcommand*{\BT}{\mathbf{T}}

\newcommand{\wh}{\widehat}
\newcommand{\whG}{\widehat{\GG}}
\newcommand{\whH}{\widehat{\HH}}

\newcommand{\hvp}{\widehat{\vp}}
\newcommand{\Lhvp}{\Lambda_{\hvp}}

\newcommand{\LdG}{\operatorname{L}^{2}(\GG)}

\newcommand{\oxx}{\bar{\otimes}}
\newcommand{\oon}{\operatorname}

\newcommand{\wt}{\widetilde}

\newcommand{\red}{\mathsf{red}}

\DeclareMathOperator{\lin}{span}
\DeclareMathOperator{\Irr}{Irr}
\DeclareMathOperator{\Pol}{Pol}

\DeclareMathOperator{\Tr}{Tr}
\DeclareMathOperator{\B}{B}
\DeclareMathOperator{\A}{A}

\DeclareMathOperator{\M}{M}

\DeclareMathOperator{\N}{N}

\DeclareMathOperator{\Dom}{Dom}

\DeclareMathOperator{\Mor}{Mor}

\DeclareMathOperator{\LL}{L}

\DeclareMathOperator{\Rep}{Rep}
\DeclareMathOperator{\SU}{SU}

\DeclareMathOperator{\w}{w}
\DeclareMathOperator{\CB}{CB}
\DeclareMathOperator{\Tub}{Tub}

\DeclareMathOperator{\ad}{ad}

\newtheorem{theorem}{Theorem}[section]

\newtheorem{proposition}[theorem]{Proposition}
\newtheorem{lemma}[theorem]{Lemma}
\theoremstyle{definition}
\newtheorem{corollary}[theorem]{Corollary}
\newtheorem{remark}[theorem]{Remark}

\newtheorem{definition}[theorem]{Definition}
\numberwithin{equation}{section}

\begin{document}

\title{The approximation property for locally compact quantum groups}

\author{Matthew Daws}
\address{Department of Mathematics and Statistics,
Lancaster University,
Lancaster,
LA1 4YF,
United Kingdom}
\email{matt.daws@cantab.net}

\author{Jacek Krajczok}
\address{School of Mathematics and Statistics \\
         University of Glasgow \\
         University Place \\
         Glasgow G12 8QQ \\
         United Kingdom 
}
\email{jacek.krajczok@vub.be}

\author{Christian Voigt}
\address{School of Mathematics and Statistics \\
         University of Glasgow \\
         University Place \\
         Glasgow G12 8QQ \\
         United Kingdom 
}
\email{christian.voigt@glasgow.ac.uk}

\thanks{This work was supported by EPSRC grants EP/T03064X/1 and EP/T030992/1.  For the purpose of open access, the authors have applied a CC BY public copyright licence to any Author Accepted Manuscript version arising. 
No data were created or analysed in this study.}

\subjclass[2020]{Primary 46L67, Secondary 22D55, 43A30} 

\keywords{Locally compact quantum groups, approximation property}

\date{}

\begin{abstract}
We study the Haagerup--Kraus approximation property for locally compact quantum groups, generalising and unifying previous work by Kraus--Ruan and Crann.  Along the way we discuss how multipliers of quantum groups interact with the $\mathrm{C}^*$-algebraic theory of locally compact quantum groups.  Several inheritance properties of the approximation property are established in this setting, including passage to quantum subgroups, free products of discrete quantum groups, and duals of double  crossed products.  We also discuss a relation to the weak$^*$ operator approximation property.  For discrete quantum groups, we introduce a central variant of the approximation property, and relate this to a version of the approximation property for rigid $\mathrm{C}^*$-tensor categories, building on work of Arano--De Laat--Wahl.  
\end{abstract}

\maketitle

\tableofcontents

\section{Introduction}
The approximation property (AP) for locally compact groups, introduced by Haagerup and Kraus \cite{HaagerupKraus}, can be viewed as an analogue of Grothendieck's approximation property for Banach spaces \cite{Grothendieck}. It belongs to a family of widely studied analytical properties 
like amenability, weak amenability, and the Haagerup property. In fact, AP is a weakening of weak amenability and thus a very loose form of amenability. 
It is known that AP passes from locally compact groups to their lattices and vice versa, and that it has better permanence properties with respect to standard constructions 
like extensions and free products, in comparison to weak amenability.

It was an open problem for a long time to exhibit examples of exact groups without AP. In the remarkable paper \cite{LafforgueDeLaSalle}, Lafforgue and de la Salle  
proved that $ \oon{SL}(3, \mathbb{R}) $ fails to have AP, thus confirming a conjecture in \cite{HaagerupKraus}. Building on this result, it was shown later 
by Haagerup, Knudby and De Laat that a connected Lie group has AP if and only if all simple factors in its Levi decomposition have real rank at most one 
\cite{HAAGERUP_DELAAT_simpleap}, \cite{HAAGERUP_DELAAT_simpleap2}, \cite{HAAGERUP_KNUDBY_DELAAT_completecharacterization}. 

The approximation property has a wide range of applications. As shown by Haagerup-Kraus \cite{HaagerupKraus}, in the case of discrete groups there is a connection between AP and the slice map property (or equivalently, the operator approximation property) of the associated crossed products. It was recently proven in the full generality of locally compact groups that AP implies exactness \cite{Suzuki} (see also \cite{CrannNeufang}), which makes it relevant to a number of problems in operator algebras.  Let us also mention that AP was shown to be equivalent to a non-commutative version of Fej{\'e}r's theorem \cite{CrannNeufang}, and used to prove results concerning convolution operators on $\oon{L}^p(G)$ \cite{Cowling, DawsSpronk}.

Amenability, weak amenability and the Haagerup property have also been studied extensively in the broader setting of locally compact quantum groups, see \cite{Brannan} 
for a survey. An interesting new feature in the quantum setting is the interplay between discrete quantum groups, their Drinfeld doubles, and 
the associated \cst-tensor categories \cite{DFY_CCAP}. In fact, the central versions of amenability, the Haagerup property, weak amenability and central property (T) for discrete quantum groups have been recast at the level of \cst-tensor categories \cite{PVcstartensor}, thus building a natural bridge to the study 
of subfactors. 

In the present paper we undertake a systematic study of the approximation property for locally compact quantum groups.  Kraus and Ruan introduced a version of the approximation property for Kac algebras in \cite{KrausRuan}, requiring 
the existence of a net in the Fourier algebra $\A(\GG)$ such that the associated net of completely bounded operators on $\Linfd$ converges to the identity in the stable point weak$^*$-topology of $\CB^\sigma(\Linfd)$. Crann studied this property for general locally compact quantum 
groups, and showed for example that in the presence of this property, amenability is equivalent to coamenability of the dual quantum group \cite[Corollary 7.4]{CRANN_inneramenability}.

Our starting point is the original work by Haagerup and Kraus.  We say that a locally compact quantum group has AP if it admits a net of elements in the Fourier algebra $\A(\GG)$ which converges weak$^*$ to $\I$ in the space of left CB multipliers $\M^l_{cb}(\A(\GG))$.  We show that this definition is in fact equivalent to the definition of AP used in \cite{KrausRuan}, \cite{CRANN_inneramenability}, thereby verifying a conjecture in \cite{KrausRuan}.  This result was proved by Haagerup and Kraus in the case of locally compact groups (\cite[Theorem 1.9]{HaagerupKraus}).  Our methods are necessarily different, as we cannot work with points or compactly supported functions.  Along the way, we obtain a useful alternative description of the weak$^*$-topology on the space of left CB multipliers (Theorem \ref{thm5}). We prove that working with left or right multipliers does not change the theory, and that passing from a quantum group to its opposite or commutant preserves AP. We also show that if a quantum group $\GG$ has the AP exhibited by a net which is uniformly bounded in the norm of $\A(\GG)$ (resp.~$\M^l_{cb}(\A(\GG))$), then the dual is coamenable (resp.~ the quantum group is weakly amenable).

We then derive a number of permanence properties of AP in analogy to the classical setting. In particular, we show that AP passes to closed quantum subgroups 
of locally compact quantum groups, and to duals of double crossed products. This includes the passage to direct products of quantum groups 
as a special case. In the setting of discrete quantum groups we verify that AP is inherited by free products and direct limits of directed systems 
of discrete quantum groups with injective connecting maps. We also introduce a central version of the approximation property for discrete quantum groups and show that it is related to a natural notion of AP for rigid \cst-tensor categories, building on work 
of Arano--De Laat--Wahl \cite{ARANO_DELAAT_WAHL_fourier}, \cite{ARANO_DELAAT_WAHL_howemoore}. 

Using our results we can provide a range of examples of quantum groups with and without AP. For instance, the Drinfeld double of $\oon{SL}(3,\RR)$ does not have AP, and the free 
product $\oon{SL}(3,\ZZ)\star \widehat{\oon{SU}_q(2)}$ does not have AP. Examples with AP can be obtained via free products as well, for example $\oon{SL}(2,\ZZ)\star\widehat{\oon{SU}_q(2)}$.  Using infinite products, one may find examples of discrete quantum groups with AP which are not weakly amenable. We make further remarks at the start of Section~\ref{section ap}.

Let us now briefly describe more of our results and explain how the paper is organised. In Section~\ref{section preliminaries} we collect some background material on locally compact quantum groups and fix our notation.
In Section~\ref{section multiplier} we review several characterisations of the space $\M_{cb}^l(\A(\GG))$ of left cb-multipliers of the Fourier algebra $\A(\GG)$ and 
its natural predual $ Q^l(\A(\GG)) $.  By definition, $\M_{cb}^l(\A(\GG))$ is a (in general not closed) subalgebra of $\Linf$, but it is also isomorphic to $\prescript{}{\Ljd}{\CB}^{\sigma}(\Linfd)$, the space of a normal left module maps on $\Linfd$.  Any map in $\prescript{}{\Ljd}{\CB}^{\sigma}(\Linfd)$ restricts to $\CZGD$, and we provide a characterisation of the maps on $\CZGD$ which so arise.  This leads to a description of $ Q^l(\A(\GG)) $ as a quotient of the projective tensor product $\mrm{C}_0(\whG)\wh{\otimes}\LL^1(\whG) $. We were unable to locate this description in the literature, even for classical groups.

In Section~\ref{section ap} we define the Haagerup-Kraus approximation property for locally compact quantum groups, and survey some examples and counter-examples, making use of later results in the paper.  We verify that AP passes to opposite and commutant quantum groups.  
We compare our definition to the version of AP given by Kraus and Ruan, showing that they are equivalent.  We show that $\M^l_{cb}(\A(\GG))$ admits an involution linked to the antipode of $\GG$, and to the fact that elements of $\prescript{}{\Ljd}{\CB}^{\sigma}(\Linfd)$ act boundedly on the Hilbert space $\LL^2(\GG)$.  We finish the section by showing that AP is independent of the choice of working with left or right multipliers.

In Section~\ref{section relation} we discuss the relation of AP with weak amenability and coamenability.

Section~\ref{section discrete} is devoted to the special case of discrete quantum groups.  When $\GGamma$ is discrete, we have the notion of a finitely-supported function leading to the algebra $\mrm{c}_{00}(\GGamma)$.  It suffices to work with $\mrm{c}_{00}(\GGamma)$ when considering AP, and we show further that the approximating net can be chosen to satisfy other properties.  We introduce the central approximation property for discrete quantum groups and prove 
that central AP is equivalent to AP in the unimodular case. Building on the work of Kraus-Ruan \cite{KrausRuan} and Crann \cite{CRANN_inneramenability}, we show that if a locally compact quantum group $\GG$ has AP then the von Neumann algebra $\LL^{\infty}(\whG)$ has W$^*$OAP.  We study the relation between AP of $\GGamma$ and (strong) OAP of $\mrm{C}(\wh{\GGamma})$ or 
W$^*$OAP of $\LL^\infty(\wh{\GGamma})$.  Finally, we introduce strengthenings of these concepts which take into account also the algebra $\ell^\infty(\GGamma)$, and show that these are equivalent to AP even in the non-unimodular case.

In Section \ref{section permanence} we establish a number of permanence properties. We show that the AP is inherited by arbitrary closed quantum subgroups and by the duals 
of double crossed products.  In particular, the direct product of two quantum groups with AP also has AP.  For discrete quantum groups we investigate the passage to free products and direct unions, again showing that AP is preserved.

Finally, in Section~\ref{section categorical} we define the approximation property for rigid \cst-tensor categories and verify that the categorical AP is equivalent 
to the central AP for discrete quantum groups. This implies in particular that the central AP is invariant under monoidal equivalence. We also relate these properties to the AP of the Drinfeld double.

We conclude with some general remarks on notation. If $\mc{A}$ is a $\cst$-algebra we write $\M(\mc{A})$ for its multiplier algebra. 
For a map $\Phi\colon \mc{A}\rightarrow \mc{A}$, the symbol $\Phi^\dagger$ stands for the map $\mc{A}\ni a\mapsto \Phi(a^*)^*\in \mc{A}$. 
If $\omega:\mc{A} \rightarrow \mathbb{C} $ is a linear functional we write $\ov{\omega}$ for the linear functional given by $\ov{\omega}(x)=\overline{\omega(x^*)} $. 

We write $ \odot $ for the algebraic tensor product, $\otimes$ for the tensor product of Hilbert spaces or the minimal tensor product of \cst-algebras, 
$ \check \otimes$ for the injective tensor product of operator spaces and $\bar\otimes$ for the spatial tensor product of von Neumann algebras. 
We denote by $\chi$ the flip map for tensor products of algebras, and use the symbol $\Sigma$ for the flip map of Hilbert spaces. 

We freely use the basic theory of operator spaces, following \cite{EffrosRuan}, see also \cite{Paulsen_CB_Book, Pisier_OS_Book} for example.  When $X,Y$ are operator spaces, $\CB(X,Y)$ denotes the space of completely bounded (CB) linear maps $X\rightarrow Y$. For dual operator spaces $X,Y$ we write $\CB^\sigma(X,Y)$ 
for the subset of $\CB(X,Y)$ consisting of all maps which are weak$^*$-weak$^*$-continuous. 
In the case $X=Y$ we simply write $\CB(X)=\CB(X,X)$ and $\CB^\sigma(X)=\CB^\sigma(X,X)$. If $\M$ is a von Neumann algebra, then $\CB^\sigma(\M)$ can be equipped with the stable point-weak$^*$-topology: $T_i\xrightarrow[i\in I]{}T$ with respect to this topology if and only if $(T_i\otimes \id)x\xrightarrow[i\in I]{} (T\otimes \id)x$ in the weak$^*$-topology, for any separable Hilbert space $\msf{H}$ and $x\in \M\bar{\otimes}\B(\msf{H})$, see \cite{HaagerupKraus}. Whenever we have a left $\N$-module structure on an operator space $X$, the space of 
left $\N$-module CB maps is denoted by $\prescript{}{\N}{\CB}(X)$. Similarly, if $X$ is a right module or a bimodule, the corresponding spaces are denoted 
by $\CB_{\M}(X)$ and $\prescript{}{\N}{\CB}_{\M}(X)$, respectively.  We denote the operator space projective tensor product by $\wh\otimes$, and recall that $(X\wh\otimes Y)^* = \CB(X,Y^*)$ completely isometrically.  The canonical pairing between an operator space $X$ and its dual $X^*$ is denoted by $\la \omega,x\ra_{X^*,X}$
for $\omega\in X^*,x\in X$, or simply $\la\omega,x\ra$ if there is no risk of confusion. 

For a n.s.f.~weight $\theta$ on a von Neumann algebra $\M$, we denote the GNS Hilbert space by $\msf{H}_\theta$, and we use the notation $\mf{N}_{\theta}=\{x\in \M\,|\, \theta(x^*x)<+\infty\}$. We write $\Lambda_\theta\colon \mf{N}_\theta\rightarrow \msf{H}_\theta$ for the GNS map. Typically we then represent $\M$ on $\msf{H}_\theta$ and  
identify $\M\subseteq\B(\msf{H}_\theta)$.

\section{Preliminaries} \label{section preliminaries}

Throughout the paper we will work in the setting of locally compact quantum groups introduced by Kustermans and Vaes \cite{KustermansVaes}. In this section we recall some 
fundamental constructions and results of the theory, more information can be found in \cite{KustermansUniversal, KVVN, LCQGDaele}. For background on operator algebras and 
operator spaces we refer to \cite{BrownOzawa, EffrosRuan, TakesakiII}.

By definition, a \emph{locally compact quantum group} $\GG$ is given by a von Neumann algebra $\Linf$ together with a normal unital $\star$-homomorphism $\Delta_{\GG}\colon \Linf\rightarrow \Linf \bar{\otimes }\Linf$ called \emph{comultiplication}, satisfying $(\Delta_{\GG}\otimes \id) \Delta_{\GG}=(\id \otimes \Delta_{\GG}) \Delta_{\GG}$, 
and \emph{left} resp. \emph{right} \emph{Haar integrals} $\vp$ and $\psi$. These are normal, semifinite, faithful (n.s.f.) weights 
on $\Linf$ satisfying certain invariance conditions with respect to $\Delta_{\GG}$. In general, the von Neumann algebra $\Linf$ is non-commutative and will not be an algebra of functions on a measure space. Following this notational convention, the predual of $\Linf$ is denoted by $\Lj$ and the GNS Hilbert space of $\vp$ is denoted by $\LdG$. 

Every locally compact group $G$ can be seen as a locally compact quantum group $\GG$ by taking $\Linf=\LL^{\infty}(G)$, the algebra of classes of measurable, bounded functions on $G$, 
and letting $\Delta_\GG$ be the pullback of multiplication in $G$. The weights $\vp,\psi$ are given by integration with respect to left (right) Haar measure in this case.

Out of the axioms, one can construct a number of additional objects associated to a locally compact quantum group $\GG$. 
First of all, there is the \emph{Kac-Takesaki operator} $\ww^{\GG}\in \Linf\bar\otimes \Linfd$, which is a unitary operator on $\LdG\otimes \LdG$ defined via
\begin{equation}\label{eq23}
((\omega\otimes\id)\ww^{\GG *})\Lambda_\vp(x)=\Lambda_{\vp}((\omega\otimes \id)\Delta_{\GG}(x))\qquad(\omega\in \Lj, x\in\mf{N}_\vp).
\end{equation}
It implements the comultiplication via $\Delta_{\GG}(x)=\ww^{\GG *}(\I\otimes x)\ww^{\GG}$ for $x\in\Linf$. 
Tomita-Takesaki theory yields two groups of \emph{modular automorphisms} $(\sigma^\vp_t)_{t\in \RR}, (\sigma^\psi_t)_{t\in \RR}$ 
and \emph{modular conjugations} $J_{\vp},J_{\psi}$ associated with the weights $\vp,\psi$, respectively \cite{TakesakiII}. The left and right Haar integrals are linked by the \emph{modular element} $\delta_{\GG}$, which is a strictly positive, self-adjoint operator affiliated with $\Linf$. 

In the theory of quantum groups, the role of the inverse operation is played by two maps: the \emph{antipode} $S_{\GG}$ and the \emph{unitary antipode} $R_{\GG}$. 
The antipode is in general an unbounded (densely defined) map on $\Linf$ such that
\[
(\id\otimes \omega)\ww^{\GG}\in \oon{Dom}(S_{\GG})\;\textnormal{ and }S_{\GG}( (\id\otimes\omega)\ww^{\GG})=(\id\otimes \omega)\ww^{\GG *}\qquad(\omega\in \Ljd).
\]
The unitary antipode, on the other hand, is a bounded, normal, $\star$-preserving, antimultiplicative map on $\Linf$ satisfying $\Delta_{\GG}R_{\GG}=\chi (R_{\GG}\otimes R_{\GG}) \Delta_{\GG}$. These maps are linked via $S_{\GG}=R_{\GG} \tau^{\GG}_{-i/2}=\tau^{\GG}_{-i/2} R_{\GG}$, where $(\tau^{\GG}_t)_{t\in \RR}$ is the group of \emph{scaling automorphisms} of $\Linf$. There is a self-adjoint, strictly positive (in general unbounded) operator $P_{\GG}$ acting on $\LL^2(\GG)$ which implements scaling automorphisms via $\tau^{\GG}_t(x)=P^{it}_{\GG} x P^{-it}_{\GG}\,(x\in \LL^{\infty}(\GG),t\in\RR)$. Furthermore, $P_{\GG}$ is self-dual in the sense that $P_{\whG}=P_{\GG}$. The left and right Haar integrals are unique up to a scalar, and we shall fix normalisations such that $\vp=\psi\circ R_{\GG}$. The scaling constant of $\GG$ will be denoted by $\nu_{\GG}>0$. It is characterised via for example $\vp\circ \tau^{\GG}_t=\nu_{\GG}^{-t}\vp$ for $t\in \RR$.

With any locally compact quantum group $\GG$ one can associate the \emph{dual} locally compact quantum group $\whG$ in such a way that the correspondence between $\GG$ and $\whG$ extends 
Pontryagin duality. Furthermore, the Hilbert spaces $\LdG, \LL^2(\whG)$ can be identified in a canonical way, 
and the Kac-Takesaki operators of $\GG$ and $\whG$ are linked via $\ww^{\whG}=\chi (\ww^{\GG *})$. 
If there is no risk of confusion we will simply write $\Delta$ for $\Delta_{\GG}$, $\wh\Delta$ for $\Delta_{\whG}$, and similarly $R,S,\wh{R},\wh{S}$ for the (unitary) antipode 
of $\GG$ or $\whG$. 
Using the canonical identification of $\LdG$ and $\LL^2(\whG)$ one obtains a number of useful formulae. Let us mention in particular that the right regular 
representation $\vv^{\GG}\in \Linfd'\bar\otimes \Linf$ is given by $ \vv^{\GG}=(J_{\hvp}\otimes J_{\hvp}) \chi(\ww^{\GG})^* (J_{\hvp}\otimes J_{\hvp}) $.

We will also work with the weak$^*$-dense \cst-subalgebra $\mrm{C}_0(\GG)\subseteq\Linf$. It is defined as the norm-closure of the space $\{(\id\otimes\omega)\ww^{\GG}\,|\,\omega\in \Ljd\}$. After restriction, the comultiplication becomes a non-degenerate $\star$-homomorphism $\mrm{C}_{0}(\GG)\rightarrow \M(\mrm{C}_0(\GG)\otimes \mrm{C}_0(\GG))$. 
Similarly one defines $\mrm{C}_0(\whG)$, and then $\ww^{\GG}\in \M(\mrm{C}_0(\GG)\otimes \mrm{C}_0(\whG))$. 
Using the comultiplication of $\Linf$, we define a Banach algebra structure on $\Lj$ via $\omega\star\nu=(\omega\otimes\nu)\Delta_{\GG}$ for $\omega,\nu\in \Lj$.  As $\Linf$ is the dual of $\Lj$, we have a canonical $\Lj$-bimodule structure on $\Linf$, which is given by $\omega\star x=(\id\otimes\omega)\Delta_{\GG}(x)$ and $x\star\omega=(\omega\otimes\id)\Delta_{\GG}(x)$.  Treating $\Lj$ as the predual of the von Neumann algebra $\Linf$ gives, as usual, an $\Linf$-bimodule structure on $\Lj$ defined via $x\omega=\omega(\cdot \,x),\omega x=\omega(x\,\cdot)$ for $x\in\Linf,\omega\in\Lj$.

Let us introduce the map $\lambda_{\GG}\colon \Lj\rightarrow \mrm{C}_0(\whG)$ by $\lambda_{\GG}(\omega)=(\omega\otimes\id)\ww^{\GG}$, and similarly for $\whG$. Using this we define the \emph{Fourier algebra of $\GG$} as $\A(\GG)=\lambda_{\whG}(\Ljd)$. One can check that $\lambda_{\whG}$ is multiplicative, so that $\A(\GG)$ is a dense subalgebra of $\mrm{C}_0(\whG)$. 
As $\lambda_{\whG}$ is also injective, we can define an operator space structure on $\A(\GG)$ by imposing the condition that $\lambda_{\whG}\colon \Ljd\rightarrow \A(\GG)$ is 
completely isometric.

In the text, we will use certain subspaces of $\Lj$, which consist of functionals having nice additional properties. Firstly, let us introduce
\begin{equation}\begin{split}\label{eq18}
\LL^1_{\sharp}(\GG)&=\{\omega\in \Lj\,|\,
\exists_{ \theta\in \Lj}\,\lambda_{\GG}(\omega)^*=\lambda_{\GG}(\theta)\}\\
&=
\{\omega\in \Lj\,|\, \exists_{\theta\in \Lj}\,
\forall_{x\in \oon{Dom}(S_{\GG})}\, \ov{\omega}(S_{\GG}(x))=\theta(x)\}.
\end{split}\end{equation}
For a given $\omega\in \LL^1_{\sharp}(\GG)$, the functional $\theta$ is characterised uniquely by any of the properties in \eqref{eq18}, hence we can write $\theta=\omega^\sharp$.  The mapping $\omega\mapsto \omega^\sharp$, and the restriction of the multiplication from $\Lj$, turn $\LL^1_{\sharp}(\GG)$ into a $\star$-algebra. The second subspace we will use is
\begin{equation}
\mscr{J}=\{\omega\in \Lj\,|\,
\exists_{M>0}\,\forall_{x\in \mf{N}_{\vp}}\,
|\omega(x^*)|\le M \|\Lambda_\vp(x)\|\}. \label{eq:scriptI}
\end{equation}
This subspace appears in the construction of the left Haar integral $\hvp$ for $\whG$. Indeed, for $\omega\in \mscr{J}$ we have
\[
\lambda_{\GG}(\omega)\in \mf{N}_{\hvp}\quad\textnormal{ and }\quad
\forall_{x\in \mf{N}_{\vp}}\,
\ismaa{\Lvp(x)}{\Lambda_{\hvp}(\lambda_{\GG}(\omega))}=\omega(x^*).
\]
In a couple of places we will need the following result, which says that there are many functionals with desirable properties.

\begin{lemma}\label{lemma22}
The subspace
\[\begin{split}
\mscr{J}_{0}=\{\omega\in \mscr{J}\cap \LL^1_{\sharp}(\GG)\,|\,& \ov{\omega},\ov{\omega}^{\sharp} \in \mscr{J}\cap\LL^1_{\sharp}(\GG)
\textnormal{ and }f\colon \RR\ni t\mapsto (\omega \delta_{\GG}^{it})\circ \tau^{\GG}_{t}\in \Lj\\
&\textnormal{extends to an entire function such that } \forall_{z\in \CC}\, f(z)\in \mscr{J}\cap \LL^1_{\sharp}(\GG)\}
\end{split}\]
is dense in $\Lj$, and $\lambda_{\GG}(\mscr{J}_0)$ forms a $\ssots\times \|\cdot\|$ core for $\Lambda_{\hvp}$.
\end{lemma}

\begin{proof}
Our approach is standard, compare for example \cite[Lemma 14.5]{KrajczokDissertationes} for a similar result. Therefore we only give a sketch of the argument.

According to \cite[Proposition 2.6]{KVVN}, the space $\mscr{J}^{\sharp}=\{\omega\in \mscr{J}\cap\LL^1_{\sharp}(\GG)\,|\, \omega^{\sharp}\in \mscr{J}\}$ is dense in $\Lj$ and $\lambda_{\GG}(\mscr{J}^{\sharp})$ is a $\ssots\times\|\cdot\|$ core of $\Lambda_{\hvp}$. Let us introduce three mollifier operations: for $n\in\NN$ let
\[\begin{split}
M^\vp_n&\colon \Lj\ni \omega\mapsto \sqrt{\tfrac{n}{\pi}}
\int_{\RR} e^{-nt^2} \omega\circ \sigma^\vp_t\md t\in \Lj,\\
M^\tau_n&\colon \Lj\ni \omega\mapsto \sqrt{\tfrac{n}{\pi}}
\int_{\RR} e^{-ns^2} \omega\circ \tau^{\GG}_s \md s\in \Lj,
\\
M^\delta_n&\colon \Lj\ni \omega\mapsto \sqrt{\tfrac{n}{\pi}}
\int_{\RR} e^{-np^2} \omega\delta_{\GG}^{ip} \md p\in \Lj.
\end{split}\]
Next, let $\omega_n=M^\tau_n\circ M^\vp_n\circ M^\delta_n(\omega)$ and set $\mscr{J}_{00}=\lin \{\omega_n \,|\, n\in\NN,\omega\in \mscr{J}^{\sharp}\}$. 
It suffices to show that $\mscr{J}_{00}$ is dense in $\Lj$, that $\mscr{J}_{00}\subseteq \mscr{J}_0$, and that $\lambda_{\GG}(\mscr{J}_{00})$ forms a $\ssots\times \|\cdot\|$ core for $\Lambda_{\hvp}$.

Choose $n\in\NN,\omega\in \mscr{J}^{\sharp}, x\in \mf{N}_\vp,y\in \Dom(S_{\GG})$. It is elementary to check that $\omega_n\in \mscr{J}\cap \LL^1_{\sharp}(\GG)$ and $\RR\ni t\mapsto (\omega_n \delta_{\GG}^{it})\circ \tau_t^{\GG}\in \Lj$ extends to an entire function with the desired property. Since
\[\begin{split}
&\quad\;
|\ov{\omega_n}(x^*)|=
|\omega_n(x)|=
(\tfrac{n}{\pi})^{3/2}\bigl|\int_{\RR^3}
e^{-n( t^2+s^2+p^2)}
\omega ( \delta_{\GG}^{ip} \sigma^\vp_t\circ\tau^{\GG}_s(x))
\md t\md s \md p
 \bigr|\\
 &=
(\tfrac{n}{\pi})^{3/2}\bigl|\omega\bigl(
\int_{\RR^3}
e^{-n( t^2+s^2+p^2)}
 \delta_{\GG}^{ip} \sigma^\vp_t\circ\tau^{\GG}_s(x)
\md t\md s \md p\bigr)
 \bigr| \\
 &\le 
 (\tfrac{n}{\pi})^{3/2}
 \|\Lambda_{\hvp}(\lambda_{\GG}(\omega))\|
 \bigl\| \Lambda_{\vp}\bigl(
\int_{\RR^3}
e^{-n( t^2+s^2+p^2)}
\sigma^\vp_t\circ\tau^{\GG}_s(x^*) \delta_{\GG}^{-ip} 
\md t\md s \md p\bigr)
\bigr\| \\
&=
 (\tfrac{n}{\pi})^{3/2}
 \|\Lambda_{\hvp}(\lambda_{\GG}(\omega))\|
 \bigl\|J_{\vp} \nabla_\vp^{1/2} \Lambda_{\vp}\bigl(
\int_{\RR^3}
e^{-n( t^2+s^2+p^2)}
 \delta_{\GG}^{ip} 
\sigma^\vp_t\circ\tau^{\GG}_s(x)
\md t\md s \md p\bigr)
\bigr\| \\
&=
 (\tfrac{n}{\pi})^{3/2}
 \|\Lambda_{\hvp}(\lambda_{\GG}(\omega))\|
 \bigl\| \Lambda_{\vp}\bigl(
\int_{\RR^3}
e^{-n( (t+i/2)^2+s^2+p^2)}
\nu_{\GG}^{p/2}
 \delta_{\GG}^{ip} 
\sigma^\vp_t\circ\tau^{\GG}_s(x)
\md t\md s \md p\bigr)
\bigr\| \\
&\le
 (\tfrac{n}{\pi})^{3/2}
 \|\Lambda_{\hvp}(\lambda_{\GG}(\omega))\|
\bigl(\int_{\RR^3}
|e^{-n( (t+i/2)^2+s^2+p^2)}|
\nu_{\GG}^{p/2-s/2}
\md t\md s \md p\bigr)
\|\Lvp(x)\|,
\end{split}\]
where $\nu_{\GG}$ is the scaling constant of $\GG$, we have $\ov{\omega_n}\in\mscr{J}$. Here we used $\sigma^\vp_t(\delta_{\GG}^{ip})=\nu_{\GG}^{itp} \delta^{ip}$ and $\vp\circ \tau_s=\nu_{\GG}^{-s}\vp$. It is automatic that $\ov{\omega_n}\in \LL^1_{\sharp}(\GG)$. Indeed,
\[\begin{split}
\ov{ \ov{\omega_n}}(S_{\GG}(y))=
\omega_n(S_{\GG}(y))&=
\sqrt{\tfrac{n}{\pi}} \int_{\RR}e^{-ns^2}  M^\vp_n\circ M^\delta_n (\omega)(\tau^{\GG}_s\circ R_{\GG}\circ \tau^{\GG}_{-i/2}(y)) \md s\\
&=
\bigl(
\sqrt{\tfrac{n}{\pi}} \int_{\RR}e^{-n(s+i/2)^2}  M^\vp_n\circ M^\delta_n (\omega)\circ\tau^{\GG}_s\circ R_{\GG} \md s
\bigr)(y).
\end{split}\]
The above calculation shows also that $\ov{\omega_n}^{\sharp}=
\sqrt{\tfrac{n}{\pi}} \int_{\RR}e^{-n(s+i/2)^2}  M^\vp_n\circ M^\delta_n (\omega)\circ\tau^{\GG}_s\circ R_{\GG} \md s$. Moreover we have $\ov{\omega_n}^{\sharp}\in \mscr{J}$, which is a consequence of the following calculation:
\[\begin{split}
&\quad\;
|\ov{\omega_n}^{\sharp}(x^*)|=
(\tfrac{n}{\pi})^{3/2} \bigl| \int_{\RR^3} 
e^{-n( t^2+(s+i/2)^2+p^2)} \omega\bigl(
\delta_{\GG}^{ip}\sigma^\vp_t\circ \tau_{s}^{\GG}\circ R_{\GG}(x^*)
\bigr)\md t \md s\md p\bigr|\\
&\le 
(\tfrac{n}{\pi})^{3/2}
\|\Lhvp(\lambda_{\GG}(\omega))\|
\bigl\|
\Lambda_{\vp}\bigl(
\int_{\RR^3} e^{-n (t^2 +(s-i/2)^2+p^2)}
\sigma^\vp_t\circ \tau^{\GG}_s\circ R_{\GG}(x) \delta_{\GG}^{-ip} \md t \md s \md p
\bigr)
\bigr\|\\
&=
(\tfrac{n}{\pi})^{3/2}
\|\Lhvp(\lambda_{\GG}(\omega))\|
\bigl\|
\Lambda_{\psi}\bigl(
\int_{\RR^3} e^{-n (t^2 +(s-i/2)^2+(p+i/2)^2)}
\sigma^\vp_t\circ \tau^{\GG}_s\circ R_{\GG}(x) \delta_{\GG}^{ip} \md t \md s \md p
\bigr)
\bigr\|\\
&=
(\tfrac{n}{\pi})^{3/2}
\|\Lhvp(\lambda_{\GG}(\omega))\|
\bigl\|
\Lambda_{\psi}\bigl(
\int_{\RR^3} e^{-n (t^2 +(s-i/2)^2+(p+i/2)^2)}
\delta_{\GG}^{-it} \sigma^\psi_t\circ \tau^{\GG}_s\circ R_{\GG}(x) \delta_{\GG}^{i(t+p)} \md t \md s \md p
\bigr)
\bigr\|\\
&=\!
(\tfrac{n}{\pi})^{3/2}
\|\Lhvp(\lambda_{\GG}(\omega))\|
\bigl\|
J_{\psi} \nabla_{\psi}^{1/2}
\Lambda_{\psi}\bigl(\!
\int_{\RR^3}\!\! e^{-n (t^2 +(s+i/2)^2+(p-i/2)^2)}
\delta_{\GG}^{-i(t+p)} \sigma^\psi_t\circ \tau^{\GG}_s\circ R_{\GG}(x^*) \delta_{\GG}^{it} \md t \md s \md p
\bigr)
\bigr\|\\
&= \!
(\tfrac{n}{\pi})^{3/2}
\|\Lhvp(\lambda_{\GG}(\omega))\|
\bigl\|
\Lambda_{\psi}\!\bigl(\!
\int_{\RR^3}\!\! e^{-n ((t+i/2)^2 \!+(s+i/2)^2\!+(p-i/2)^2)}\!
\nu_{\GG}^{-p/2}\!
\delta_{\GG}^{-i(t+p)} \sigma^\psi_t\!\circ\! \tau^{\GG}_s\!\circ\! R_{\GG}(x^*) \delta_{\GG}^{it} \md t \md s \md p
\bigr)
\bigr\|\\
&\le  
(\tfrac{n}{\pi})^{3/2}
\|\Lhvp(\lambda_{\GG}(\omega))\|
\int_{\RR^3} |e^{-n ((t+i/2)^2 +(s+i/2)^2+(p-i/2)^2)}|
\nu_{\GG}^{s/2-t/2-p/2}
\|\Lambda_{\psi}(R_{\GG}(x^*))\| \md t \md s \md p\\
&=
(\tfrac{n}{\pi})^{3/2}
\|\Lhvp(\lambda_{\GG}(\omega))\|
\bigl(\int_{\RR^3} |e^{-n ((t+i/2)^2 +(s+i/2)^2+(p-i/2)^2)}|
\nu_{\GG}^{s/2-t/2-p/2}
 \md t \md s \md p\bigr)
 \|\Lambda_{\vp}(x)\|,
\end{split}\]
where we used $\Lvp(z)=\Lambda_{\psi}(z \delta_{\GG}^{1/2})$ for sufficiently good operators $z$, $\sigma^\psi_t=\oon{Ad}(\delta^{it}_{\GG})\circ \sigma^\vp_t$, $\psi\circ \tau^{\GG}_s=\nu_{\GG}^{-s}\psi$ and $\sigma^\psi_t(\delta_{\GG}^{ip})=\nu_{\GG}^{itp}\delta_{\GG}^{ip}$. We are left to argue that $\mscr{J}_{00}$ is dense and that $\lambda_{\GG}(\mscr{J}_{00})$ is a $\sots\times \|\cdot\|$ core for $\Lhvp$. It follows from 
\[
\omega_n\xrightarrow[n\to\infty]{} \omega,\quad
\Lambda_{\hvp}(\lambda_{\GG}(\omega_n))\xrightarrow[n\to\infty]{} \Lhvp(\lambda_{\GG}(\omega))
\]
for $\omega\in \mscr{J}^{\sharp}$. The first claim is standard (see e.g.~\cite[Proposition 2.25]{KustermansOneParam} where a similar method is used), the second one can be seen as follows: first, observe that for $t,s,p\in \RR$ we have using $(\sigma^\vp_t\otimes \id)\ww^{\GG}=(\id\otimes \tau^{\whG}_{-t})(\ww^{\GG}) (\I\otimes \delta_{\whG}^{it})$ (see equation \eqref{eq20}) and $(\tau^{\GG}_s\otimes \tau^{\whG}_{s})\ww^{\GG}=\ww^{\GG}$ that $(\tau^{\GG}_{-t}\otimes \sigma^{\hvp}_{-t})\ww^{\GG}=(\delta_{\GG}^{it}\otimes \I)\ww^{\GG}$ and
\[\begin{split}
&\quad\;
\lambda_{\GG}( (\omega\delta_{\GG}^{ip})\circ \tau^{\GG}_s\circ \sigma^\vp_t)=
((\omega\delta_{\GG}^{ip})\circ \tau^{\GG}_s\otimes \id)\bigl( 
(\id\otimes \tau^{\whG}_{-t})(\ww^{\GG}) (\I\otimes \delta_{\whG}^{it})
\bigr)\\
&=
(\omega\otimes\tau^{\whG}_{-t-s})\bigl( 
(\delta^{ip}_{\GG}\otimes \I)\ww^{\GG}\bigr) \delta_{\whG}^{it}=
(\omega\circ \tau^{\GG}_{-p} \otimes\tau^{\whG}_{-t-s}\circ \sigma^{\hvp}_{-p})(
\ww^{\GG}\bigr) \delta_{\whG}^{it}\\
&=
(\omega \otimes\tau^{\whG}_{p-t-s}\circ \sigma^{\hvp}_{-p})(
\ww^{\GG}\bigr) \delta_{\whG}^{it}
=
\tau^{\whG}_{p-t-s}\circ \sigma^{\hvp}_{-p}(\lambda_{\GG}(\omega)) \delta_{\whG}^{it}.
\end{split}\]
Hence
\[\begin{split}
\Lambda_{\hvp}(\lambda_{\GG}(\omega_n))&=
(\tfrac{n}{\pi})^{3/2}\int_{\RR^3} e^{-n(t^2+s^2+p^2)}
\Lambda_{\hvp}(\lambda_{\GG}( (\omega \delta_{\GG}^{ip})\circ \tau^{\GG}_s\circ \sigma^\vp_t))\md t \md s \md p\\
&=
(\tfrac{n}{\pi})^{3/2}\int_{\RR^3} e^{-n(t^2+s^2+p^2)}
\Lambda_{\hvp}\bigl(
\tau^{\whG}_{p-t-s}\circ \sigma^{\hvp}_{-p}(\lambda_{\GG}(\omega)) \delta_{\whG}^{it}\bigr)
\md t \md s \md p\\
&=
(\tfrac{n}{\pi})^{3/2}\int_{\RR^3} e^{-n(t^2+s^2+p^2)}
\nu_{\whG}^{s/2-p/2}J_{\hvp} \delta_{\whG}^{-it} J_{\hvp}
P^{i(p-t-s)} \nabla_{\hvp}^{-ip}\Lambda_{\hvp}(\lambda_{\GG}(\omega)) 
\md t \md s \md p,
\end{split}\]
where $P$ is the operator implementing the scaling group via $P^{it}\Lambda_{\vp}(x)=\nu_{\GG}^{t/2}\Lambda_{\vp}(\tau^{\GG}_t(x))$. Convergence $\Lambda_{\hvp}(\lambda_{\GG}(\omega_n))\xrightarrow[n\to\infty]{} \Lambda_{\hvp}(\lambda_{\GG}(\omega))$ follows now using a standard argument.
\end{proof}

\section{Completely bounded multipliers} \label{section multiplier}

Unless stated otherwise, in this section $\GG$ is an arbitrary locally compact quantum group. 

\subsection{Definitions and fundamental facts}

We start by discussing the notions of (left/right) centralisers and multipliers. In the main part of the text we will focus on the left version of these 
objects -- this is simply a matter of choice, compare Proposition~\ref{prop14}.

Following \cite{CBMultipliers}, we say that a linear map $T\colon \Ljd\rightarrow\Ljd$ is a \emph{left (respectively, right) centraliser} if
\[
T(\omega\star \omega')=T(\omega)\star \omega'\quad\bigl(
\textnormal{respectively}\; T(\omega\star \omega')=\omega\star T(\omega')\bigr)
\qquad (\omega,\omega'\in \Ljd).
\]
We denote by $C^l_{cb}(\Ljd)$ the space of completely bounded left centralisers. Together with the completely bounded norm and composition as product, $C^l_{cb}(\Ljd)$ becomes a Banach algebra. Similarly, $C^r_{cb}(\Ljd)$ stands for the space of completely bounded right centralisers, where now it is natural to use the opposite composition as product. We equip these spaces with an operator space structure by requiring that the embeddings $C^l_{cb}(\Ljd),C^r_{cb}(\Ljd)\hookrightarrow \CB(\Ljd)$ are completely isometric; both then become completely contractive Banach algebras.

An operator $b\in \Linf$ is said to be a \emph{completely bounded left multiplier} if $b \A(\GG)\subseteq \A(\GG)$ and the associated map
\[
\Theta^l(b)_*\colon \Ljd\rightarrow\Ljd
\quad\text{satisfying}\quad
b\wh\lambda(\omega)=\wh{\lambda}(\Theta^l(b)_*(\omega))\qquad(\omega\in\Ljd)
\]
is completely bounded.  As $\wh\lambda$ is injective, this definition makes sense.  We follow here the notation of \cite{Crann}; sometimes the notation $m^l_{b} = \Theta^l(b)_*$ is used instead.  As $\wh\lambda$ is multiplicative, for any completely bounded left multiplier $b$ we have that $\Theta^l(b)_*\in C^l_{cb}(\Ljd)$.  We write $\Theta^l(b)=(\Theta^l(b)_*)^*$, and denote the space of CB left multipliers by $\M^l_{cb}(\A(\GG))$.  Any Fourier algebra element $\wh{\lambda}(\omega)\in\A(\GG)$ is a CB left multiplier with $\Theta^l(\wh{\lambda}(\omega))_*\in \CB(\Ljd)$ being the left multiplication by $\omega$ and $\Theta^l(\wh{\lambda}(\omega))=(\omega\otimes\id)\wh\Delta$. Moreover, it holds that $\M^l_{cb}(\A(\GG))\subseteq \M(\mrm{C}_0(\GG))$, see \cite[Theorem 4.2]{Daws_Hilbert}.

Conversely, if $T\in C^l_{cb}(\Ljd)$ is a left centraliser, then its Banach space dual $T^*$ is a normal CB map on $\Linfd$ which is a left $\Ljd$-module homomorphism, i.e. $T^*\in \prescript{}{\Ljd}{\CB^{\sigma}(\Linfd)}$. Then, by \cite[Corollary 4.4]{JungeNeufangRuan}, there exists a unique CB left multiplier $b\in \M^l_{cb}(\A(\GG))$ satisfying $\Theta^l(b) = T^*$, that is, $\Theta^l(b)_* = T$.  These constructions are mutually inverse, and so the map $\Theta^l(\cdot)_*:\M^l_{cb}(\A(\GG)) \rightarrow C^l_{cb}(\Ljd)$ is bijective. We define the operator space structure on $\M^l_{cb}(\A(\GG))$ so that these spaces become completely isometric.

The above notions have right counterparts. Recalling that $\vv^{\whG}\in \Linf'\bar{\otimes}\Linfd$ is the right Kac-Takesaki operator, let us introduce the map $\wh{\rho}\colon \Ljd\ni \omega\mapsto (\id\otimes\omega)\vv^{\whG}\in \Linf'$. Its image $\wh{\rho}(\Ljd)$ should be thought of as a right analogue of the Fourier algebra $\A(\GG)=\wh{\lambda}(\Ljd)$. An operator $b'\in \Linf'$ is called a \emph{completely bounded right multiplier} if $\wh{\rho}(\Ljd)b'\in \wh{\rho}(\Ljd)$ and the associated map $\Theta^r(b')_*\colon \Ljd\rightarrow\Ljd$ is CB. Similarly as in the left version, we write $\Theta^r(b')\in\CB^\sigma_{\Ljd}(\Linfd)$ for $(\Theta^r(b')_*)^*$ and $\M^r_{cb}(\wh{\rho}(\Ljd))$ for the space of CB right multipliers. Any CB right centraliser $T\in C^r_{cb}(\Ljd)$ is associated to a unique CB right multiplier $b'\in \M^r_{cb}(\wh{\rho}(\Ljd))$ via $T=\Theta^r(b')_*$ and this assignment is bijective.  We similarly define an operator space structure on $\M^r_{cb}(\wh{\rho}(\Ljd))$ to make it completely isometric with $C^r_{cb}(\Ljd)$.

We will write e.g.~$\|b\|_{cb}=\|\Theta^l(b)\|_{cb}$ for $b\in \M^l_{cb}(\A(\GG))$.  Observe that $b\wh\lambda(\omega)=\wh{\lambda}(\Theta^l(b)_*(\omega))$ for each $\omega\in\Ljd$ if and only if $(\I\otimes b)\ww^{\whG}=(\Theta^l(b)\otimes \id)(\ww^{\whG})$, and from this, it follows that $\|b\|\le \|b\|_{cb}$. Similarly we have $\|b'\|\le \|b'\|_{cb}$ for $b'\in \M^r_{cb}(\wh{\rho}(\Ljd))$.\\

As a consequence of the above discussion, we have a commutative diagram
\begin{equation}\label{eq16} \xymatrix{
  C^l_{cb}(\Ljd)  \ar[r]^{\cong} & \M^l_{cb}(\A(\GG)) \ar[rd] \\
  && \Linf \\
  \Ljd \ar[uu]\ar[r]^{\cong} & \A(\GG)\ar[uu] \ar[ru]
  }
\end{equation}
The two diagonal maps to $\Linf$ are the canonical inclusions, as is the map $\A(\GG)\rightarrow \M^l_{cb}(\A(\GG))$, while the vertical map $\Ljd\rightarrow C^l_{cb}(\Ljd)$ is given by left multiplication. A simple calculation shows that this diagram indeed commutes. We obtain an immediate corollary: the map $\Ljd \rightarrow C^l_{cb}(\Ljd)$ is injective, equivalently, if $\omega\in\Ljd$ with $\omega\star\omega'=0$ for all $\omega'\in\Ljd$, then $\omega=0$.

There is a canonical way of moving between left and right CB multipliers using the extension of the unitary antipode $\wh{R}$ of $\whG$. Recall that it is implemented 
via $\wh{R}=J_{\vp} (\cdot)^* J_\vp$, and let us denote its canonical extension to a bounded linear map on $\B(\LdG)$ by ${\wh{R}}^{\sim}=J_\vp (\cdot)^* J_\vp$. The following result will be used in Proposition~\ref{prop14} to show that it does not matter if we use left CB multipliers, or right CB multipliers, when we introduce the approximation property (AP), see Definition~\ref{def1} below.

\begin{lemma}\label{lemma19}
For $\omega\in \Ljd$ we have $\wh{\rho}(\omega)=\wh{R}^{\sim}(\wh{\lambda}(\omega\circ \wh{R}))$. Furthermore, $\wh{R}^{\sim}( \M^l_{cb}(\A(\GG)))=\M^r_{cb}(\wh{\rho}(\Ljd))$ and for $a\in \M^l_{cb}(\A(\GG))$ we have $\Theta^r(\wh{R}^{\sim}(a))=\wh{R}\circ \Theta^l(a)\circ \wh{R} $.
\end{lemma}
\begin{proof}
Recall that $\vv^{\whG}=(J_\vp\otimes J_\vp) \chi(\ww^{\whG})^* (J_\vp\otimes J_\vp)$, see \cite[Proposition~2.15]{KVVN}.  It follows that
\begin{equation}\label{eq12}
\wh{\rho}(\omega)=
(\id\otimes\omega)\vv^{\whG}=
\wh{R}^{\sim}\bigl((\id \otimes \omega\circ \wh{R})(
\chi(\ww^{\whG}))\bigr)=
\wh{R}^{\sim}\bigl( (\omega\circ \wh{R}\otimes\id)(\ww^{\whG})\bigr)=
\wh{R}^{\sim} ( \wh{\lambda}(\omega\circ \wh{R})).
\end{equation}
Next, take $a\in \M^l_{cb}(\A(\GG))$ and $\omega\in \Ljd$. We have $\wh{R}^{\sim}(a) =J_\vp a^* J_\vp \in \Linf'$, so by \eqref{eq12},
\begin{align*}
\wh{\rho}(\omega) \wh{R}^{\sim}(a)
&= \wh{R}^{\sim} ( \wh{\lambda}(\omega\circ \wh{R})) \wh{R}^{\sim}(a)=
\wh{R}^{\sim}( a \wh{\lambda}(\omega\circ \wh{R})) \\
&=
\wh{R}^{\sim}\bigl( \wh{\lambda}(\Theta^l(a)_*(\omega\circ \wh{R}))\bigr)=
\wh{\rho}\bigl(\Theta^l(a)_*(\omega\circ \wh{R})\circ \wh{R}\bigr).
\end{align*}
Hence $\wh{R}^{\sim}(a)\in \M^r_{cb}(a)$ with $\Theta^r(\wh{R}^{\sim}(a))=\wh{R}\circ \Theta^l(a)\circ\wh{R}$; this map is indeed CB, compare Lemma~\ref{lemma18}. We have shown that $\wh{R}^{\sim}(\M^l_{cb}(\A(\GG)))\subseteq \M^r_{cb}(\wh{\rho}(\Ljd))$; the converse inclusion is analogous.
\end{proof}

We finish by recording a known result for which we have not found a convenient reference.

\begin{lemma}\label{lem:cent_unit_preserving}
Let $b\in\M^l_{cb}(\A(\GG))$.  There is $\beta\in\mathbb C$ with $\Theta^l(b)(\I) = \beta \I$.
\end{lemma}
\begin{proof}
It suffices to show that for $T\in \prescript{}{\Ljd}{\CB^{\sigma}(\Linfd)}$ there is $\beta$ with $T(\I)=\beta \I$.  By definition, $\Delta\circ T = (T\otimes\id)\Delta$ and so $\Delta(T(\I)) = T(\I)\otimes \I$.  By \cite[Theorem~2.1]{DawsCPMults} (and references therein) it follows that $T(\I)\in\mathbb C \I$, as required.
\end{proof}

\subsection{Predual}\label{sec:predual}

Since the inclusion $\M^l_{cb}(\A(\GG))\hookrightarrow \Linf$ is bounded (actually, contractive), we can consider the restriction of the Banach space adjoint of this map, giving a map $\alpha^l\colon \Lj\rightarrow \M^l_{cb}(\A(\GG))^*$. Let us define the space $Q^l(\A(\GG))$ as the closure of the image of $\alpha^l$, so that
\[
Q^l(\A(\GG))=\ov{\alpha^l(\Lj)}\subseteq \M^l_{cb}(\A(\GG))^*.
\]
According to \cite[Theorem 3.4]{CBMultipliers}, the space $Q^l(\A(\GG))$ is a predual of $\M^l_{cb}(\A(\GG))$, i.e.~we have
\[
Q^l(\A(\GG))^*\cong \M^l_{cb}(\A(\GG))
\]
completely isometrically. Whenever we speak about the weak$^*$-topology on $\M^l_{cb}(\A(\GG))$ we will have in mind this particular choice of predual -- to the best of our knowledge, uniqueness of predual of $\M^l_{cb}(\A(\GG))$ is unknown.

Similarly, we can restrict functionals in $\LL^1(\GG')$ to $\M^r_{cb}(\wh{\rho}(\Ljd))$ via map $\alpha^r$, and after taking the closure obtain the predual $Q^r(\wh{\rho}(\Ljd))\subseteq \M^r_{cb}(\wh{\rho}(\Ljd))^*$. From now on, we will restrict our discussion to the ``left'' setting.

\begin{proposition}\label{prop18}
$\M^l_{cb}(\A(\GG))$ is a dual Banach algebra, that is, the multiplication of $\M^l_{cb}(\A(\GG))$ is separately weak$^*$-continuous.
\end{proposition}

\begin{proof}
We turn $\M^l_{cb}(\A(\GG))^*$ into a $\M^l_{cb}(\A(\GG))$-bimodule in the usual way.  Let $a,b\in \M^l_{cb}(\A(\GG)) \subseteq \Linf$ and $f\in \Lj$.  Writing $\la\cdot,\cdot\ra$ for the pairing between $\M^l_{cb}(\A(\GG)$ and $Q^l(\A(\GG))$, or between $\Linf$ and $\Lj$, we have
\begin{align*}
\langle ab, \alpha^l(f) \rangle
= \langle ab, f \rangle
&= \langle a, bf \rangle = \langle a, \alpha^l(bf) \rangle  \\
&= \langle b, fa \rangle = \langle b, \alpha^l (fa) \rangle.
\end{align*}
This calculation shows that $b\cdot\alpha^l(f) = \alpha^l(bf)$ and $\alpha^l(f)\cdot a = \alpha^l(fa)$, so by continuity, it follows that $Q^l(\A(\GG))$ is a closed submodule of $\M^l_{cb}(\A(\GG))^*$.  It is now standard, see \cite[Proposition~1.2]{runde_dba} for example, that the product is separately weak$^*$-continuous in $\M^l_{cb}(\A(\GG))$.
\end{proof}

Our next goal is to obtain a characterisation of functionals in $Q^l(\A(\GG))$. We will do this by obtaining an alternative description of the weak$^*$-topology on $\M^l_{cb}(\A(\GG))$. In the process, we also discuss CB maps on the \cst-algebra $\mrm{C}_0(\whG)$ which are associated to left centralisers.

To start, we observe that the adjoint $T^*$ of a CB left centraliser $T\in C^l_{cb}(\Ljd)$ restricts to a CB map on $\mrm{C}_0(\whG)$. Indeed, we can write $T^*=\Theta^l(a)$ for some $a\in \M^l_{cb}(\A(\GG))$ and then the claim follows from the equality $(\I\otimes a)\ww^{\whG}=(T^*\otimes \id)\ww^{\whG}$ and density of $\A(\whG)$ in $\mrm{C}_0(\whG)$. We seek a characterisation of which CB maps on $\mrm{C}_0(\whG)$ occur in this way as restrictions of duals to left CB centralisers, in terms of a property similar to the characterisation $C^l_{cb}(\Ljd) \cong \prescript{}{\Ljd}{\CB}^{\sigma}(\Linfd)$.

In the following statement, recall that $(\CZGD, \wh\Delta)$ is bisimplifiable, and so elements of the form $\wh\Delta(a)(\I\otimes b)$, for $a,b\in \CZGD$, form a linearly dense subset of $\CZGD \otimes \CZGD$.  Hence the left-hand-side of (\ref{eq:weak_cent_cond}) is contained in $\Linfd \otimes \CZGD \subseteq \Linfd\oxx\Linfd$, while the right-hand-side is 
in $\Linfd\oxx\Linfd$. We also recall that, by Kaplansky density, $\Ljd$ is (completely) isometrically a subspace of $\CZGD^*$.

\begin{lemma}\label{lemma20}
Let $L\in\CB(\CZGD, \Linfd)$ be such that
\begin{equation}
(L\otimes\id)( \wh\Delta(a)(\I\otimes b) )=  \wh\Delta(L(a))(\I\otimes b)
  \qquad (a,b\in \CZGD). \label{eq:weak_cent_cond}
\end{equation}
Embedding $\Ljd$ into the duals of $\Linfd$ and $\CZGD$ in the usual way, we have that $L^*$ maps $\Ljd$ to itself, and the resulting restriction $T\in\CB(\Ljd)$ is a left centraliser.  Furthermore $T^*\in\CB(\Linfd)$ restricts to $L$, so consequently $L\in \CB(\mrm{C}_0(\whG))$.
\end{lemma}

\begin{proof}
As $\Ljd$ is an essential $\CZGD$-module, by Cohen--Hewitt factorisation, given $\omega_2\in\Ljd$ there are $\omega_3\in\Ljd$ and $b\in\CZGD$ with $\omega_2 = b\omega_3$. 
Then, for $\omega_1\in\Ljd$ and $a\in\CZGD$ we have 
\[\begin{split}
&\quad\;
\langle L^*(\omega_1\star\omega_2), a \rangle_{\CZGD^*, \CZGD}= \langle \wh\Delta(L(a)), \omega_1\otimes\omega_2 \rangle_{\Linfd\bar\otimes\Linfd, \Ljd\wh\otimes\Ljd} \\
&= \langle \wh\Delta(L(a))(\I\otimes b), \omega_1\otimes\omega_3 \rangle_{\Linfd\bar\otimes\Linfd, \Ljd\wh\otimes\Ljd} \\
&= \langle (L\otimes\id)( \wh\Delta(a)(\I\otimes b) ), \omega_1\otimes\omega_3 \rangle_{\Linfd\bar\otimes\Linfd, \Ljd\wh\otimes\Ljd}  \\
&= \langle L^*(\omega_1)\otimes \omega_3, \wh\Delta(a)(\I\otimes b) \rangle_{(\CZGD\otimes \CZGD)^*, \CZGD\otimes \CZGD} \\
&= \langle L^*(\omega_1), (\id\otimes \omega_3)(\wh\Delta(a)(\I\otimes b)) \rangle_{\CZGD^*, \CZGD} \\
&= \langle L^*(\omega_1), (\id\otimes \omega_2)\wh\Delta(a) \rangle_{\CZGD^*, \CZGD} \\
&= \langle L^*(\omega_1)\star \omega_2, a \rangle_{\CZGD^*, \CZGD}.
\end{split}\]
It follows that $ L^*(\omega_1\star\omega_2) = L^*(\omega_1) \star \omega_2$ in $\CZGD^*$.  As $\Ljd$ is an ideal in $\CZGD^*$ (\cite[Proof of Proposition 8.3]{KustermansVaes}), this shows that $L^*(\omega_1\star\omega_2) \in \Ljd$, and as products have dense linear span in $\Ljd$ (\cite[Section 3]{Crann}), we conclude that $L^*$ restricts to a map on $\Ljd$, say $T\in\CB(\Ljd)$.  Then $T(\omega_1\star\omega_2) = T(\omega_1) \star \omega_2$ for all $\omega_1,\omega_2$, and so $T\in C^l_{cb}(\Ljd)$.  We finally calculate that, for $a\in\CZGD, \omega\in\Ljd$,
\begin{align*}
\langle T^*(a), \omega \rangle_{\Linfd, \Ljd}
&= \langle a, T(\omega) \rangle_{\Linfd,\Ljd}
= \langle L^*(\omega), a \rangle_{\CZGD^*, \CZGD}
= \langle L(a), \omega \rangle_{\Linfd, \Ljd},
\end{align*}
and so $T^*$ restricts to $L$, as required.
\end{proof}

We can now characterise what it means for an operator in $\CB(\CZGD)$ to be a centraliser.  Condition~(\ref{prop:cent_to_cz:two}) in the following proposition should be thought of as a $\CZGD$ variant of what it means to be a left $\Ljd$-module homomorphism.

\begin{proposition}\label{prop13}
For $L \in \CB(\CZGD)$ the following are equivalent:
\begin{enumerate}
\item\label{prop:cent_to_cz:one} there is $T\in C^l_{cb}(\Ljd)$ such that $T^*$ restricts to $L$;
\item\label{prop:cent_to_cz:two} $(L\otimes\id)( \wh\Delta(a)(\I\otimes b) ) = \wh\Delta(L(a))(\I\otimes b)$ for each $a,b\in \CZGD$;
\end{enumerate}
Furthermore, the restriction map $C^l_{cb}(\Ljd) \cong \prescript{}{\Ljd}{\CB}^{\sigma}(\Linfd) \rightarrow \CB(\CZGD)$ is a complete isometry; in particular, there is a bijection between $ C^l_{cb}(\Ljd)$ and the space of all maps $L \in \CB(\CZGD)$ satisfying (\ref{prop:cent_to_cz:two}).
\end{proposition}

\begin{proof}
If (\ref{prop:cent_to_cz:one}) holds then $\widehat{\Delta} T^* = (T^*\otimes\id)\widehat{\Delta}$ and so certainly the condition in (\ref{prop:cent_to_cz:two}) will hold for $T^*$ and hence also for $L$.  Conversely, suppose that (\ref{prop:cent_to_cz:two}) holds. Then due to Lemma~\ref{lemma20} we know that $L^*$ restricts to a map $T\in C^l_{cb}(\Ljd)$ such that $T^*$ restricts to $L$, showing (\ref{prop:cent_to_cz:one}).

The restriction map $\prescript{}{\Ljd}{\CB}^{\sigma}(\Linfd) \rightarrow \CB(\CZGD)$ is clearly a complete contraction.  With $T,L$ as above, this restriction map is given by $T^* \mapsto L$, and as $T$ is the restriction of $L^*$ and $L\mapsto L^*,T\mapsto T^*$ are completely isometric, it follows that $\prescript{}{\Ljd}{\CB}^{\sigma}(\Linfd) \rightarrow \CB(\CZGD)$ is a complete isometry.
\end{proof}

\begin{proposition}\label{prop:new_wstar_cent}
Equip the space $\CB(\CZGD, \Linfd)$ with the weak$^*$-topology arising from the canonical predual $\CZGD \wh\otimes \Ljd$. The restriction map 
$$
C^l_{cb}(\Ljd) \cong \prescript{}{\Ljd}{\CB}^{\sigma}(\Linfd) \rightarrow \CB(\CZGD, \Linfd)
$$ 
is a complete isometry which has weak$^*$-closed image.
\end{proposition}

\begin{proof}
Proposition~\ref{prop13} shows that the restriction map $\prescript{}{\Ljd}{\CB}^{\sigma}(\Linfd) \rightarrow \CB(\CZGD, \Linfd)$ is a complete isometry.  Let $(T_i)_{i\in I}$ be a net in $C^l_{cb}(\Ljd)$ such that the image of the net $(T_i^*)_{i\in I}$ in $\CB(\CZGD, \Linfd)$ converges weak$^*$ to $L\in \CB(\CZGD, \Linfd)$.  Let $a,b\in\CZGD$ and $\omega_1,\omega_2\in \Ljd$, and note that
$(\id\otimes\omega_2)( \wh\Delta(a)(\I\otimes b) ) \in \CZGD$.  We now calculate that
\begin{align*}
&\quad\;
\langle \wh\Delta(L(a))(\I\otimes b),  \omega_1\otimes\omega_2 \rangle
= \lim_{i\in I} \langle T_i^*(a), \omega_1 \star (b\omega_2) \rangle
= \lim_{i\in I} \langle a, T_i(\omega_1) \star (b\omega_2) \rangle\\
&= \lim_{i\in I} \langle \wh\Delta(a)(\I\otimes b), T_i(\omega_1) \otimes \omega_2 \rangle = \lim_{i\in I} \langle T_i^*\big( (\id\otimes\omega_2)(\wh\Delta(a)(\I\otimes b)) \big), \omega_1\rangle\\
&= \langle L\big( (\id\otimes\omega_2)(\wh\Delta(a)(\I\otimes b)) \big), \omega_1\rangle = \langle (L\otimes\id)(\wh\Delta(a)(\I\otimes b)), \omega_1\otimes\omega_2 \rangle.
\end{align*}
All the above pairings are between a von Neumann algebra and its predual. It follows that we have $\wh\Delta(L(a))(\I\otimes b) = (L\otimes\id)(\wh\Delta(a)(\I\otimes b))$ in $\Linfd\bar\otimes\Linfd$.  By Lemma~\ref{lemma20}, $L^*$ restricts to $T\in C^l_{cb}(\Ljd)$ such that $T^*$ restricts back to give $L$.  That is, $T_i^* \xrightarrow[i\in I]{} T^*$ weak$^*$ in $\CB(\CZGD, \Linfd)$, as required.
\end{proof}

We now wish to show that the resulting weak$^*$-topology on $C^l_{cb}(\Ljd)$ given by Proposition~\ref{prop:new_wstar_cent} agrees with the weak$^*$-topology on $C^l_{cb}(\Ljd) \cong \M^l_{cb}(\A(\GG))$ given by the predual $Q^l(\A(\GG))$. In the following, for a Banach space $E$, we denote by $\kappa_E\colon E\rightarrow E^{**}$ the canonical map to the bidual.  

\begin{lemma} \label{biduallemma}
Let $E,F$ be Banach spaces, and let $\alpha:E^*\rightarrow F^*$ be a bounded linear map. Let $D\subseteq F$ be a subset with dense linear span.  Then $\alpha$ is weak$^*$-weak$^*$-continuous if and only if $\alpha^*\kappa_F(D) \subseteq \kappa_E(E)$.  In this case, and when further $\alpha$ is a bijection, the resulting preadjoint $\alpha_*:F\rightarrow E$ is also an isomorphism of Banach spaces and $\alpha$ is a weak$^*$-weak$^*$-homeomorphism.
\end{lemma}

\begin{proof}
If $\alpha$ is weak$^*$-continuous, then there is a preadjoint operator $\alpha_*:F\rightarrow E$ with $(\alpha_*)^* = \alpha$, and so $\alpha^*\kappa_F(D) = (\alpha_*)^{**} \kappa_F(D) = \kappa_E\alpha_*(D) \subseteq \kappa_E(E)$, as claimed.  Conversely, if $\alpha^*\kappa_F(D) \subseteq \kappa_E(E)$ then by norm density of $\lin D$ in $F$, and norm continuity of $\alpha^*$, we have that $\alpha^*\kappa_F(F) \subseteq \kappa_E(E)$.  We could now directly apply \cite[Lemma~10.1]{Daws_mults_si}, but let us give the argument.  There is a linear map $T:F\rightarrow E$ with $\alpha^*\kappa_F(x) = \kappa_E(T(x))$ for each $x\in F$.  As $\kappa_E, \kappa_F$ are isometries, $T$ is bounded with $\|T\|\leq\|\alpha^*\|=\|\alpha\|$.  Then for $x\in F, \mu\in E^*$,
\[ \langle T^*(\mu), x \rangle = \langle \mu, T(x) \rangle
= \langle \kappa_E(T(x)), \mu \rangle
= \langle \alpha^*\kappa_F(x), \mu \rangle
= \langle \alpha(\mu), x \rangle. \]
Hence $T^* = \alpha$ and so $\alpha$ is weak$^*$-continuous, with preadjoint $T$.

When $\alpha$ is a bijection, by the Open Mapping Theorem, it is an isomorphism.  Thus also $\alpha^*$ is an isomorphism, and so as $\alpha^*\kappa_F = \kappa_E \alpha_*$ it follows that $\alpha_*$ is bounded below and so has closed image.  If $\mu\in (\alpha_*(F))^\perp$ then $0 = \langle \mu, \alpha_*(x) \rangle = \langle \alpha(\mu), x \rangle$ for all $x,\mu$ and so $\alpha(\mu) = 0$ so $\mu=0$.  Hence $\alpha_*$ is a surjection, and so an isomorphism.
\end{proof}

\begin{theorem}\label{thm5}
The weak$^*$-topology on $C^l_{cb}(\Ljd)$ given by the embedding into $\CB(\CZGD, \Linfd)$ agrees with the weak$^*$-topology on $\M_{cb}^l(\A(\GG))$ given by $Q^l(\A(\GG))$.
\end{theorem}

\begin{proof}
We use Lemma~\ref{biduallemma}.  Set $E = Q^l(\A(\GG))$.  To avoid confusion, for this proof only, we shall write $\theta:C^l_{cb}(\Ljd) \rightarrow \CB(\CZGD, \Linfd)$ for the complete isometry $T\mapsto T^*|_{\CZGD}$, given by Proposition~\ref{prop:new_wstar_cent}.  As the image of $\theta$ is weak$^*$-closed, it has canonical predual $F$ which is a quotient of $\CZGD \wh\otimes \Ljd$.  Let $\pi : \CZGD \wh\otimes \Ljd \rightarrow F$ be the quotient map.  We hence corestrict $\theta$ to give an isomorphism $\theta:C^l_{cb}(\Ljd) \rightarrow F^*$.  Let $\alpha_0:E^* = \M_{cb}^l(\A(\GG)) \rightarrow C^l_{cb}(\Ljd)$ be the canonical bijection, and set $\alpha = \theta \circ \alpha_0: E^*\rightarrow F^*$.

Given $a\in \M^l_{cb}(\A(\GG))$ set $T=\alpha_0(a)$, so by definition, $a \wh\lambda(\omega) = \wh\lambda(T(\omega))$ for each $\omega\in \Ljd$.  Equivalently, $(\I\otimes a) {\ww}^{\whG} = (T^*\otimes\id)({\ww}^{\whG})$.  Given $\omega\in \Ljd, f\in \Lj$, set $u = \pi((\id\otimes f)({\ww}^{\whG}) \otimes \omega) \in F$, and calculate that
\begin{align*}
& \ \la \kappa_E(\widehat{\lambda}(\omega)f ) , a \ra_{E^{**}, E^*}
= \langle a, \wh\lambda(\omega) f \rangle_{E^*,E}
= \la a \wh{\lambda}(\omega),f\ra_{\Linf, \Lj}
= \langle (T^*\otimes\id)({\ww}^{\whG}), \omega\otimes f \rangle \\
&= \langle T^*\big( (\id\otimes f)({\ww}^{\whG}) \big), \omega \rangle_{\Linfd, \Ljd}
= \langle \theta(T) \big( (\id\otimes f)({\ww}^{\whG}) \big), \omega \rangle_{\Linfd, \Ljd}
= \la \alpha(a), u \ra_{F^*,F}.
\end{align*}
It follows that $\alpha^*(\kappa_F(u)) = \kappa_E(\widehat{\lambda}(\omega)f ) \in \kappa_E(E)$.  As the linear span of such elements $u$ is dense in $F$, the conditions of the lemma are verified, and the result follows.
\end{proof}

Using this result we can characterise functionals in the predual space $Q^l(\A(\GG))$.  In the following, we work with infinite matrices with entries in operator spaces, see \cite[Chapter~10]{EffrosRuan}.

\begin{proposition}\label{prop15}
For any Hilbert space $\msf{H}$ and $x\in \CZGD\otimes \mc{K}(\msf{H}),\,\omega\in \Ljd\wh{\otimes}\B(\msf{H})_*$, the bounded linear functional
\[
\Omega_{x,\omega}\colon \M^l_{cb}(\A(\GG))\ni a \mapsto
\la (\Theta^l(a)\otimes\id) x , \omega \ra \in \CC.
\]
belongs to $Q^l(\A(\GG))$. Furthermore, all functionals in $Q^l(\A(\GG))$ are of this form for some separable Hilbert space.
\end{proposition}

This result was recorded without proof in \cite[Proposition 3.2]{Crann}. In the classical context of locally compact groups, an analogous result was proved by Haagerup and Kraus in \cite[Proposition 1.5]{HaagerupKraus}. For the convenience of the reader, we give a proof using Theorem \ref{thm5}.

\begin{proof}
We first show that $\Omega_{x,\omega}$ is a member of $Q^l(\A(\GG))$. As $Q^l(\A(\GG))\subseteq \M^l_{cb}(\A(\GG))^*$ is norm closed, by first approximating $x,\omega$ by sums of elementary tensors, and then collapsing the pairing between $\mc{K}(\msf{H})$ and $\B(\msf{H})_*$, we reduce the problem to the case when $\msf{H} = \mathbb C$, when $x = (\id\otimes\theta){\ww}^{\whG}$ for some $\theta\in \Lj$, and when $\omega\in \Ljd$.  We then calculate
\[
\la \Theta^l(a) \bigl( (\id\otimes\theta){\ww}^{\whG}\bigr) , \omega\ra=
\la \wh{\lambda}(\omega\circ \Theta^l(a)),\theta \ra =
\la a \wh{\lambda}(\omega),\theta \ra =
\la a , \wh{\lambda}(\omega)\theta\ra\quad(a\in \M^l_{cb}(\A(\GG))),
\]
which shows that $\Omega_{x,\omega} = \alpha^l(\wh{\lambda}(\omega)\theta) \in Q^l(\A(\GG))$, as required.

Now, take any functional in $Q^l(\A(\GG))$, which by Theorem \ref{thm5} is represented by some element $\rho\in \Ljd\wh{\otimes}\CZGD$ (note that the projective operator space tensor product is symmetric). By \cite[Theorem 10.2.1]{EffrosRuan}, we can find infinite matrices
\[
\alpha\in \M_{1,\infty\times \infty},\; \beta\in \msf{K}_{\infty}(\Ljd),\;\gamma\in \msf{K}_\infty(\CZGD),\;
\alpha'\in \M_{\infty\times \infty,1}
\]
such that $\rho=\alpha(\beta\otimes \gamma)\alpha'$ (for the introduction to infinite matrices with entries in an operator space, see \cite[Sections 10.1, 10.2]{EffrosRuan}). Writing $\alpha=[\alpha_{1,(i,j)}]_{(i,j)\in \NN^2}$ etc., this means that 
\begin{equation}\label{eq14}
\la \Theta^l(a) ,\rho\ra=
\sum_{i,j,k,l=1}^{\infty} \alpha_{1,(i,j)} \la \Theta^l(a)( \gamma_{j,l}) , \beta_{i,k}\ra \alpha'_{(k,l),1}\qquad(a\in \M^l_{cb}(\A(\GG))).
\end{equation}
Let $\msf{H}$ be an infinite dimensional, separable Hilbert space with orthonormal basis $\{e_n\}_{n=1}^{\infty}$ and let $e_{i,j}\,(i,j\in\NN)$ be the corresponding rank one operators. Write $T(\msf{H})$ for the operator space of trace class operators, identified in a completely isometric way with $\B(\msf{H})_*$. For any $n\in\NN$ we have $\| [e_{j,i}]_{i,j=1}^{n}\|_{\M_n(T(\msf{H}))}=1$. Indeed, the matrix $[e_{j,i}]_{i,j=1}^{n}$ corresponds to the map $E_n\in \CB(\B(\msf{H}), \M_n)$ given by $E_n(x)=[ \Tr(e_{j,i} x)]_{i,j=1}^{n}$. If we denote by $V_n\colon \CC^n\rightarrow \msf{H}$ the canonical inclusion associated with the choice of basis, one easily sees that $E_n(x)=V_n^* x V_n\,(x\in \B(\msf{H}))$ and $\|E_n\|_{cb}=1$ follows. Consequently $[e_{j,i}]_{i,j=1}^{\infty}$ is a well defined matrix in $\M_\infty( T(\msf{H}))$. Finally, define
\[
\omega=\alpha(\beta\otimes [e_{j,i}]_{i,j=1}^{\infty})\alpha'\in \Ljd\wh{\otimes} T(\msf{H})=
\Ljd\wh\otimes \B(\msf{H})_*.
\]
A choice of basis gives us an isomorphism $\msf{H}\cong \ell^2$ and consequently we can consider $\gamma$ as an element of $\CZGD\otimes \mc{K}(\msf{H})$ (\cite[Equation 10.1.2]{EffrosRuan}). Finally, using equation \eqref{eq14} we can show that the functional associated to $\rho$ is of the form $\Omega_{\gamma,\omega}$.  Indeed, we have
\[\begin{split}
\la a , \Omega_{\gamma,\omega}\ra =
\la (\Theta^l(a)\otimes \id)\gamma,\omega \ra &=
\sum_{i,j,k,l=1}^{\infty}
\alpha_{1,(i,j)}
\la (\Theta^l(a)\otimes\id)\gamma,
\beta_{i,k}\otimes e_{l,j}\ra \alpha'_{(k,l),1}\\
&=
\sum_{i,j,k,l=1}^{\infty} \alpha_{1,(i,j)} \la \Theta^l(a)( \gamma_{j,l}) , \beta_{i,k}\ra \alpha'_{(k,l),1}
\end{split}\]
for any $a\in \M^l_{cb}(\A(\GG))$.
\end{proof}

\subsection{Viewing multipliers as bimodule maps}\label{section isomorphism}

In this section we provide another way of looking at CB multipliers and the associated weak$^*$-topology which will be useful in later considerations.

Let us first introduce some terminology. As usual, let $\CB^\sigma(\B(\LdG))$ be the space of normal CB maps on $\B(\LdG)$.  Observe $\CB^\sigma(\B(\LdG)) \cong \CB( \mc{K}(\LdG), \B(\LdG) )$. Indeed, any normal CB map on $\B(\LdG)$ can be restricted to $\mc{K}(\LdG)$ without changing its norm, and conversely, since $\mc{K}(\LdG)^{**}\simeq \B(\LdG)$, any CB map $\mc{K}(\LdG)\rightarrow \B(\LdG)$ uniquely extends to a normal CB map on $\B(\LdG)$. $\CB^\sigma(\B(\LdG))$ is an operator space which is equipped with the weak$^*$-topology given by the predual $\mc{K}(\LdG)\wh{\otimes} \B(\LdG)_*$. Via left and right multiplication, $\B(\LdG)$ becomes a $\Linf'$-bimodule, hence we can consider normal CB bimodule maps on $\B(\LdG)$. We can also look at those maps which leave $\Linfd\subseteq \B(\LdG)$ globally invariant. We will denote the set of CB normal $\Linf'$-bimodule maps on $\B(\LdG)$ which leave $\Linfd$ globally invariant by $\prescript{}{\Linf'}{\CB}_{\Linf'}^{\sigma,\Linfd}(\B(\LdG))$. One easily checks that this space is weak$^*$-closed 
in $\CB^\sigma(\B(\LdG))$, hence it naturally inherits an operator space structure and a weak$^*$-topology.

According to \cite[Theorem 4.5]{JungeNeufangRuan} (and \cite[Proposition 3.3]{DawsCPMults} for the left version), for any $a\in \M^l_{cb}(\A(\GG))$ there exists a unique map $\Phi(a)\in \prescript{}{\Linf'}{\CB}_{\Linf'}^{\sigma,\Linfd}(\B(\LdG))$ which extends $\Theta^l(a)\in \CB^\sigma(\Linfd)$. This map satisfies
\[
\I\otimes \Phi(a)(x)=
{\ww}^{\whG}\bigl(
(\Theta^l(a)\otimes \id)({\ww}^{\whG *}(\I\otimes x){\ww}^{\whG})\bigr)
{\ww}^{\whG *}\qquad(x\in \B(\LdG)).
\]
Furthermore, the resulting map
\begin{equation}\label{eq15}
\M^l_{cb}(\A(\GG))\ni a \mapsto
\Phi(a)\in
\prescript{}{\Linf'}{\CB}_{\Linf'}^{\sigma,\Linfd}(\B(\LdG))
\end{equation}
is a completely isometric isomorphism which is additionally a weak$^*$-homeomorphism (\cite[Theorem 6.2]{DawsCPMults}).

When $a$ arises from an element of the Fourier algebra, $\Phi(a)$ takes a special form.

\begin{lemma}\label{lemma9}
For $\omega\in\LL^1(\whG)$ let $a = \wh{\lambda}(\omega)\in \A(\GG)$, so that $\Theta^l(a) = (\omega\otimes\id)\wh{\Delta}$.  The associated map $\Phi(a)$ is
\[
\Phi(a)\colon \B(\LL^{2}(\GG))\ni x\mapsto (\omega\otimes\id)(\ww^{\whG *}(\I\otimes x)\ww^{\whG})\in \B(\LL^{2}(\GG)).
\]
\end{lemma}

A similar formula holds for an arbitrary element of the Fourier-Stieltjes algebra, see the discussion after Proposition 4.1 in \cite{DawsCPMults}.
 
\begin{proof}
As before, the left centraliser associated with $a$ is simply left multiplication by $\omega$, and so $\Theta^l(a)$ has the given form. Then for $x\in\B(\LL^2(\GG))$, using that $(\wh\Delta\otimes\id)(\ww^{\whG}) = \ww^{\whG}_{13}\ww^{\whG}_{23}$,
\begin{align*}
\I\otimes\Phi(a)(x) &= \ww^{\whG}\bigl( ((\omega\otimes\id)\wh{\Delta}\otimes\id)(\ww^{\whG *}(\I\otimes x)\ww^{\whG}) \bigr)\ww^{\whG *} \\
&= \ww^{\whG}\bigl( (\omega\otimes\id\otimes\id)(\ww_{23}^{\whG *} \ww_{13}^{\whG *} (\I\otimes \I\otimes x)\ww^{\whG}_{13} \ww^{\whG}_{23}) \bigr)\ww^{\whG *} \\
&= \ww^{\whG}\ww^{\whG *} (\omega\otimes\id\otimes\id)(\ww_{13}^{\whG *}(\I\otimes \I\otimes x) \ww_{13}^{\whG}) \ww^{\whG} \ww^{\whG *} \\
&= \I\otimes (\omega\otimes\id)(\ww^{\whG *}(\I\otimes x)\ww^{\whG}),
\end{align*}
and so $\Phi(a)$ has indeed the claimed form.
\end{proof}

\section{The approximation property} \label{section ap}

We define the approximation property for a locally compact quantum group $\GG$ (abbreviated AP) in a way completely analogous to the definition of AP for locally compact groups by Haagerup and Kraus in \cite{HaagerupKraus}. Recall that we fix the predual $Q^l(\A(\GG))$ of $\M^l_{cb}(\A(\GG))$, and we always refer to the corresponding weak$^*$-topology on $\M^l_{cb}(\A(\GG))$.

\begin{definition}\label{def1}
We say that a locally compact quantum group $\GG$ has the \emph{approximation property (AP)} if there is a net $(a_i)_{i\in I}$ in $\A(\GG)$ which converges to $\I$ in the weak$^*$-topology of $\M^l_{cb}(\A(\GG))$.
\end{definition}

\begin{remark}$ $
\begin{itemize}
\item We could call the above property ``left AP'' and introduce also a right variant of AP. However, in Proposition \ref{prop14} we will show that these properties are equivalent, so that
there is no need to distinguish between them.
\item A variant of AP was considered by Kraus-Ruan \cite{KrausRuan} (for Kac algebras) and Crann in \cite{CRANN_inneramenability}. Their property is a priori stronger, but in Theorem \ref{thm6} we show that this variant is in fact equivalent to Definition \ref{def1}. This proves a conjecture by Kraus-Ruan \cite[Remark 4.2]{KrausRuan}.
\end{itemize}
\end{remark}

Let us list some examples and counter-examples:

\begin{itemize}
\item In Section \ref{section relation} we show weak amenability implies AP, therefore all compact quantum groups and the discrete quantum groups $\wh{O_F^+}, \wh{U_F^+}$ have AP (\cite{Freslon, DFY_CCAP}). Furthermore, the locally compact quantum group $\SU_q(1,1)_{ext}$ has AP, see \cite{Caspers}.
\item Permanence properties of AP with respect to quantum subgroups (Theorem \ref{thm2}), direct products (Proposition \ref{prop9}), free products (Theorem \ref{thm1}) and the 
Drinfeld double construction (Theorem \ref{thm3}) allow us to construct examples with and without AP. For instance, the Drinfeld double of $\oon{SL}(3,\RR)$, or of any classical 
locally compact group without AP, see \cite{LafforgueDeLaSalle}, gives rise to non-classical quantum groups without AP.  Similarly, $\oon{SL}(3,\ZZ)\star \widehat{\oon{SU}_q(2)}$ is a quantum group without AP, as AP passes to quantum subgroups.
\item Examples of quantum groups with AP similarly arise.  As AP passes to free products, 
$\oon{SL}(2,\ZZ)\star\widehat{\oon{SU}_q(2)}$ is a quantum group with AP.  Similarly, $D(\oon{SL}(2,\RR))^{\wedge}$, the dual of the quantum double of $\oon{SL}(2,\RR)$, has AP (Theorem~\ref{thm3}).
\item The paper \cite{FimaMukherjeePatri} by Fima, Mukherjee and Patri makes a careful study of compact bicrossed products, and in particular, \cite[Theorem 6.7]{FimaMukherjeePatri} provides estimates on the Cowling--Haagerup constant (see Definition~\ref{defotherap}) of the resulting discrete quantum groups.  Using this, if one starts with a discrete group $\Gamma$ with AP and with $\Lambda_{cb}(\Gamma)\ge n$ (see for example \cite[Theorem~12.3.8]{BrownOzawa} and references therein) then one can construct a non-classical discrete quantum group with AP but with Cowling--Haagerup constant $\ge n$. Then by taking infinite products we obtain a discrete quantum group which has AP but which is not weakly amenable.
\item We leave as a question whether the Drinfeld double of $\SU_q(3)$ might be an example of a locally compact quantum group which does not have AP; for more see Remark~\ref{rem:suq3}.
\end{itemize}

\subsection{Equivalent characterisations} 

We check first that the approximation property is preserved under taking the \emph{commutant} quantum group $\GG'$, or the \emph{opposite} quantum group $\GG^{\oon{op}}$, for definitions see \cite[Section 4]{KVVN}.

\begin{proposition}\label{prop11} 
The following conditions are equivalent:
\begin{enumerate}
\item $\GG$ has AP,
\item $\GG'$ has AP.
\item $\GG^{\oon{op}}$ has AP,
\end{enumerate}
\end{proposition}
\begin{proof}
Assume that $\GG$ has AP, i.e.~ there is a net $(\omega_i)_{i\in I}$ in $\LL^1(\whG)$ such that $\lambda_{\whG}(\omega_i) \xrightarrow[i\in I]{} \I$ weak$^*$ in $\M_{cb}^l(\A(\GG))$.

First we will prove that $\GG'$ has AP. Recall that $\wh{\GG'}=\wh{\GG}^{\oon {op}}$ (\cite[Proposition 4.2]{KVVN}). Using Proposition \ref{prop15}, choose an arbitrary functional $\Omega_{y,\nu}\in Q^l(\A(\GG'))$ where $\msf{H}$ is a separable Hilbert space, $y\in \mrm{C}_0(\whG)\otimes\mc{K}(\msf{H})$ and $\nu \in \LL^1(\whG) \wh\otimes \B(\msf{H})_*$. Pick some selfadjoint antiunitary $\mc{J}$ on $\msf{H}$ (for example, pick an orthonormal basis and let $\mc{J}$ be coordinate-wise complex conjugation) and define $j:\mc{K}(\msf{H}) \rightarrow\mc{K}(\msf{H})$ by $j(x) = \mc{J}x^*\mc{J}$, an $\star$-antihomomorphism with $j^2=\id$.  Then $\wh{R}\otimes j$ is a well-defined bounded map on the spatial tensor product $\mrm{C}_0(\whG)\otimes\mc{K}(\msf{H})$, and $\wh{R}_* \otimes j^*$ is well-defined on $\LL^1(\whG) \wh\otimes \B(\msf{H})_*$. Indeed, $\wh{R}\otimes j$ acts via $(\wh{R}\otimes j)(X)=(J_{\vp} \otimes \mc{J})X^*(J_{\vp}\otimes \mc{J})$ for $X\in \mrm{C}_0(\whG)\otimes \mc{K}(\msf{H})$ and then we can define $\wh{R}_*\otimes j^*$ as the restriction of $(\wh{R}\otimes j)^*$ to $\LL^1(\whG)\wh\otimes \B(\msf{H})_*$: $(\wh{R}\otimes j)^*$ preserves this subspace as $(\wh{R}\otimes j)^*(\omega_{\xi\otimes \eta})=\omega_{J_\vp \xi\otimes \mc{J}\eta}$ for $\xi\in \LdG,\eta\in \msf{H}$. Since $\LL^1(\whG)\wh\otimes \mc{K}(\msf{H})^*\subseteq (\mrm{C}_0(\whG)\otimes\mc{K}(\msf{H}))^*$ is closed, the claim follows. Furthermore, both these maps are isometric bijections.

Set $x = (\wh{R}\otimes j)(y), \omega = (\wh{R}_*\otimes j^*)(\nu)$. Then using Lemma \ref{lemma9}
\begin{align*}
\la \lambda_{\whG}(\omega_i),\Omega_{x,\omega}\ra 
&= \langle ((\omega_i\otimes\id)\Delta_{\whG}\otimes\id)(\wh{R}\otimes j)(y), (\wh{R}_*\otimes j^*)(\nu) \rangle \\
&= \langle (\wh{R}(\omega_i\otimes\id)\Delta_{\whG}\wh{R}\otimes\id) (y), \nu \rangle
= \langle ((\id\otimes \wh{R}_*(\omega_i))\Delta_{\whG} \otimes\id) (y), \nu \rangle \\
&= \langle ((\wh{R}_*(\omega_i) \otimes \id)\Delta_{\whG^{\oon{op}}} \otimes\id) (y), \nu \rangle
= \la \lambda_{\whG^{\oon{op}}}(\wh{R}_*(\omega_i)) , \Omega_{y,\nu} \ra 
\end{align*}
and since $\lambda_{\whG}(\omega_i) \xrightarrow[i\in I]{}\I$ weak$^*$, we conclude $\lambda_{\whG^{\oon{op}}}(\wh{R}_*(\omega_i)) \xrightarrow[i\in I]{}\I$ weak$^*$. 
This shows that $\GG'$ has AP.

Next we prove that $\GG^{\oon{op}}$ has AP. By \cite[Proposition 4.2]{KVVN} we have $\wh{\GG^{\oon{op}}}=\wh{\GG}'$. Write $R^{\sim}$ for the extension of the unitary antipode on $\GG$, $\B(\LdG)\ni x\mapsto J_{\hvp} x^* J_{\hvp} \in \B(\LdG)$, so that $\tilde{\omega}_i=\omega_i\circ R^{\sim}\in \LL^1({\whG}')$. We claim that the net $(\lambda_{\whG'}(\omega_i\circ R^{\sim}))_{i\in I}$ converges weak$^*$ to $\I$ in $\M^l_{cb}(\A(\GG^{\oon{op}}))$. Take $z\in \mrm{C}_0(\whG') \otimes \mc{K}(\msf{H}), \theta\in \LL^1(\whG')\wh{\otimes}\B(\msf{H})_*$. Recall that $\mrm{C}_0(\whG')=J_{\hvp} \mrm{C}_0(\whG)J_{\hvp}$ and 
\[\begin{split}
\Delta_{\whG'}\colon \Linfd'\ni x \mapsto
&(J_{\hvp}\otimes J_{\hvp})\Delta_{\whG}( J_{\hvp}xJ_{\hvp})
(J_{\hvp}\otimes J_{\hvp})\\
=&
(R^{\sim}\otimes R^{\sim})\Delta_{\whG}(R^{\sim}(x))\in
\Linfd'\bar{\otimes}\Linfd'.
\end{split}\]
Using this, we obtain
\[\begin{split}
&\quad\;
\la \lambda_{\whG'}(\omega_i\circ R^{\sim}) , \Omega_{z,\theta}\ra=
\la (\Theta^l(\lambda_{\whG'}(\omega_i \circ R^{\sim}))\otimes\id)
z,\theta \ra =
\la
((\omega_i\circ R^{\sim}\otimes \id)\Delta_{\whG'}\otimes \id)z,\theta\ra \\
&=
\la (R^{\sim}\otimes j)\bigl((\omega_i \otimes \id)\Delta_{\whG}\otimes \id\bigr)(R^{\sim}\otimes j)z,\theta\ra =
\la \lambda_{\whG}(\omega_i),\Omega_{(R^{\sim}\otimes j)z,\theta\circ (R^{\sim}\otimes j)}\ra\\
&\xrightarrow[i\in I]{}
\la \I, \Omega_{(R^{\sim}\otimes j)z,\theta\circ (R^{\sim}\otimes j)}\ra =
\la z,\theta\ra = \la \I,\Omega_{z,\theta}\ra
\end{split}\]
and thus $\GG^{\oon{op}}$ has AP.   The converse implications follow since $(\GG')'=\GG$ and $(\GG^{\oon{op}})^{\oon{op}}=\GG$.
\end{proof}

The next result shows that the version of the approximation property considered in \cite{KrausRuan} and \cite{CRANN_inneramenability} is equivalent to AP as defined in Definition~\ref{def1}.  Both \cite[Definition~4.1]{KrausRuan} and \cite[Page~1728]{CRANN_inneramenability} take condition (\ref{thm6:2}) of the following theorem as their definition of AP.

\begin{theorem}\label{thm6}
The following conditions are equivalent:
\begin{enumerate}
\item\label{thm6:1} $\GG$ has AP,
\item\label{thm6:2} there is a net $(a_i)_{i\in I}$ in the Fourier algebra $\A(\GG)$, such that the corresponding net $(\Theta^l(a_i))_{i\in I}$ converges to the identity in the stable point-weak$^*$-topology of $\CB^\sigma(\Linfd)$.
\end{enumerate}
\end{theorem}

In order to prove Theorem \ref{thm6} we need to establish some preliminary results. Recall that $\Linf$ is a right $\Lj$-module via $x\star\omega=(\omega\otimes\id)\Delta(x)$ 
for $x\in\Linf,\omega\in\Lj$.

\begin{proposition}\label{prop16}
Let $a\in \Linf$ and $\omega\in \Lj$.
\begin{enumerate}
\item\label{prop16:one} If $a\in \A(\GG)$ then $a\star\omega\in \A(\GG)$.
\item\label{prop16:two} If $a\in \M^l_{cb}(\A(\GG))$ then $a\star\omega \in \M_{cb}^l(\A(\GG))$ with $\|a\star\omega\|_{cb} \leq \|a\|_{cb} \|\omega\|$ and
\[
\Theta^l(a\star\omega)(\wh x) = (\omega\otimes\id)\big( (\id\otimes\Theta^l(a))((\I\otimes\wh x){\ww}^{\GG *}) \ww^{\GG} \big)
\qquad (\wh x\in L^\infty(\wh\GG)).
\]
\end{enumerate}
\end{proposition}

\begin{proof}
(1) Write $a=\wh{\lambda}(\wh\omega)$ for $\wh\omega\in\Ljd$. Then
\[\begin{split}
&\quad\;
a\star\omega=
(\omega\otimes\id)\Delta\bigl( (\wh\omega\otimes\id){\ww}^{\whG}\bigr)=
(\wh\omega \otimes \omega \otimes\id)\bigl({\ww}^{\whG}_{13}{\ww}^{\whG}_{12}\bigr)\\
&=
(\wh\omega\otimes\id)\bigl( {\ww}^{\whG} ((\id\otimes\omega)\ww^{\whG}\otimes \I)\bigr)=
\wh{\lambda}\bigl(\wh\omega\bigl(\cdot\;
(\id\otimes\omega)\ww^{\whG}\bigr)\bigr)\in \A(\GG)
\end{split}\]
as required. \\
(2) As $\ww^{\GG}\in \Linf\bar{\otimes}\Linfd$, there is a well-defined linear map $T$ on $\Linfd$ given by
\[
T(\wh x) = (\omega\otimes\id)\big( (\id\otimes\Theta^l(a))((\I\otimes\wh x)\ww^{\GG *}) \ww^{\GG} \big)
\qquad (\wh x\in L^\infty(\wh\GG)).
\]
Clearly $T$ is completely bounded with $\|T\|_{cb}\le \|a\|_{cb}\|\omega\|$ and weak$^*$-continuous.  We first show that $T$ is the adjoint of a centraliser,
equivalently, that $\wh\Delta T = (T\otimes\id)\wh\Delta$. If $\wh x\in L^\infty(\wh\GG)$ then using $\wh\Delta\Theta^l(a) = (\Theta^l(a)\otimes\id)\wh\Delta$ gives
\begin{align*}
\wh\Delta T(\wh x) &= (\omega\otimes\id\otimes\id)\bigl( (\id\otimes\wh\Delta\Theta^l(a))((\I\otimes\wh x)\ww^{\GG *}) \ww^{\GG}_{13} \ww^{\GG}_{12} \bigr) \\
&= (\omega\otimes\id\otimes\id)\bigl( (\id\otimes\Theta^l(a)\otimes\id)\bigl((\I\otimes\wh\Delta(\wh x))\ww^{\GG *}_{12} \ww^{\GG*}_{13}\bigr) \ww^{\GG}_{13} \ww^{\GG}_{12} \bigr) \\
&= (\omega\otimes\id\otimes\id)\bigl( (\id\otimes\Theta^l(a)\otimes\id)\bigl((\I\otimes\wh\Delta(\wh x))\ww^{\GG *}_{12}\bigr) \ww^{\GG}_{12} \bigr) \\
&= (T\otimes\id)\wh\Delta(\wh x).
\end{align*}
Consequently $T$ is the adjoint of a centraliser, and so there exists $b\in \M_{cb}^l(\A(\GG))$ with $T = \Theta^l(b)$. Then $b \wh\lambda(\wh\omega) = \wh\lambda(\wh\omega\circ T)$ for each $\wh\omega\in\Ljd$, equivalently, $(\I\otimes b)\ww^{\whG} = (T\otimes\id)(\ww^{\whG})$. In other words, we have 
\[
(b\otimes \I)\ww^{\GG *} = (\id\otimes T)(\ww^{\GG *})
= (\id\otimes\omega\otimes\id)\big( (\id\otimes\id\otimes\Theta^l(a))( \ww^{\GG *}_{13} \ww^{\GG *}_{23} ) \ww^{\GG}_{23} \big), 
\]
or equivalently 
\begin{align}
b\otimes \I &= (\id\otimes\omega\otimes\id)\bigl( (\id\otimes\id\otimes\Theta^l(a))( \ww^{\GG *}_{13} \ww^{\GG *}_{23} ) \ww^{\GG}_{23} \ww^{\GG}_{13} \bigr) \notag \\
&= (\omega\otimes\id\otimes\id)\bigl( (\id\otimes\id\otimes\Theta^l(a))( \ww^{\GG *}_{23} \ww^{\GG *}_{13} ) \ww^{\GG}_{13} \ww^{\GG}_{23} \bigr). \label{eq:p16:one}
\end{align}
Now, $(\I\otimes a) \ww^{\whG}= (\Theta^l(a)\otimes\id)(\ww^{\whG})$ so $a \otimes \I = (\id\otimes\Theta^l(a))(\ww^{\GG *})\ww^{\GG}$ and hence
\begin{align}
a\star\omega\otimes \I &= (\omega\otimes\id)\Delta(a)\otimes \I
= ((\omega\otimes\id)\Delta\otimes\id)\bigl( (\id\otimes\Theta^l(a))(\ww^{\GG *})\ww^{\GG}\bigr) \notag \\
&= (\omega\otimes\id\otimes\id)\bigl( (\id\otimes\id\otimes\Theta^l(a))(\ww^{\GG *}_{23} \ww^{\GG *}_{13})\ww^{\GG}_{13} \ww^{\GG}_{23} \bigr). \label{eq:p16:two}
\end{align}
As \eqref{eq:p16:one} and \eqref{eq:p16:two} agree, we conclude that $b = a\star\omega$.  Thus $a\star\omega\in \M_{cb}^l(\A(\GG))$ with $\Theta^l(a\star\omega)=T$ as required. 
\end{proof}

\begin{lemma}\label{lemma21}
For any locally compact quantum group $\GG$ and $x\in \Linf, a\in \mrm{C}_0(\GG), \omega\in \Lj$ we have $\omega \star (xa) \in \mrm{C}_0(\GG)$.
\end{lemma}

\begin{proof}
As $\Lj$ is a closed $\mrm{C}_0(\GG)$-submodule of $\mrm{C}_0(\GG)^*$, by \cite[Lemma~2.1]{RundeUCB} we know that $\omega = b \omega_1$ for some $b\in \mrm{C}_0(\GG)$ and $\omega_1\in \Lj$.  Then
\[ \omega \star (xa) = (\id\otimes\omega)\Delta(xa)
= (\id\otimes \omega_1)\big( \Delta(x) \Delta(a)(\I\otimes b) \big). \]
As $a,b\in \mrm{C}_0(\GG)$ we know that $\Delta(a)(\I\otimes b) \in \mrm{C}_0(\GG)\otimes \mrm{C}_0(\GG)$, the minimal $C^*$-algebraic tensor product.  By continuity, it hence suffices to prove that
\[
 (\id\otimes \omega_1)\big( \Delta(x)(c\otimes d) \big) \in \mrm{C}_0(\GG) 
\]
for $c,d\in \mrm{C}_0(\GG)$.  However, this equals $\bigl( (\id\otimes d\omega_1)\Delta(x)\bigr) c$ and by \cite[Theorem~2.4]{RundeUCB} we know that $(\id\otimes d\omega_1)\Delta(x) \in \M (\mrm{C}_0(\GG))$, and so the result follows.
\end{proof}

Next we introduce certain functionals in $Q^l(\A(\GG))$ in analogy to \cite[Proposition~1.3]{HaagerupKraus}. For a Hilbert space $\msf{H}$, $x\in \Linfd\bar{\otimes}\B(\msf{H}), \omega\in \Ljd\wh{\otimes}\B(\msf{H})_*$ and $f\in\Lj$ define
\begin{equation}\label{eq17}
\Omega_{x,\omega,f}\colon \M^l_{cb}(\A(\GG))\ni a\mapsto
\la 
(\Theta^l(a\star f)\otimes \id)x,\omega \ra \in \CC.
\end{equation}
Note that $\Omega_{x,\omega,f}$ is well-defined and bounded by Proposition \ref{prop16}. 

\begin{proposition}\label{prop17}
The linear functional $\Omega_{x,\omega,f}$ is weak$^*$-continuous, hence is contained in $Q^l(\A(\GG))$.
\end{proposition}

\begin{proof}
Clearly $\Omega_{x,\omega,f}$ is bounded with $\|\Omega_{x,\omega,f}\|\le \|x\| \|f\|\|\omega\|$, so it suffices to prove the result when $\omega$ is in the algebraic tensor 
product of $\Ljd$ with $\B(\msf H)_*$, and hence by linearity, we may suppose that $\omega = \wh\omega\otimes u$.  Then
\[ \Omega_{x,\omega,f}(a) = \langle (\Theta^l(a\star f)\otimes \id)(x), \wh\omega\otimes u \rangle
= \langle \Theta^l(a\star f)( (\id\otimes u)(x) ), \wh\omega \rangle
\qquad(a\in\M^l_{cb}(\A(\GG))).
\]
Thus, it suffices to show that for $\wh\omega\in \Ljd$ and $\wh x\in \Linfd$
\[
\mu\colon\M^l_{cb}(\A(\GG))\ni a\mapsto \langle \Theta^l(a\star f)(\wh x), \wh\omega \rangle\in \CC
 \]
is weak$^*$-continuous. By Proposition~\ref{prop16}, given the form of $\Theta^l(a\star f)$,
\[
\mu(a) = \langle (\id\otimes\Theta^l(a))((\I\otimes\wh x)\ww^{\GG *}) \ww^{\GG}, f\otimes \wh\omega \rangle
\qquad(a\in \M^l_{cb}(\A(\GG))).
\]
As $\ww^{\GG}\in \Linf\bar{\otimes}\Linfd$, also $\ww^{\GG}(f\otimes \wh\omega) \in \Lj\wh\otimes \Ljd$, so again by approximation, it suffices to show that for $f'\in \Lj, \wh\omega'\in \Ljd$ the map 
\[
\mu'\colon \M^l_{cb}(\A(\GG))\ni a\mapsto \langle (\id\otimes\Theta^l(a))((\I\otimes\wh x)\ww^{\GG *}), f'\otimes \wh\omega' \rangle
= \langle \Theta^l(a)( \wh x \wh y \big), \wh\omega' \rangle\in \CC
 \]
is weak$^*$-continuous, where $\wh y = (f'\otimes\id)(\ww^{\GG *}) \in \mrm{C}_0(\wh\GG)$.

By linear density of products, compare Lemma \ref{lemma17}, it suffices to consider the case when $\wh\omega' = \wh\omega_1\star\wh\omega_2$ for $\wh\omega_1,\wh\omega_2\in \Ljd$.  As $\Theta^l(a)_*(\wh\omega_1\star\wh\omega_2) = \Theta^l(a)_*(\wh\omega_1)\star\wh\omega_2$ by the left centraliser property, we see that
\[
\mu'(a) = 
\la \Theta^l(a)(\wh x \wh y),\wh\omega_1\star \wh\omega_2\ra=
\langle \wh\omega_2\star(\wh x \wh y), \Theta^l(a)_*(\wh\omega_1) \rangle
\qquad(a\in \M^l_{cb}(\A(\GG))).
\]
As $\wh y\in \mrm{C}_0(\wh\GG)$, by Lemma~\ref{lemma21} applied to $\wh\GG$, we know that $\wh\omega_2\star(\wh x \wh y) \in \mrm{C}_0(\wh\GG)$.  Thus $\mu' \in Q^l(\A(\GG))$ by Proposition \ref{prop15}.
\end{proof}

We are now ready to prove Theorem~\ref{thm6}.

\begin{proof}[Proof of Theorem \ref{thm6}]
$(2)\Rightarrow (1)$ follows directly from the characterisation of functionals in the predual $Q^l(\A(\GG))$ of $\M^l_{cb}(\A(\GG))$ in Proposition~\ref{prop15}.

$(1)\Rightarrow (2)$ Assume that $(a_i)_{i\in I}$ is a net in $\A(\GG)$ which converges weak$^*$ to $\I$ in $\M^l_{cb}(\A(\GG))$. Pick a state $f\in \Lj$, and for each $i\in I$ set $b_i = a_i \star f$. By Proposition~\ref{prop16} we have  $b_i \in \A(\GG)$. We now show that $(\Theta^l(b_i))_{i\in I}$ converges to the identity in the stable point-weak$^*$-topology. 

Given a separable Hilbert space $\msf{H}$, $x\in \Linfd\bar\otimes \B(\msf{H})$, and $\omega \in \Ljd \wh\otimes \B(\msf H)_*$, using Proposition~\ref{prop17} we see that
\[\begin{split}
&
\langle (\Theta^l(b_i)\otimes\id)x, \omega \rangle=
\la (\Theta^l(a_i\star f)\otimes\id) x, \omega \ra 
= \langle a_i, \Omega_{x,\omega,f} \rangle\\
\xrightarrow[i\in I]{}&
 \langle \I, \Omega_{x,\omega,f} \rangle
= \langle (\Theta^l(\I\star f)\otimes\id)(x), \omega \rangle,
\end{split}\]
Since $f$ is a state, we have $\I\star f=(f\otimes \id)\Delta(\I)=\I$, hence $\Theta^l(\I\star f)=\id$, and so $(\Theta^l(b_i)\otimes\id)(x)\xrightarrow[i\in I]{} x$ weak$^*$ as required.
\end{proof}

\subsection{Further general properties}

Let us start with an auxiliary technical result. Recall that for a von Neumann algebra $\M$ and a linear map $T\colon \M\rightarrow \M$ we define $T^\dagger\colon \M\ni x \mapsto T(x^*)^*\in \M$.

\begin{lemma}\label{lemma18}
If $T\in \CB^\sigma(\M)$, then $T^\dagger\in \CB^\sigma(\M)$ and $\|T^\dagger\|_{cb}=\|T\|_{cb}$. If $\M=\Linf$ for a locally compact quantum group $\GG$ then $R\circ T\circ R\in \CB^\sigma(\Linf)$ and $\|R\circ T\circ R\|_{cb}=\|T\|_{cb}$. Both operations $T\mapsto T^\dagger, T\mapsto R\circ T \circ R$ are continuous with respect to the stable point-weak$^*$-topology.
\end{lemma}

\begin{proof}
Let $\M\subseteq \B(\msf{H})$. Using the normal version of Wittstock's Theorem (compare with the start of the proof of \cite[Theorem~2.5]{haagerupmusat} for example), we can find a Hilbert space $\msf{K}$, bounded linear maps $V,W\colon \msf{H}\rightarrow\msf{K}$ and a normal representation $\pi\colon \M\rightarrow\B(\msf{K})$ such that $T=W^* \pi(\cdot ) V$ and $\|T\|_{cb}=\|V\|\|W\|$. Now proving the claimed properties of $T^{\dagger}$ is elementary.

Assume next that $\M=\Linf$. If $T\in \CB^\sigma(\Linf)$, then clearly $R\circ T\circ R$ is normal since the unitary antipode $R$ is normal. Using again Wittstock's Theorem,
write $T^\dagger=W^*\pi(\cdot)V$ and choose an antiunitary $\mc{J}$ on $\msf{H}$ which satisfies $\mc{J}^*=\mc{J}$. Then, for $x\in\Linf$,
\[\begin{split}
&\quad\;
R\circ T\circ R(x)=
J_{\hvp} T( J_{\hvp} x^* J_{\hvp})^* J_{\hvp}=
J_{\hvp} T^\dagger( J_{\hvp} x J_{\hvp}) J_{\hvp}\\
&=
J_{\hvp} W^* \pi( J_{\hvp} x J_{\hvp}) V J_{\hvp}=
(\mc{J} W J_{\hvp})^* \mc{J} \pi(J_{\hvp} x J_{\hvp})
\mc{J} (\mc{J} V J_{\hvp}).
\end{split}\]
As $\Linf \ni x\mapsto \mc{J} \pi( J_{\hvp} x J_{\hvp} ) \mc{J}\in \B(\msf{H})$ is a $\star$-homomorphism, it follows that $R\circ T \circ R$ is CB with $\|R\circ T\circ R\|_{cb}\leq \|T\|_{cb}$.  Again, as $R\circ(R\circ T\circ R)\circ R = T$, we have in fact $\|R\circ T\circ R\|_{cb} = \|T\|_{cb}$.

Let $(T_i)_{i\in I}$ be a net in $\CB^\sigma(\Linf)$ converging to $T$ in the stable point-weak$^*$-topology. Choose a self-adjoint antiunitary $\mc{J}'$ on $\ell^2$ and define $j=\mc{J}'(\cdot)^* \mc{J}',$ a normal $\star$-antiautomorphism of $\B(\ell^2)$. Then $R\otimes j$ extends to a well-defined normal bounded linear map on $\Linf\bar{\otimes} \B(\ell^2)$. Indeed, in the proof of Proposition~\ref{prop11} we argued that $R_*\otimes ( j|_{\mc{K}(\ell^2)})^*$ is a bounded linear map on $\LL^1(\GG)\wh{\otimes} \B(\ell^2)_*$, and we just need to take the dual map $R\otimes j=(R_*\otimes (j|_{\mc{K}(\ell^2)})^*)^*$. For $x\in \Linf\bar{\otimes} \B(\ell^2), \omega\in \Lj\wh{\otimes}\B(\ell^2)_*$ we have
\[\begin{split}
&
\la (R\circ T_i\circ R\otimes \id) x, \omega \ra =
\la (R\circ T_i\circ R\otimes j^2) x, \omega \ra =
\la (T_i\otimes \id) \bigl(
(R\otimes j)(x)\bigr) , \omega \circ (R\otimes j)\ra \\
\xrightarrow[i\in I]{}&
\la (T\otimes \id) \bigl(
(R\otimes j)(x)\bigr) , \omega \circ (R\otimes j)\ra =
\la (R\circ T\circ R\otimes \id)x,\omega \ra,
\end{split}\]
which concludes the proof.
\end{proof}

We now show that $\M^l_{cb}(\A(\GG))$ admits an interesting involution. Recall that $S$ denotes the antipode of $\GG$.

\begin{proposition}\label{prop12}
Let $a\in \M^l_{cb}(\A(\GG))$. Then $a^*\in \Dom(S)$ and $S(a^*)\in \M^l_{cb}(\A(\GG))$ with $\Theta^l( S(a^*))=\Theta^l(a)^{\dagger}$.
\end{proposition}

\begin{proof}
By Lemma \ref{lemma18}, we know that $\Theta^l(a)^{\dagger}\in \CB^{\sigma}(\Linf)$. One easily checks that $\Theta^l(a)^{\dagger}$ is a left $\Ljd$-module map, hence $\Theta^l(a)^{\dagger}=\Theta^l(b)$ for some $b\in \M^l_{cb}(\A(\GG))$. The claim follows now from \cite[Theorem 5.9]{Daws_Hilbert}. 

For the convenience of the reader let us also indicate a direct argument. As $\Theta^l(a)^{\dagger}=\Theta^l(b)$ we have that
\[
(\I\otimes b) \ww^{\whG}=(\Theta^l(b)\otimes\id)\ww^{\whG}=
(\Theta^l(a)^{\dagger}\otimes\id)\ww^{\whG}=
( (\Theta^l(a)\otimes\id)\ww^{\whG *})^*,
\]
and hence
\[
\ww^{\whG *}(\I\otimes b^*)=
(\Theta^l(a)\otimes \id)\ww^{\whG *}\quad\Rightarrow\quad
(\id\otimes \wh\omega)(\ww^{\GG}) b^*=(\id\otimes \wh\omega\circ \Theta^l(a))\ww^{\GG}
\]
for each $\wh\omega\in \Ljd$.  In what follows, we treat $S$ as a densely-defined, closed operator on $\Linf$ equipped with the weak$^*$-topology.  From \cite[Proposition~2.24]{LCQGDaele}, compare \cite[Proposition~8.3]{KustermansVaes}, we know that for any $\wh\omega$ we get $(\id\otimes\wh\omega)\ww^{\GG} \in D(S)$ and $S((\id\otimes\wh\omega)\ww^{\GG}) = (\id\otimes\wh\omega)\ww^{\GG *}$. It follows that $(\id\otimes \wh\omega)(\ww^{\GG}) b^* \in D(S)$ with
\begin{align}
S\bigl( (\id\otimes\wh\omega) (\ww^{\GG}) b^*\bigr) &=
(\id\otimes \wh\omega\circ \Theta^l(a))\ww^{\GG *} =
(\wh\omega\circ \Theta^l(a)\otimes \id)\ww^{\whG} =
a\,(\wh\omega\otimes\id)(\ww^{\whG})    \notag  \\
&= a\,(\id\otimes\wh\omega)\ww^{\GG *}
= a\,S\big( (\id\otimes\wh\omega)\ww^{\GG} \big).
\label{eq19}
\end{align}
Let $C = \{ (\id\otimes\wh\omega)\ww^{\GG} \,|\, \wh\omega \in \Ljd \} \subseteq D(S)$.  We shall show that $C$ contains a net $(a_i)_{i\in I}$ such that both $a_i\xrightarrow[i\in I]{}\I$ and $S(a_i)\xrightarrow[i\in I]{}\I$ weak$^*$ in $\Linf$.  It follows that $a_i b^* \xrightarrow[i\in I]{} b^*$ and $a S(a_i) \xrightarrow[i\in I]{} a$ weak$^*$, and so as $S$ is weak$^*$-closed, it follows from \eqref{eq19} that $b^* \in D(S)$ with $S(b^*) = a$.  Hence $a^*=S(b^*)^*\in \Dom(S)$ and $S(a^*)=S( S(b^*)^*)=b$, as claimed.

We now show the claim about $C$, using some standard ``smearing'' techniques, compare \cite{KustermansOneParam}.  For $a\in\CZG$ and $r>0, z\in\mathbb C$, define
\[ a(r,z) = \frac{r}{\sqrt\pi} \int_{\mathbb R} \exp(-r^2(t-z)^2) \tau_t(a) \md t. \]
Then $a(r,z)$ is analytic for the one-parameter automorphism group $(\tau_t)_{t\in \RR}$ with $\tau_w(a(r,z)) = a(r,z+w)$ for $w\in\mathbb C$.  Similarly, for $\wh\omega\in \Ljd$, define
\[ \wh\omega(r,z) = \frac{r}{\sqrt\pi} \int_{\mathbb R} \exp(-r^2(t-z)^2) \wh\omega\circ\wh\tau_t \md t. \]
Given $\wh\omega$, let $a = (\id\otimes\wh\omega)(\ww^{\GG})$.  As $(\tau_t\otimes\wh\tau_t)(\ww^{\GG}) = \ww^{\GG}$, it follows that $\tau_t(a) = (\id\otimes\wh\omega\circ\wh\tau_{-t})(\ww^{\GG})$ and hence $a(r,z) = (\id\otimes \wh\omega(r,-z))(\ww^{\GG})$.  Finally, as $S = R\tau_{-i/2}$, it follows that $S(a(r,z)) = R(a(r,z-i/2)) = (\id\otimes(\wh\omega\circ\wh R)(r,-z+i/2))(\ww^{\GG})$.

As $C$ is norm dense in $\CZG$, we can find a net $(\wh\omega_i)_{i\in I}$ with $a_i = (\id\otimes\wh\omega_i)(\ww^{\GG})\xrightarrow[i\in I]{} \I$ strictly.  By \cite[Proposition~2.25]{KustermansOneParam}, the net $(a_i(r,z))_{i\in I}$ converges strictly to $\I(r,z) = \I$, for any choice of $r,z$.  By the above discussion, $a_i(r,z)\in C$.  Then also $S(a_i(r,z)) = R(a_i(r,z-i/2)) \xrightarrow[i\in I]{} R(\I)=\I$ strictly. Cohen--Hewitt's Factorisation Theorem shows that strict convergence implies weak$^*$-convergence in $\Linf$, hence we have constructed the required net.
\end{proof}

\begin{corollary}\label{cor2}
Let $a\in \M^l_{cb}(\A(\GG))$. Then $\Theta^l(a)\in\CB^\sigma(\Linfd)$ preserves the adjoint if and only if $S(a^*)=a$.
\end{corollary}

Next we check that $\M^l_{cb}(\A(\GG))$ is globally invariant under the scaling and modular automorphism groups.

\begin{lemma}\label{lemma6}
Let $\GG$ be a locally compact quantum group and let $a\in \M^l_{cb}(\A(\GG))$.  For any $t\in \RR$,
\begin{itemize}
\item $\tau_t(a)\in \M^l_{cb}(\A(\GG))$ with $\Theta^l(\tau_t(a))=\hat{\tau}_t\circ \Theta^l(a)\circ \hat{\tau}_{-t}$.
\item $\sigma^\vp_t(a)\in \M^l_{cb}(\A(\GG))$ with 
$\Theta^l(\sigma^\vp_t(a))\colon x\mapsto 
\hat{\delta}^{it}\,\hat{\tau}_t\circ\Theta^l(a)\bigl(\hat{\delta}^{-it} \hat{\tau}_{-t}(x) \bigr)$.
\item $\sigma^\psi_t(a)\in \M^l_{cb}(\A(\GG))$ with 
$\Theta^l(\sigma^\psi_t(a))\colon x\mapsto\hat{\tau}_{-t}\circ\Theta^l(a)\bigl( \hat{\tau}_{t}(x) \hat{\delta}^{-it} \bigr) \hat{\delta}^{it}$.
\end{itemize}
\end{lemma}

\begin{proof}
Take $\omega\in \Ljd$. Using $(\hat{\tau}_t\otimes \tau_t)\ww^{\whG}=\ww^{\whG}$ (see proof of \cite[Proposition 6.10]{KustermansVaes}) we obtain
\[\begin{split}
&\quad\;
\tau_t(a) \wh{\lambda}(\omega)=\tau_t \bigl(a\, \tau_{-t}( \wh{\lambda}(\omega ))\bigr)=
\tau_t\bigl( a \wh{\lambda}( \omega\circ \hat{\tau}_t)\bigr)=
\wh{\lambda}\bigl( \Theta^l(a)_*( \omega\circ\hat{\tau}_t)\circ \hat{\tau}_{-t}\bigr)
\end{split}\]
which implies $\tau_t(a)\in \M^l_{cb}(\A(\GG))$ and $\Theta^l(\tau_t(a))=\hat{\tau}_t\circ \Theta^l(a)\circ \hat{\tau}_{-t}$; notice that clearly the right-hand-side of this final expression gives a completely bounded map.

We now use the following facts.  Firstly, by the definition of $\ww^{\GG}$ we have $(\rho\otimes \id)(\ww^{\GG *})\Lambda_{\vp}(x)=\Lvp( (\rho\otimes \id)\Delta(x))$ for $\rho \in \LL^1(\GG), x\in \mf{N}_{\vp}$.  Secondly, $(\sigma^\vp_t\otimes \sigma^\psi_{-t})\circ \Delta=\Delta\circ\tau_t$, see \cite[Proposition~6.8]{KustermansVaes}.  Thus
\[\begin{split}
&\quad\;
(\rho\circ \sigma^\vp_t\otimes \id)(\ww^{\GG *})\Lvp(x)=
\Lvp((\rho\circ\sigma^\vp_t\otimes\id)\Delta(x))=
\Lvp\bigl((\rho\otimes \sigma^\psi_t)\Delta(\tau_t(x))\bigr)\\
&=
\nu^{\frac{t}{2}} \nabla_\psi^{it} \Lvp( (\rho\otimes \id)\Delta(\tau_t(x)))=
\nabla_\psi^{it} (\rho\otimes \id)(\ww^{\GG *}) P^{it}\Lvp(x),
\end{split}\]
which implies that $(\sigma^{\vp}_t\otimes \id)(\ww^{\GG *})=
(\I\otimes \nabla_{\psi}^{it})\ww^{\GG *}(\I\otimes P^{it})$; here we have also used \cite[Definition 5.1, Remark 5.2]{LCQGDaele}. Next, since $\ww^{\whG}=\chi(\ww^{\GG *})$ and $\hat{\delta}^{-it}=\nabla_\psi^{it}P^{it}$ (\cite[Theorem 5.17]{LCQGDaele}), and using that $\hat{\tau}_{-t}(y) = P^{-it}yP^{it}$ for $y\in \Linfd$ see \cite[Proposition~8.23]{KustermansVaes}, we arrive at
\begin{equation}\label{eq20}
(\id\otimes \sigma^{\vp}_t )(\ww^{\whG})=
( \nabla_\psi^{it}\otimes \I)\ww^{\whG}(P^{it}\otimes \I)=
(\hat{\delta}^{-it}\otimes \I)(\hat{\tau}_{-t}\otimes\id)(\ww^{\whG}).
\end{equation}
Using this formula and $\hat{\tau}_t(\hat{\delta})=\hat{\delta}$ (\cite[Theorem 3.11]{LCQGDaele}) we calculate
\[\begin{split}
&\quad\;
\sigma^\vp_t(a)\wh{\lambda}(\omega)=
\sigma^\vp_t\bigl( a \sigma^\vp_{-t}(\wh{\lambda}(\omega))\bigr)=
\sigma^\vp_t\bigl(a
(\omega\otimes\id)\bigl( 
(\hat{\delta}^{it}\otimes \I)(\hat{\tau}_{t}\otimes\id)(\ww^{\whG})
\bigr)\bigr)\\
&=
\sigma^\vp_t\bigl(
a \wh{\lambda}( (\omega\hat{\delta}^{it})\circ\hat{\tau}_t)
\bigr)=
\sigma^\vp_t\bigl( \wh{\lambda} ( \Theta^l(a)_*( (\omega\hat{\delta}^{it})\circ\hat{\tau}_t))\bigr)\\
&=
\wh{\lambda}\bigl(
\bigl(\bigl( \Theta^l(a)_*((\omega\hat{\delta}^{it})\circ\hat{\tau}_t ) \bigr) \hat{\delta}^{-it}\bigr)\circ \hat{\tau}_{-t}
\bigr).
\end{split}\]
The above shows that $\sigma^\vp_t(a)\in \M^l(\A(\GG))$, and for $x\in \Linfd$ we get
\[\begin{split}
&\quad\;
\la \Theta^l(\sigma^\vp_t(a))(x),\omega \ra 
=
\la x, \Theta^l(\sigma^\vp_t(a))_*(\omega) \ra =
\la x , \bigl( 
\Theta^l(a)_*( (\omega\hat{\delta}^{it})\circ\hat{\tau}_t)
\hat{\delta}^{-it}\bigr)\circ \hat{\tau}_{-t}\ra \\
&=
\la \hat{\delta}^{-it} \hat{\tau}_{-t}(x) , \Theta^l(a)_*(
(\omega\hat{\delta}^{it})\circ\hat{\tau}_t)\ra =
\la \hat{\delta}^{it}\hat{\tau}_t\circ\Theta^l(a)\bigl(\hat{\delta}^{-it} \hat{\tau}_{-t}(x) \bigr),\omega\ra.
\end{split}\]

The final claim is shown in a similar way.  We have $(\sigma^\psi_t\otimes\tau_{-t})\circ\Delta=\Delta \circ\sigma^\psi_t$ (\cite[Proposition 8.23]{KustermansVaes}) and hence for $x\in \mf{N}_{\vp}$,
\[\begin{split}
&\quad\;
(\rho\circ \sigma^\psi_t\otimes \id)(\ww^{\GG *})\Lvp(x)=
\Lvp((\rho\circ\sigma^\psi_t\otimes\id)\Delta(x))=
\Lvp\bigl((\rho\otimes \tau_t)\Delta(\sigma^\psi_t(x))\bigr)\\
&=
\nu^{-\frac{t}{2}} P^{it} \Lvp( (\rho\otimes \id)\Delta(\sigma^\psi_t(x)))=
P^{it} (\rho\otimes \id)(\ww^{\GG *}) \nabla_\psi^{it}\Lvp(x),
\end{split}\]
consequently $(\sigma^\psi_t\otimes\id)(\ww^{\GG *})=(\I\otimes P^{it})\ww^{\GG *} (\I\otimes \nabla_\psi^{it})$ and so
\[
(\id\otimes \sigma^\psi_t)(\ww^{\whG})=(P^{it}\otimes\I)\ww^{\whG}(\nabla_\psi^{it}\otimes \I)=
(\hat{\tau}_t\otimes\id)(\ww^{\whG})(\hat{\delta}^{-it}\otimes\I).
\]
Calculating as before,
\[\begin{split}
&\quad\;
\sigma^\psi_t(a)\wh{\lambda}(\omega)=
\sigma^\psi_t\bigl(a
(\omega\otimes\id)\bigl( 
(\hat{\tau}_{-t}\otimes\id)(\ww^{\whG})
(\hat{\delta}^{it}\otimes \I)
\bigr)\bigr)\\
&=
\sigma^\psi_t\bigl(
a \wh{\lambda}( ( \hat{\delta}^{it}\omega)\circ\hat{\tau}_{-t})
\bigr)=
\wh{\lambda}\bigl(
\bigl(\bigl( \hat{\delta}^{-it}\,\Theta^l(a)_*((\hat{\delta}^{it}\omega)\circ\hat{\tau}_{-t} ) \bigr) \bigr)\circ \hat{\tau}_{t}
\bigr)
\end{split}\]
and so
\[\begin{split}
&\quad\;
\la \Theta^l(\sigma^\psi_t(a))(x),\omega \ra =
\la \hat{\tau}_{-t}\circ\Theta^l(a)\bigl( \hat{\tau}_{t}(x) \hat{\delta}^{-it} \bigr) \hat{\delta}^{it},\omega\ra \qquad(x\in\Linfd)
\end{split}\]
as desired. 
\end{proof}

In the next proposition we show that for any $a\in \M^l_{cb}(\A(\GG))$ the map $\Theta^l(a)$ is bounded on the Hilbert space level; the final claim should be compared with 
Proposition~\ref{prop12}.

\begin{proposition}\label{prop2}
Let $a\in \M^l_{cb}(\A(\GG))$. Then for $b\in \mf{N}_{\hvp}$ we have $\Theta^l(a)(b)\in \mf{N}_{\hvp}$, and the densely defined operator
\[
\LdG\supseteq \Lambda_{\hvp}(\mf{N}_{\hvp})\ni \Lambda_{\hvp} (b)\mapsto 
\Lambda_{\hvp} (\Theta^l(a)(b))\in \LdG
\]
is bounded. In fact, we have 
\[
\Lhvp(\Theta^l(a)(b))=S^{-1}(a)\Lhvp(b) = S(a^*)^* \Lhvp(b)\qquad(b\in\mf{N}_{\hvp}).
\]
\end{proposition}

Here and in the sequel we use the GNS implementation of map $\Theta^l(a)$ -- in the literature people have considered also different embeddings of (a subspace of) $\Linf$ into $\LdG$, c.f.~\cite[Section 2]{SkalskiViselterConvolution}.

In order to prove Proposition \ref{prop2} we start with a general lemma; recall \eqref{eq:scriptI} for the definition of $\mscr{J}$.

\begin{lemma}\label{lemma7}
Let $\omega\in \mscr{J}\subseteq \LL^1(\GG)$, let $a\in \Linf$, and let $b\in \Dom(\sigma^\vp_{-i/2}) \subseteq \Linf$. Then $a\omega b\in \mscr{J}$ with $\Lhvp(\lambda(a\omega b))=a J_{\vp} \sigma^{\vp}_{-i/2}(b )^* J_{\vp}\Lhvp(\lambda(\omega))$.
\end{lemma}

\begin{proof}
Take $x\in \mf{N}_{\vp}$. We have
\[\begin{split}
&\quad\;
\la x^*,a\omega b\ra =\la (a^* x b^*)^* , \omega \ra =
\ismaa{\Lvp( a^* x b^*) }{\Lhvp(\lambda(\omega))}\\
&=
\ismaa{ a^* J_{\vp} \sigma^\vp_{i/2}(b^*)^*J_{\vp}
\Lvp(x)}{\Lhvp(\lambda(\omega))}=
\ismaa{\Lvp(x) }{a J_{\vp} \sigma^{\vp}_{i/2}(b^*)J_{\vp} \Lhvp(\lambda(\omega))}
\end{split}\]
which proves the claim.
\end{proof}

\begin{proof}[Proof of Proposition \ref{prop2}]
Take $b=(\id\otimes \omega ) \ww^{\whG}$ for $\omega\in \LL^1(\GG)$ such that $\ov{\omega}\in \LL^1_{\sharp}(\GG)$ and $\ov{\omega}^{\sharp}\in \mscr{J}$; that such an $\omega$ exists follows from Lemma~\ref{lemma22}, for example.  Then, for any $\wh{\omega}\in \LL^1(\whG)$,

\[\begin{split}
&\quad\;
\la \Theta^l(a)(b) , \wh{\omega} \ra=
\la (\id\otimes \omega)\ww^{\whG}, \Theta^l(a)_*(\wh{\omega})\ra=\la \wh{\lambda}(\Theta^l(a)_*(\wh{\omega})) , \omega \ra \\
&=
\la a \wh{\lambda}(\wh{\omega}) , \omega \ra =
\la (\wh{\omega}\otimes \id )\ww^{\whG} , \omega a\ra =
\la (\id\otimes \omega a)\ww^{\whG} , \wh{\omega}\ra,
\end{split}\]
which shows that
\[
\Theta^l(a)(b)=(\id\otimes\omega a)\ww^{\whG}.
\]

Observe also that
\begin{equation}\label{eq3}
b=(\omega\otimes\id) (\ww^{\GG *})=
((\ov{\omega}\otimes \id)\ww^{\GG})^*=
(\ov{\omega}^{\sharp}\otimes \id)\ww^{\GG}=
\lambda(\ov{\omega}^{\sharp}),
\end{equation}
in particular $b\in \mf{N}_{\hvp}$. Now, by Proposition~\ref{prop12} we have $a^*\in D(S)$, hence for $y\in \Dom(S)$ it holds that $S(y)^* a^* \in \Dom(S)$ (as $S(y)^*\in \Dom(S)$ and $\Dom(S)$ is closed under multiplication \cite[Proposition 5.22]{KustermansVaes}), consequently
\[\begin{split}
&\quad\;
\la S(y),  \ov{\ov{\omega a}}\ra =\la a S(y) , \omega \ra=
\ov{\la S(y)^* a^* , \ov{\omega} \ra }=
\ov{\la S(y)^* a^* , \ov{\omega}^{\sharp \sharp} \ra }\\
&=
\ov{\la S( S(y)^* a^*) , \ov{ \ov{\omega}^{\sharp}}\ra }=
\la y S(a^*)^* , \ov{\omega}^{\sharp}\ra =
\la y S^{-1}(a) , \ov{\omega}^{\sharp}\ra =
\la y, S^{-1}(a)  \ov{\omega}^{\sharp}\ra 
\end{split}\]
hence $\ov{\omega a }\in \LL^1_{\sharp}(\GG)$ and $\ov{\omega a}^{\sharp} = S^{-1} (a) \ov{\omega}^{\sharp}$. Consequently
\[
\Theta^l(a)(b)= (\id\otimes \omega a)\ww^{\whG}=
(\ov{\omega a}^{\sharp}\otimes\id) \ww^{\GG}=
( S^{-1}(a) \ov{\omega}^{\sharp}\otimes \id)\ww^{\GG}=
\lambda( S^{-1}(a) \ov{\omega}^{\sharp}).
\]
This calculation, combined with Lemma~\ref{lemma7}, shows that $\Theta^l(a)(b) \in \mf{N}_{\hvp}$ with
\[
\Lhvp( \Theta^l(a) (b) ) = 
\Lhvp( \lambda( S^{-1}(a) \ov{\omega}^{\sharp}) )
= S^{-1}(a) \Lhvp( \lambda( \ov{\omega}^{\sharp}) )
= S^{-1}(a) \Lhvp(b).
\]

Now let $b\in \mf{N}_{\hvp}$ be arbitrary. By Lemma \ref{lemma22} the space
\[\begin{split}
&\quad\;
\{(\id\otimes \omega)\ww^{\whG}\,|\, \omega\in \LL^1(\GG)\colon \ov{\omega}\in \LL^1_{\sharp}(\GG),\, \ov{\omega}^{\sharp}\in \mscr{J}\}\\
&=
\{\lambda(\ov{\omega}^{\sharp})\,|\, \omega\in \LL^1(\GG)\colon \ov{\omega}\in \LL^1_{\sharp}(\GG),\, \ov{\omega}^{\sharp}\in \mscr{J}\}
\end{split}\]
is a $\ssot\times \|\cdot\|$ core for $\Lhvp$. Hence there is a net $(\omega_i)_{i\in I}$ of suitable functionals with $b=\ssot-\lim_{i\in I} \lambda(\omega_i)$ and $\Lhvp(b)=\lim_{i\in I} \Lhvp(\lambda(\omega_i))$. By \ssot-continuity of $\Theta^l(a)$ and the previous reasoning we obtain
\[
\Theta^l(a)(\lambda(\omega_i))\xrightarrow[i\in I]{\ssot} \Theta^l(a)(b)
\]
and
\[
\Lhvp( \Theta^l(a) (\lambda(\omega_i)))=
S^{-1}(a) \Lhvp(\lambda(\omega_i))\xrightarrow[i\in I]{}S^{-1}(a) \Lhvp(b).
\]
Thus, since $\Lhvp$ is $\ssot\times \|\cdot\|$ closed, 
\[
\Theta^l(a)(b)\in \Dom(\Lhvp)=\mf{N}_{\hvp}\quad\textnormal{and}\quad
\Lhvp( \Theta^l(a)(b))=
S^{-1}(a)\Lhvp(b)
\]
as claimed. 
\end{proof}

\subsection{Left versus right} 

We have defined AP of a locally compact quantum group $\GG$ in Definition~\ref{def1} using left CB multipliers. Let us verify that we would have obtained the same notion using 
right CB multipliers.

\begin{proposition}\label{prop14}
The following conditions are equivalent:
\begin{enumerate}
\item $\GG$ has AP, i.e.~there is a net in $\A(\GG)$ which converges to $\I$ in the weak$^*$-topology of $\M^l_{cb}(\A(\GG))$,
\item there is a net in $\widehat{\rho}(\Ljd)$ which converges to $\I$ in the weak$^*$-topology of $\M^r_{cb}(\wh{\rho}(\Ljd))$.
\end{enumerate}
\end{proposition}

\begin{proof}
Assume that $(a_i)_{i\in I}$ is a net in $\A(\GG)$ which converges to $\I$ in the weak$^*$-topology of $\M^l_{cb}(\A(\GG))$. Let $\wh{R}^{\sim}\colon \B(\LdG)\rightarrow\B(\LdG)$ be the extension of the unitary antipode of $\GG$ given by $\wh{R}^{\sim}(x)=J_{\vp}x^*J_\vp\,(x\in \B(\LdG))$.  For each $i\in I$ there is $\omega_i\in\Ljd$ with $a_i = \wh{\lambda}(\omega_i)$, and so by Lemma~\ref{lemma19}, if we define $b'_i = \wh{R}^{\sim}(a_i)$ then $b'_i = \wh\rho(\omega_i\circ\wh R) \in \wh{\rho}(\Ljd) \subseteq \M^r_{cb}(\wh\rho(\Ljd))$ with $\Theta^r(b'_i) = \wh{R}\circ \Theta^l(a_i)\circ \wh{R}\in\CB^\sigma(\Linfd)$.

Let us now argue that $b'_i\xrightarrow[i\in I]{\w^*}\I$: choose a functional $\theta\in Q^r(\wh\rho(\Ljd))$. We have to show that $\theta\circ\wh{R}^{\sim}\in Q^l(\A(\GG))$ (this makes sense by Lemma \ref{lemma19}). We can write $\theta=\lim_{j\in J} \alpha^r(\nu_j)$ for some $\nu_j\in \LL^1(\GG')$ (see Section \ref{sec:predual}). Then, as $\wh{R}^{\sim}\colon \M^l_{cb}(\Ljd)\rightarrow \M^r_{cb}(\Ljd)$ is an isometry (Lemma \ref{lemma18}) we obtain
\[\begin{split}
&\quad\;
\|\theta\circ\wh{R}^{\sim}-\alpha^l(\nu_j\circ \wh{R}^{\sim})\|=
\sup_{a\in \M^l_{cb}(\A(\GG))_1}|\langle \theta \circ \wh{R}^{\sim}-\alpha^l(\nu_j\circ\wh{R}^{\sim}), a \rangle|\\
&=
\sup_{a\in \M^l_{cb}(\A(\GG))_1}|\langle\theta - \alpha^r(\nu_j) , \wh{R}^{\sim}(a) \rangle|=
\sup_{b'\in \M^r_{cb}(\A(\GG))_1}|\langle\theta - \alpha^r(\nu_j) , b' \rangle|=
\|\theta-\alpha^r(\nu_j)\|
\xrightarrow[j\in J]{}0,
\end{split}\]
so $\theta\circ\wh{R}^{\sim}=\lim_{j\in J} \alpha^l(\nu_j\circ \wh{R}^{\sim})\in Q^l(\A(\GG)) $.
Now we can calculate
\[\begin{split}
&\quad\;
\langle  \I-b'_i,\theta\rangle=
\langle \I-\wh{R}^{\sim}(a_i),\theta\rangle=
\langle \I-a_i,\theta\circ\wh{R}^{\sim}\rangle
\xrightarrow[i\in I]{}0,
\end{split}\]
which shows that $(b'_i)_{i\in I}$ is a net giving us condition $(2)$. The converse implication is analogous.
\end{proof}

\section{Relation to other approximation properties} \label{section relation}

In this section we discuss the relation between AP and other approximation properties for locally compact quantum groups which have been studied in the literature. 
Recall the following definitions, compare \cite[Theorem~3.1]{BT_Amenability}, \cite[Section~5.2]{Brannan}.

\begin{definition} \label{defotherap}
Let $\GG$ be a locally compact quantum group.
\begin{itemize}
\item $\wh\GG$ is \emph{coamenable} if $\A(\GG)$ has a bounded approximate identity, 

\item $\GG$ is \emph{weakly amenable} if $\A(\GG)$ has a left approximate identity $(e_i)_{i\in I}$ which is bounded in $\M_{cb}^l(\A(\GG))$. In this case, the smallest number $C$ for which we can choose net $(e_i)_{i\in I}$ with $\|e_i\|_{cb}\leq C\,(i\in I)$ is the \emph{Cowling--Haagerup constant} of $\GG$, denoted $\Lambda_{cb}(\GG)$. 
\end{itemize}
\end{definition}

\begin{remark}
If $(e_i)_{i\in I}$ is a left approximate identity for $\A(\GG)$ then $(R(e_{i}))_{i\in I}$ is a right approximate identity, where $R$ is the unitary antipode on $\LL^\infty(\GG)$.  Indeed, let $\wh{R}$ be the unitary antipode on $\LL^{\infty}(\whG)$. As $R \wh\lambda = \wh\lambda \wh R_*$, each $R(e_i)$ is a member of $\A(\GG)$, and as $\wh R_*$ is anti-multiplicative, it follows that $(R(e_i))_{i\in I}$ is indeed a right approximate identity.
Thus it does not matter if we work in $\M_{cb}^l(\A(\GG))$ or in $\M_{cb}^r(\A(\GG))$, see also Lemma~\ref{lemma18}.
\end{remark}

We shall show that if $\GG$ has the AP, and the approximating net can be chosen in an appropriately bounded way, then $\GG$ will enjoy one of the stronger properties 
in Definition \ref{defotherap}. 

Let us first record some general results. For a proof of the following fact see for example \cite[Section 3]{Crann}. 

\begin{lemma}\label{lemma17}
For any locally compact quantum group $\GG$, the linear span of $\{ ab \,|\, a,b\in \A(\GG) \}$ is dense in $\A(\GG)$.
\end{lemma}

Next we recall a standard result in Banach algebra theory which follows from the Hahn--Banach Theorem and the fact that convex sets have the same norm and weak closures; see for 
example \cite[Theorem~5.1.2(e)]{PalmerBook}.

\begin{proposition}\label{prop:weak_bai_gives_bai}
Let $A$ be a Banach algebra which has a weak bounded left approximate identity, meaning that there is a bounded net $(e_i)_{i\in I}$ in $A$ such that $\mu(e_i a-a) \xrightarrow[i\in I]{} 0$ for each $a\in A, \mu\in A^*$.  Then $A$ has a bounded left approximate identity (of the same bound).
\end{proposition}

Proposition \ref{prop:weak_bai_gives_bai} does not say that the weak blai (bounded left approximate identity) is itself a blai, but rather that having a weak blai implies there is a possibly different net forming a blai. 

Let us denote by $\A_{cb}^l(\GG)$ the closure of $\A(\GG)$ inside $\M_{cb}^l(\A(\GG))$ and use Proposition \ref{prop:weak_bai_gives_bai} 
to obtain the following result, compare \cite[Proposition~1]{Forrest_CB_Ideals}.

\begin{lemma}\label{lem:Acb_bai_weak_amen}
The inclusion map $\A(\GG)\rightarrow \A_{cb}^l(\GG)$ is an injective contraction. The locally compact quantum group $\GG$ is weakly amenable with Cowling--Haagerup constant at 
most $K$ if and only if $\A_{cb}^l(\GG)$ has a bounded left approximate identity of bound at most $K$.
\end{lemma}

\begin{proof}
Let $\omega\in \LL^1(\wh\GG)$ be an element of norm one. Then the map
\[
\LL^1(\wh\GG) \rightarrow \LL^1(\wh\GG) \wh\otimes \LL^1(\wh\GG)\colon \nu \mapsto \omega\otimes\nu
\]
is a complete isometry, and so applying $\wh\Delta_*$ shows that $\LL^1(\wh\GG) \rightarrow \LL^1(\wh\GG)\colon \nu\mapsto\omega\star\nu$ is a complete contraction.  Thus $\|\wh{\lambda}(\omega)\|_{cb} \leq \|\wh{\lambda}(\omega)\|$, here and below writing $\|\cdot\|_{cb}$ for the norm on $\A_{cb}^l(\GG)$ and $\|\cdot\|$ for the norm on $\A(\GG)$.  The diagram \eqref{eq16} shows in particular that $\A(\GG) \rightarrow \M_{cb}^l(\A(\GG))$ is injective, and hence $\A(\GG)\rightarrow \A_{cb}^l(\GG)$ is injective.

Now suppose that $\GG$ is weakly amenable and that $(e_i)_{i\in I}$ is a left approximate identity for $\A(\GG)$ with $\|e_i\|_{cb} \leq K$ for each $i$. Then $(e_i)_{i\in I}$ is a bounded net in $\A_{cb}^l(\GG)$. For $x\in \A_{cb}^l(\GG)$ and $\epsilon>0$ there is $a\in \A(\GG)$ with $\|x-a\|_{cb} <\epsilon$, and there is $i_0$ so that $\|e_i a - a\| < \epsilon$ when $i\geq i_0$.  Thus
\[
\| e_i x - x \|_{cb} \leq \| e_i x - e_i a\|_{cb} + \|e_i a-a\|_{cb} + \|a-x\|_{cb}
< K\epsilon + \epsilon + \epsilon \qquad (i\geq i_0). \]
It follows that $e_i x\xrightarrow[i\in I]{} x$ in $\A_{cb}^l(\GG)$. Consequently $(e_i)_{i\in I}$ is a blai for $\A_{cb}^l(\GG)$ of bound $\leq K$.

Conversely, suppose that $\A_{cb}^l(\GG)$ has a bounded left approximate identity of bound $K$, say $(f_i)_{i\in I}$.  For $(i,n) \in I\times\mathbb N$ pick $e_{i,n}\in \A(\GG)$ with $\|e_{i,n}\|_{cb}=\|f_{i}\|_{cb}$ and $\|e_{i,n} - f_i\|_{cb} < \tfrac{1}{n}$.  For $x\in \A_{cb}^l(\GG)$ and $\epsilon>0$ there is $i_0$ so that $\| f_i x - x \|_{cb} < \epsilon$ for $i\geq i_0$.  With $n>\tfrac{1}{\eps}$,
\[
\|e_{i,n} x - x\|_{cb} \leq \| e_{i,n} x - f_i x\|_{cb} + \| f_i x- x\|_{cb}
< \eps \|x\|_{cb} + \eps,
\]
and so $e_{i,n} x\xrightarrow[(i,n)\in I\times \NN]{} x$.  We conclude that we may assume that $(f_i)_{i\in I}$ was actually a net in $\A(\GG)$ and $\|f_i\|_{cb}\leq K$ for each $i$.  It remains to show that $(f_i)_{i\in I}$ is a left approximate identity for $\A(\GG)$.  Given $a\in \A(\GG)$ and $\epsilon>0$, by Lemma~\ref{lemma17}, 
we can find elements $ a_k, b_k \in \A(\GG) $ for $ k = 1, \dots, n $ for some $n$ such that $a_0 = \sum_{k=1}^n a_k b_k \in \A(\GG)$ is within $\epsilon$ distance of $a$. Then
\begin{align*}
\| f_i a - a \| &\leq \| f_i a - f_i a_0 \| + \| f_i a_0 - a_0 \| + \| a_0 - a \| \\
&\leq K\| a-a_0\| + \sum_{k=1}^n \| f_i a_k b_k - a_k b_k\| + \|a_0-a\| \\
&\leq (K+1)\epsilon + \sum_{k=1}^n \|f_i a_k-a_k\|_{cb} \|b_k\|.
\end{align*}
Here we used that for $x\in \A_{cb}^l(\GG)$ and $y\in \A(\GG)$ we have $\|xy\| \leq \|x\|_{cb} \|y\|$. As $(f_i)_{i\in I}$ is a left approximate identity in $\A_{cb}^l(\GG)$, if $i$ is sufficiently large then $\|f_i a_k-a_k\|_{cb} \|b_k\| < \tfrac{\epsilon}{n}$ for each $k$, and so $\| f_i a - a \| \leq (K+2)\epsilon$.  Hence $(f_i)_{i\in I}$ is a left approximate identity for $\A(\GG)$.
\end{proof}

The following is an improvement on \cite[Theorem~1.13]{HaagerupKraus} in the classical situation; Haagerup and Kraus only consider the case where 
each element of the approximating net comes from a state. 

\begin{proposition}\label{prop3}
Let $\GG$ be a locally compact quantum group.  The following are equivalent:
\begin{enumerate}
\item $\GG$ has the approximation property, and we can choose the approximating net $(e_i)_{i\in I}$ in $\A(\GG)$ to be bounded;
\item $\wh\GG$ is coamenable.
\end{enumerate} 
\end{proposition}

\begin{proof}
Suppose that $\GG$ has the approximation property with approximating net $(e_i)_{i\in I}$ which is bounded in $\A(\GG)$.  By definition, $e_i\xrightarrow[i\in I]{} \I$ weak$^*$ in $\M_{cb}^l(\A(\GG))$.  By Proposition \ref{prop17}, we consider elements of the form $\Omega_{\wh x,\wh\omega,f}\in Q^l(\A(\GG))$.  Here we will just consider $\wh x\in \LL^\infty(\wh\GG)$ and $\wh\omega\in \LL^1(\wh\GG)$, with $f\in \LL^1(\GG)$ a state. According to equation \eqref{eq17}, we get 
\[
\lim_{i\in I} \langle \Theta^l(e_i\star f)(\wh x), \wh\omega \rangle = 
\lim_{i\in I} \la e_i,\Omega_{\wh x,\wh\omega,f}\ra=
\la \I,\Omega_{\wh x,\wh \omega,f}\ra=
\langle \Theta^l(\I\star f)(\wh x), \wh\omega \rangle = \langle \wh x, \wh\omega \rangle,
\]
using $\I\star f = \I$ and $\Theta^l(\I) = \id$ in the last step.
For each $i\in I$ let $e_i$ be associated to $\wh{\omega}_i\in \LL^1(\wh\GG)$, so that $e_i = \wh{\lambda}(\wh\omega_i)$. Then
\begin{align*}
e_i\star f
&= (f\otimes\id)\Delta(e_i)
= (f\otimes\id)\Delta( (\id\otimes\wh\omega_i)(\ww^{\GG *}) )
= (f\otimes\id\otimes\wh\omega_i)(\ww_{23}^{\GG *} \ww_{13}^{\GG *}) \\
&= (\id\otimes\wh\omega_i)\bigl(\ww^{\GG *} (\I \otimes (f\otimes\id)(\ww^{\GG *})) \bigr)
= \wh\lambda(\wh\omega'_i),
\end{align*}
where $\wh\omega'_i = (f\otimes\id)(\ww^{\GG *}) \wh\omega_i \in \LL^1(\wh\GG)$.  Notice that $\|\wh\omega'_i\| \leq \|f\| \|\wh\omega_i\| = \|\wh\omega_i\|$.

As computed in Lemma \ref{lemma9}, it follows that $\Theta^l(e_i\star f)(\wh x) = (\wh\omega'_i\otimes\id)\wh\Delta(\wh x)$, so we see that
\[
\lim_{i\in I} \langle \wh x, \wh\omega'_i \star \wh\omega \rangle
= \lim_{i\in I} \langle (\wh\omega'_i\otimes\id)\wh\Delta(\wh x), \wh\omega \rangle
= \langle \wh x, \wh\omega \rangle 
\qquad (\wh x\in \LL^\infty(\wh\GG), \wh\omega\in \LL^1(\wh\GG)). \]
Thus $(\wh\omega'_i)_{i\in I}$ is a weak bounded left approximate identity for $\LL^1(\wh\GG)$.
By Proposition~\ref{prop:weak_bai_gives_bai}, we obtain that $\LL^1(\wh\GG)$ has a bounded left approximate identity. 
This already implies that $\wh\GG$ is coamenable, see \cite[Theorem~3.1]{BT_Amenability}.

Conversely, suppose that $\wh\GG$ is coamenable, so that in particular, $\A(\GG)$ has a bounded left approximate identity $(e_i)_{i\in I}$.  For each $i$, let $\wh\omega_i\in \Ljd$ with $\wh\lambda(\wh\omega_i)=e_i$, so $(\wh\omega_i)_{i\in I}$ is a bounded net, and $\wh\omega_i\star\wh\omega \xrightarrow[i\in I]{} \wh\omega$ in norm, for each $\wh\omega\in\Ljd$.  Thus, given $x\in\CZGD$ and $\wh\omega\in\Ljd$, we find that
\[ \langle \Theta^l(e_i)(x), \wh\omega \rangle
= \langle x, \wh\omega_i \star \wh\omega \rangle
\xrightarrow[i\in I]{} \langle x, \wh\omega \rangle. \]
Let $\msf H$ be a separable Hilbert space.
As the net $(\Theta^l(e_i))_{i\in I}$ is bounded, given $x\in \CZGD\otimes \mc{K}(\msf{H})$ and $\omega\in \Ljd\wh{\otimes}\B(\msf{H})_*$, we find that
\[ \lim_{i\in I} \Omega_{x,\omega}(e_i)
= \lim_{i\in I} \langle (\Theta^l(e_i)\otimes\id) x , \omega \ra
= \la x, \omega \ra = \Omega_{x,\omega}(\I). \]
By Proposition~\ref{prop15}, it follows that $a_i \xrightarrow[i\in I]{} \I$ weak$^*$ in $\M_{cb}^l(\A(\GG))$ as required.
\end{proof}

The following is \cite[Theorem~1.12]{HaagerupKraus} in the classical situation.

\begin{proposition}\label{prop7}
Let $\GG$ be a locally compact quantum group.  The following are equivalent:
\begin{enumerate}
\item $\GG$ has the approximation property, and we can choose the approximating net $(e_i)_{i\in I}$ in $\A(\GG)$ to be bounded with respect to $\|\cdot\|_{cb}$;
\item $\GG$ is weakly amenable.
\end{enumerate}
Furthermore, in this case, the Cowling--Haagerup constant is at most the bound of $(e_i)_{\in I}$, while when $\GG$ is weakly amenable, we can choose each $e_i$ with $\|e_i\|_{cb} \leq \Lambda_{cb}(\GG)$.
\end{proposition}

\begin{proof}
We proceed as in the previous proof, starting with a net $(e_i)_{i\in I}$ in $\A(\GG)$, but now only with $\|e_i\|_{cb} \leq K$ for each $i$. We set $e_i = \wh\lambda(\wh\omega_i)$ 
and $\wh\omega'_i = (f\otimes\id)(\ww^{\GG *}) \wh\omega_i$, where $f$ is some fixed state and we are considering the natural (left) $ \LL^\infty(\wh\GG) $-module structure 
on $ \LL^1(\wh\GG) $. Then $(\wh\omega'_i)_{i\in I}$ is a weak left approximate identity for $\LL^1(\wh\GG)$.  Furthermore, $\| \wh\lambda(\wh\omega'_i) \|_{cb} = \|e_i\star f\|_{cb} \leq \|e_i\|_{cb} \|f\| \leq K$ by Proposition~\ref{prop16}.

For each $i$ let $f_i = \wh\lambda(\wh\omega'_i) \in \A(\GG)$. Let $\theta\colon \A(\GG)\rightarrow \A_{cb}^l(\GG)$ be the inclusion map, and consider the 
adjoint, $\theta^*\colon \A_{cb}^l(\GG)^*\rightarrow \A(\GG)^*$.  For $\mu\in \A_{cb}^l(\GG)^*$ and $a\in \A(\GG)$ we see that
\[ \lim_{i\in I} \langle \mu, \theta(f_i) \theta(a) \rangle
= \lim_{i\in I} \langle \theta^*(\mu), f_i a \rangle
= \langle \theta^*(\mu), a \rangle = \langle \mu, \theta(a) \rangle. \]
As $\theta$ has dense range, and $(\theta(f_i))_{i\in I}$ is bounded in $\A_{cb}^l(\GG)$, it follows that $(\theta(f_i))_{i\in I}$ is a weak bounded left approximate identity.  By Proposition~\ref{prop:weak_bai_gives_bai}, $\A_{cb}^l(\GG)$ has a bounded left approximate identity, and so Lemma~\ref{lem:Acb_bai_weak_amen} shows that $\GG$ is weakly amenable, 
with Cowling--Haagerup constant at most $K$.

Conversely, when $\GG$ is weakly amenable, we can choose a left approximate identity $(e_i)_{i\in I}$ for $\A(\GG)$ with $\|e_i\|_{cb} \leq \Lambda_{cb}(\GG)$ for each $i$.  The argument in the proof of Proposition~\ref{prop3} still works in this situation, because $(\Theta^l(e_i))_{i\in I}$ is a net bounded with respect to the CB norm.
\end{proof}

\section{Discrete quantum groups and operator algebraic approximation properties} \label{section discrete}

If the quantum group being studied is discrete, we can obtain better results.  In Proposition~\ref{prop1} we will show that the net exhibiting AP can be chosen to have additional properties.  We also relate AP to approximation properties of the associated operator algebras, in both the locally compact case, Proposition~\ref{prop5}, and the discrete case, Proposition~\ref{prop4}.

For the rest of this section $\GGamma$ stands for an arbitary discrete quantum group.  Then $\wh\GGamma$ is a compact quantum group, and we freely use the additional theory available in the compact case.  We follow \cite{NeshveyevTuset} as well as \cite{MaesVanDaele, Timmermann}, being aware that we use the ``left'' convention for multiplicative unitaries and representations.

Every irreducible unitary representation of $\wh\GGamma$ is finite-dimensional, and we denote by $\Irr(\wh{\GGamma})$ the collection of equivalence classes of irreducibles. 
We write $ \overline{\alpha} $ for the conjugate of $ \alpha \in \Irr(\wh{\GGamma}) $. For each $\alpha \in \Irr(\wh{\GGamma})$ let $U^\alpha \in \mrm{C}(\wh\GGamma) \otimes \B(\msf{H}_\alpha)$ be a unitary representation in the class of $\alpha$. 
With respect to an orthonormal basis of $\msf{H}_\alpha$ we regard $U^\alpha$ as a matrix $[U^\alpha_{i,j}]_{1\leq i,j\leq \dim(\alpha)}$.  The \emph{matrix coefficients} $U^\alpha_{i,j}$ span a dense Hopf $\star$-algebra $\Pol(\wh\GGamma) \subseteq \mrm{C}(\wh\GGamma)$.  We denote by $h$ the Haar integral on $\mrm{C}(\wh\GGamma)$ and $\LL^\infty(\wh\GGamma)$, and let $\Lambda_h\colon \mrm{C}(\wh\GGamma) \rightarrow \LL^2(\wh\GGamma)$ be the GNS map for $h$.  As as $\LL^\infty(\wh\GGamma)$ is in standard position on $\LL^2(\wh\GGamma)$, the set $\{ \omega_{\Lambda_h(a), \Lambda_h(b)} \,|\, a,b\in\Pol(\wh\GGamma) \}$ is dense in $\LL^1(\wh\GGamma)$.  As each member of $\Pol(\wh\GGamma)$ is analytic for the modular automorphism group of $h$, this agrees in fact with the set $\{ \omega_{\Lambda_h(a), \Lambda_h(\I)} \,|\, a \in \Pol(\wh\GGamma) \}$. Notice that $\omega_{\Lambda_h(a), \Lambda_h(\I)}$ is the functional $h(a^*\cdot)$.

For each $\alpha\in\Irr(\wh{\GGamma})$ there is a positive invertible operator $\uprho_\alpha$ related to the possible non-traciality of the Haar integral $h$, see \cite[Section 1.7]{NeshveyevTuset}.  We choose and fix a basis of $\msf{H}_\alpha$ such that $\uprho_\alpha$ is diagonal. We define the \emph{Woronowicz characters} $\{ f_z \,|\, z\in\mathbb C \}$ by the relation $(f_z \otimes\id)(U^\alpha) = \uprho_\alpha^z,$ valid for each $\alpha$.  The modular automorphism group is then implemented as
\[ \sigma_z^h(a) = f_{iz} \star a \star f_{iz} \quad (a\in\Pol(\wh\GGamma), z\in\mathbb C), \]
or equivalently, $(\sigma^h_z\otimes\id)(U^\alpha) = (\I\otimes \uprho_\alpha^{iz}) U^\alpha (\I\otimes \uprho_\alpha^{iz})$.  Similarly, the scaling group is implemented as
\[ \tau_z(a) = f_{-iz} \star a \star f_{iz} \quad (a\in\Pol(\wh\GGamma), z\in\mathbb C), \]
or equivalently, $(\tau_z\otimes\id)(U^\alpha) = (\I\otimes \uprho_\alpha^{iz}) U^\alpha (\I\otimes \uprho_\alpha^{-iz})$.  As we assume that $\uprho_\alpha$ is diagonal, say with entries $\uprho_{\alpha, i}\,(1\leq i\leq\dim(\alpha))$, we get 
\begin{equation}
\tau_z(U^\alpha_{i,j}) = (\uprho_{\alpha, i})^{iz} (\uprho_{\alpha,j})^{-iz} U^\alpha_{i,j}.
\label{eq:scaling_gp_cmpt}
\end{equation}

The algebra $\mrm{c}_0(\GGamma)$ is isomorphic to the direct sum of full matrix algebras $\M_{\dim(\alpha)}$ indexed by $\alpha\in\Irr(\wh\GGamma)$, and $\ell^\infty(\GGamma)$ is isomorphic to the direct product of these matrix algebras.  Given $a\in \ell^\infty(\GGamma)$ we write $a = (a^\alpha)_{\alpha\in\Irr(\wh\GGamma)}$ where $a^\alpha\in \M_{\dim(\alpha)}$, and similarly for $\mrm{c}_0(\GGamma)$.  With respect to this isomorphism,
\begin{equation}
\ww^{\wh\GGamma} = \bigoplus_{\alpha\in\Irr(\wh\GGamma)} \sum_{i,j=1}^{\dim(\alpha)} U^\alpha_{i,j} \otimes e^\alpha_{i,j} \in \M\big(\mrm{C}(\wh\GGamma) \otimes \mrm{c}_0(\GGamma) \big),
\label{eq:hatW_cmpt}
\end{equation}
where $\{e^{\alpha}_{i,j}\}_{i,j=1}^{\dim(\alpha)}$ are the matrix units of the matrix algebra $\M_{\dim(\alpha)} \subseteq \mrm{c}_0(\GGamma)$.

We start with a result expressing the action of $\Theta^l(a) \in \CB^\sigma(\LL^\infty(\wh\GGamma))$, for $a\in \M^l_{cb}(\A(\GGamma))$, on matrix elements.

\begin{lemma}\label{lemma4}
For any $a=(a^\alpha)_{\alpha\in \Irr(\wh\GGamma)}\in \M^l_{cb}(\A(\GGamma)) \subseteq \ell^\infty(\GGamma)$ with $a^\alpha=[a^\alpha_{i,j}]_{i,j=1}^{\dim(\alpha)}$ we have
\[
\Theta^l(a) (U^{\alpha}_{i,j})=
\sum_{k=1}^{\dim(\alpha)} a^\alpha_{i,k} U^{\alpha}_{k,j}\qquad(\alpha\in \Irr(\wh\GGamma), 1\le i,j\le \dim(\alpha)).
\]
\end{lemma}
\begin{proof}
Let $x\in \Pol(\wh\GGamma)$ and set $\omega = h(x\cdot) \in \LL^1(\wh\GGamma)$.
Recall that $a\wh{\lambda}(\omega)=\wh{\lambda}(\Theta^l(a)_*(\omega))$, equivalently, $a (\omega\otimes\id)(\ww^{\wh\GGamma}) = (\Theta^l(a)_*(\omega)\otimes\id)(\ww^{\wh\GGamma})$. 
Using the expression for $\ww^{\wh\GGamma}$ from \eqref{eq:hatW_cmpt}, it follows that
\begin{align*}
\sum_{\alpha\in\Irr(\wh\GGamma)} \sum_{i,j,k=1}^{\dim(\alpha)} \la U^\alpha_{k,j}, \omega \ra a^\alpha_{i,k} e^{\alpha}_{i,j} &= 
\sum_{\alpha\in\Irr(\wh\GGamma)} \sum_{k,j=1}^{\dim(\alpha)} \la U^\alpha_{k,j}, \omega \ra a e^{\alpha}_{k,j} \\ 
&= a (\omega\otimes\id)(\ww^{\wh\GGamma}) \\
&= (\Theta^l(a)_*(\omega)\otimes\id)(\ww^{\wh\GGamma})
= \sum_{\alpha\in\Irr(\wh\GGamma)} \sum_{i,j=1}^{\dim(\alpha)} \la \Theta^l(a)(U^\alpha_{i,j}), \omega \ra e^{\alpha}_{i,j}.
\end{align*}
By density, this holds for all $\omega$, and so we conclude $\Theta^l(a)(U^\alpha_{i,j}) = \sum_{k=1}^{\dim(\alpha)}a^\alpha_{i,k} U^\alpha_{k,j} $, as claimed.
\end{proof}

\begin{remark}\label{rem:cent_unit_pres_means}
Later, see Proposition~\ref{prop1}, we shall consider $a\in \M^l_{cb}(\A(\GGamma))$ with $\Theta^l(a)$ unit preserving.  Let $e$ denote the trivial representation of $\wh\GGamma$, so $\dim(e)=1$ and $U^e = \I\otimes 1$.  From Lemma~\ref{lemma4}, for such an $a$, we see that $\I = \Theta^l(a)(\I) = a^e_{1,1} \I$ and so $a^e_{1,1} = 1$.  Further, as the Haar integral $h$ annihilates all coefficients of all irreducible representations except $e$, and as $\Pol(\wh\GGamma)$ is dense in $\mrm{C}(\wh\GGamma)$, it follows that $h \circ \Theta^l(a) = h$.
\end{remark}

For discrete quantum groups we will also look at a central variation of AP. We denote by $\mrm{c}_{00}(\GGamma) \subseteq \mrm{c}_0(\GGamma)$ the dense subspace of elements $x = (x^\alpha)_{\alpha\in\Irr(\wh\GGamma)}$ such that $x^\alpha=0$ for all but finitely many $\alpha$. From the description of $\ww^{\wh\GGamma}$ in \eqref{eq:hatW_cmpt} it is clear that we have $\mrm{c}_{00}(\GGamma) \subseteq \A(\GGamma)$. Notice that the centre of $\ell^\infty(\GGamma)$, denoted $\mc{Z}(\ell^{\infty}(\GGamma))$, consists of families of matrices $x=(x^\alpha)_{\alpha\in\Irr(\wh\GGamma)}$ such that each $x^\alpha \in \M_{\dim(\alpha)}$ is a scalar multiple of the identity.

\begin{definition} \label{defcentralAP}
We say that a discrete quantum group $\GGamma$ has the \emph{central approximation property (central AP)} if there is a net $(a_i)_{i\in I}$ in $\mrm{c}_{00}(\GGamma)\cap \mc{Z}(\ell^{\infty}(\GGamma))$ which converges to $\I$ in the weak$^*$-topology of $\M^l_{cb}(\A(\GGamma))$.
\end{definition}

It is clear from the definitions that central AP implies AP. At first sight, it might seem more natural to use $\A(\GGamma)$ instead of $\mrm{c}_{00}(\GGamma)$ in 
Definition \ref{defcentralAP}, and indeed, this alternative definition (for other approximation properties) is taken in \cite[Definition~7.1]{Brannan}. However, in terms 
of applications, and also from the point of view of representation categories, working with $\mrm{c}_{00}(\GGamma)$ 
is in fact the most appropriate choice. We do not know if these two properties are equivalent (a naive approximation typically produces an element which is not central). Let us point out that the examples considered in \cite{Brannan} actually do end up working with $\mrm{c}_{00}(\GGamma)$.
We will discuss the relation of central AP to properties of the representation category $\Rep(\wh\GGamma)$ in Section~\ref{section categorical}. 

\begin{remark}\label{rem:fin_sup_finrank}
We shall say that $a\in \M^l_{cb}(\A(\GGamma))$ is \emph{finitely supported} if $a\in \mrm{c}_{00}(\GGamma)$.  Of course, we have $\mrm{c}_{00}(\GGamma) \subseteq \A(\GGamma) \subseteq \M^l_{cb}(\A(\GGamma))$.  For $a\in\mrm{c}_{00}(\GGamma)$, it follows from Lemma~\ref{lemma4} that $\Theta^l(a)(U^\alpha_{i,j}) = 0$ for all but finitely many $\alpha$.  Hence $\Theta^l(a)$ restricted to $\Pol(\wh\GGamma)$ is a finite-rank map, and so by continuity, $\Theta^l(a)$ restricted to $\mrm{C}(\wh\GGamma)$ is finite-rank, and hence by normality, $\Theta^l(a)$ is also finite-rank.
\end{remark}

In the next result we show that whenever a discrete quantum group has AP, then this is implemented by a net of elements with convenient properties.

\begin{proposition}\label{prop1}
Assume that $\GGamma$ is a discrete quantum group with AP. Then there is a net $(a_i)_{i\in I}$ of elements in $\mrm{c}_{00}(\GGamma)$ such that 
\begin{itemize} 
\item $a_i\xrightarrow[i\in I]{}\I$ in $(\M^l_{cb}(\A(\GGamma)) ,w^*)$,
\item every $a_i$ is invariant under the scaling group of $\GGamma$ and modular automorphism groups of the left/right Haar integrals,
\item every $\Theta^l(a_i)$ is star and unit preserving.
\end{itemize}
If $\GGamma$ has central AP then we can additionally assume that $a_i\in \mrm{c}_{00}(\GGamma)\cap \mc{Z}(\ell^{\infty}(\GGamma))$.
\end{proposition}

For the proof of Proposition \ref{prop1} we shall need two lemmas. For any operator space $X$, we denote by $\kappa\colon X^*\wh\otimes X\rightarrow \CC$ the canonical completely contractive map $\omega\otimes x\mapsto \la \omega,x\ra $.

\begin{lemma}\label{lemma16}
Let $\msf{H}$ be a Hilbert space, let $\M,\N$ be von Neumann algebras, and let
\[
v\in \bigl(\M\bar{\otimes}\LL^{\infty}(\wh\GGamma)\bar{\otimes}\N\bar{\otimes}\B(\msf{H})\bigr)\wh\otimes \bigl( \M_*\wh\otimes \LL^1(\wh\GGamma)\wh\otimes\N_*\wh\otimes \B(\msf{H})_*\bigr).
\]
The bounded linear functional
\[
\Omega_v\colon \M^l_{cb}(\A(\GGamma))\ni a \mapsto 
\kappa\big( ((\id\otimes \Theta^l(a)\otimes \id^{\otimes 2}) \otimes \id^{\otimes 4})v \big)\in \CC
\]
belongs to $Q^l(\A(\GGamma))$, and we have $\|\Omega_v\|\le \|v\|$.
\end{lemma}

\begin{proof}
Since $Q^l(\A(\GGamma))$ is closed in $\M^l_{cb}(\A(\GGamma))^*$, it is enough to consider $v=x\otimes (\omega_1\otimes \omega_2\otimes \omega_3\otimes \omega_4)$  for $x\in \M\bar\otimes \LL^{\infty}(\wh\GGamma)\bar\otimes \N\bar\otimes \B(\msf{H}), \omega_1\in \M_*,\omega_2\in \LL^1(\wh\GGamma),\omega_3\in \N_*,\omega_4\in  \B(\msf{H})_*$. Let $\eps\in \ell^1(\GGamma)$ be the counit of $ \ell^\infty(\GGamma)$. Define $y=(\omega_1\otimes \id\otimes\omega_3\otimes \id)x\in \LL^{\infty}(\wh\GGamma)\bar\otimes \B(\msf{H})$. For $a\in \M^l_{cb}(\A(\GGamma))$ we have
\[\begin{split}
&\quad\;
\la \Omega_v,a\ra=
\kappa\bigl(
(\id\otimes \Theta^l(a)\otimes\id^{\otimes 2})x \otimes (\omega_1\otimes \omega_2\otimes \omega_3\otimes \omega_4)\bigr)\\
&=
\la (\id\otimes \Theta^l(a)\otimes \id^{\otimes 2})x,\omega_1\otimes \omega_2\otimes \omega_3\otimes \omega_4\ra
=\la 
(\Theta^l(a)\otimes \id)y ,\omega_2\otimes \omega_4\ra \\
&=
\la (\Theta^l(a\star\eps)\otimes \id)y,\omega_2\otimes \omega_4\ra=
\la a,\Omega_{y,\omega_2\otimes \omega_4,\eps}\ra
\end{split}\]
by the definition of $\Omega_{y,\omega_2\otimes \omega_4,\eps}$ (see \eqref{eq17}) hence $\Omega_v=\Omega_{y,\omega_2\otimes \omega_4,\eps}$ and the claim follows from Proposition \ref{prop17}.
\end{proof}

In the following, recall that a \emph{mean} on $\mathbb R$ is a state $m_{\mathbb R}$ on $\LL^\infty(\mathbb R)$ which is invariant under the translation action of $\mathbb R$. Such a state exists as the group $\mathbb R$ is abelian and hence amenable.

\begin{lemma}\label{lemma8}
Let $m_{\RR}$ be a mean on $\RR$, let $\msf{H}$ be a Hilbert space, and let $x\in \mrm{C}(\wh\GGamma)\otimes\mc{K}(\msf{H}), \rho\in \LL^1(\wh\GGamma)\wh{\otimes} \B(\msf{H})_*$.
Then 
\[
\Omega^{\tau}_{x,\rho}\colon \M^l_{cb}(\A(\GGamma))\ni a \mapsto
m_{\RR}\bigl(t\mapsto 
\la (\Theta^l(a)\otimes \id) (\hat\tau_{-t}\otimes\id)(x) , \rho\circ (\hat\tau_t\otimes\id)\ra \bigr)\in \CC
\] 
defines a bounded functional on $\M^l_{cb}(\A(\GGamma))$. 
Furthermore, $\Omega^{\tau}_{x,\rho}\in Q^l(\A(\GGamma))$ and $\|\Omega^{\tau}_{x,\rho}\|\le \|x\|\|\rho\|$.
\end{lemma}

\begin{proof}
As $t\mapsto \rho\circ (\hat\tau_t\otimes\id)$ is norm continuous, it follows that the function $t \mapsto \la (\Theta^l(a)\otimes \id) (\hat\tau_{-t}\otimes\id)(x) , \rho\circ (\hat\tau_t\otimes\id)\ra$ is continuous, and bounded, and so we can indeed apply $m_{\mathbb R}$ to it.  The only nontrivial claim is that $\Omega^{\tau}_{x,\rho}$ is a normal functional. 

In order to verify this take $t\in \RR$. Using the canonical complete contraction $\kappa$ we can write
\[\begin{split}
\la (\Theta^l(a)\otimes \id) (\hat\tau_{-t}\otimes\id)(x) , \rho\circ (\hat\tau_t\otimes\id)\ra =
\kappa (\Theta^l(a)\otimes\id\otimes\id\otimes\id)\bigl(
(\hat\tau_{-t}\otimes\id)(x) \otimes  \rho\circ (\hat\tau_t\otimes\id)\bigr).
\end{split}\]
Let $x = U^\alpha_{i,j}\in\Pol(\wh\GGamma), y = U^{\beta}_{k,l}\in\Pol(\wh\GGamma)$ and set $\rho = h(y\cdot) \in \LL^1(\wh\GGamma)$.  As $h$ is $(\hat\tau_t)_{t\in\RR}$ invariant, it follows that $(\rho\circ\hat\tau_t)(z) = h(y \hat\tau_t(z))
= h(\hat\tau_t( \hat\tau_{-t}(y) z)) = h(\hat\tau_{-t}(y) z)$ for each $z\in \mrm{C}(\wh\GGamma)$, and so $\rho\circ\hat\tau_t = h( \hat\tau_{-t}(y) \cdot )$.  From \eqref{eq:scaling_gp_cmpt} we hence see that
\[ \hat\tau_{-t}(x) \otimes \rho\circ\hat\tau_t
= (\uprho_{\alpha,i})^{-it} (\uprho_{\alpha,j})^{it} (\uprho_{\beta,k})^{-it} (\uprho_{\beta,l})^{it} x \otimes \rho. \]
In particular, function $\RR\ni t\mapsto \hat\tau_{-t}(x) \otimes \rho\circ\hat\tau_t$ takes values in a finite dimensional subspace of $\Pol(\wh\GGamma)\odot \LL^1(\wh\GGamma)$. For such a function we can define its mean $m_{\mathbb R}\bigl( t\mapsto \hat\tau_{-t}(x) \otimes \rho\circ\hat\tau_t \bigr) \in \Pol(\wh\GGamma) \odot \LL^1(\wh\GGamma)$ by choosing a basis in the image, and taking mean of $t$-dependent coefficients (it will not depend on the choice of basis). By linearity, we obtain $m_{\RR}\bigl( t\mapsto \hat\tau_{-t}(x) \otimes \rho\circ\hat\tau_t \bigr) \in \Pol(\wh\GGamma) \odot \LL^1(\wh\GGamma) $ for any $x,y\in\Pol(\wh\GGamma)$. By linearity again, given $x\in \Pol(\wh\GGamma)\odot \mc{K}(\msf{H})$ and $\rho\in h(\Pol(\wh\GGamma) \,\cdot\,)\odot \B(\msf{H})_*$, it follows that
\[
v=m_{\RR}\bigl(t\mapsto
(\hat \tau_{-t}\otimes\id)(x) \otimes  \rho\circ (\hat \tau_t\otimes\id)\bigr)\in 
\bigl(\mrm{C}(\wh\GGamma)\odot\mc{K}(\msf{H})\bigr) \odot
\bigl(\LL^1(\wh\GGamma)\odot\B(\msf{H})_*\bigr)
\]
and we have
\[
\Omega^{\tau}_{x,\rho}(a)= \kappa \bigl((\Theta^l(a)\otimes\id\otimes\id\otimes\id)v\bigr)=\Omega_v(a).
\]
Consequently we have $\Omega^{\tau}_{x,\rho}\in Q^l(\A(\GGamma))$ by Lemma \ref{lemma16}. Furthermore,
\[
\|\Omega^{\tau}_{x,\rho}\|\le 
\|v\|\le \|x\|\|\rho\|.
\]
General elements $x\in \mrm{C}(\wh\GGamma)\otimes \mc{K}(\msf{H}),\rho\in \LL^1(\wh\GGamma)\wh{\otimes}\B(\msf{H})_*$ can be approximated in norm by $x,\rho$ as above, hence $\Omega^{\tau}_{x,\rho}$ is a normal functional, using again that $Q^{l}(\A(\GGamma))\subseteq \M^l_{cb}(\A(\GGamma))^*$ is a closed subspace.
\end{proof}

\begin{proof}[Proof of Proposition \ref{prop1}]
By assumption, there is a net $(a_i)_{i\in I}$ in $\A(\GGamma)$ which converges to $\I$ in the weak$^*$-topology of $\M^l_{cb}(\A(\GGamma))$.  For each $i\in I$ there is $\omega_i\in \LL^1(\wh\GGamma)$ with $a_i = \wh\lambda(\omega_i)$, and there are $\xi_i, \eta_i\in \LL^2(\wh\GGamma)$ with $\omega_i = \omega_{\xi_i, \eta_i}$.  Given $n\in\NN$ we may choose $\xi_{i,n}, \eta_{i,n} \in \Lambda_h(\Pol(\wh\GGamma))$ with $\|\xi_i-\xi_{i,n}\|\leq\epsilon_1$ and $\| \eta_i - \eta_{i,n} \| \leq \epsilon_2$, where
\[ \epsilon_1 = \tfrac{1}{1+2 n \|\eta_i\|}, \quad
\epsilon_2 = \tfrac{1}{2n(\epsilon_1 + \|\xi_i\| )}.
 \]
Set $a_{i,n} = \wh\lambda(\omega_{\xi_{i,n}, \eta_{i,n}})$, so that $a_{i,n} \in \mrm{c}_{00}(\GGamma)$, and
\begin{align*} \|a_i - a_{i,n}\|_{\A(\GGamma)} &= \| \omega_i - \omega_{\xi_{i,n}, \eta_{i,n}} \|
\leq \| \omega_{\xi_i, \eta_i} - \omega_{\xi_{i,n}, \eta_i} \| + \| \omega_{\xi_{i,n}, \eta_i} - \omega_{\xi_{i,n}, \eta_{i,n}} \| \\
&\leq \epsilon_1 \|\eta_i\| + \epsilon_2 \|\xi_{i,n}\|
\leq \epsilon_1 \|\eta_i\| + \epsilon_2 \big( \epsilon_1 + \|\xi_i\| \big) \le \tfrac{1}{n}.
\end{align*}
Equipping $I\times\mathbb N$ with the product order, it follows that $a_{i,n}\xrightarrow[(i,n)\in I\times \NN]{}\I$ in $(\M^l_{cb}(\A(\GGamma)),w^*)$.

Fix $(i,n)\in I\times \NN$ and choose $m_{\RR}\in \LL^\infty(\RR)^*$, a mean on $\RR$. Let $b_{i,n}$ be the unique element of $\ell^{\infty}(\GGamma)$ with
\[
\la b_{i,n},\omega \ra =
m_{\RR}( t\mapsto \la \tau_t(a_{i,n}) , \omega \ra )\qquad(\omega\in \ell^{1}(\GGamma)).
\]
As each $\tau_t$ leaves each matrix block $\M_{\dim(\alpha)} \subseteq \ell^\infty(\GGamma)$ invariant, it follows that $b_{i,n}\in \mrm{c}_{00}(\GGamma)$, because $a_{i,n}\in \mrm{c}_{00}(\GGamma)$.  Next, we set
\[
c_{i,n}=\tfrac{1}{2}(b_{i,n}+R(b_{i,n})^*)\in \mrm{c}_{00}(\GGamma).
\]
Clearly each $b_{i,n}$, and consequently each $c_{i,n}$, is invariant under $(\tau_t)_{t\in\RR}$.  Thus $c_{i,n}$ is analytic for $(\tau_t)_{t\in\RR}$, and so $c_{i,n} \in D(S^{-1})$ and $S^{-1}(c_{i,n}) = R(c_{i,n})$.  Then as $\nabla_\vp^{it}=\nabla_\psi^{-it}=P^{it}$, see \cite[Lemma 6.2]{KrajczokTypeI}, each $c_{i,n}$ is also invariant under the modular automorphism group. 
Since
\[ c_{i,n}^*=
\tfrac{1}{2}(b_{i,n}^*+R(b_{i,n}))=
R(c_{i,n})=S^{-1}(c_{i,n}), \]
the operator $\Theta^l(c_{i,n})$ is star preserving by Corollary~\ref{cor2}.

We now show that $(c_{i,n})_{(i,n)\in I\times \NN}$ converges to $\I$ weak$^*$ in $\M^l_{cb}(\A(\GGamma))$.  We will show this first for $(b_{i,n})_{(i,n)\in I\times \NN}$.  Choose $\rho\in\ell^1(\GGamma),\omega\in\LL^1(\wh\GGamma)$ and set $y=(\id\otimes\rho)\ww^{\wh\GGamma}$.  We have
\begin{equation}\begin{split}\label{eq1}
&\quad\;
\la \Theta^l(b_{i,n}) (y),\omega\ra =
\la (\id\otimes \rho)  (\ww^{\wh\GGamma}) ,\Theta^l(b_{i,n})_*(\omega)\ra=
\la b_{i,n} (\omega\otimes\id)(\ww^{\wh\GGamma}) , \rho\ra\\
&=
m_{\RR}\bigl(t\mapsto 
\la \tau_t(a_{i,n}) (\omega\otimes\id)(\ww^{\wh\GGamma}) , \rho\ra\bigr)=
m_{\RR}\bigl(t\mapsto 
\la \Theta^l(\tau_t(a_{i,n})) y ,\omega\ra 
\bigr).
\end{split}\end{equation}
By continuity, the above equation holds for each $y\in \mrm{C}(\wh\GGamma)$.

Let $\msf{H}$ be a separable Hilbert space, and consider $x\in \mrm{C}(\wh\GGamma)\odot \mc{K}(\msf{H})$ and $\rho\in \LL^1(\wh\GGamma)\odot \B(\msf{H})_*$.  Then using linearity and \eqref{eq1}, it follows that
\[
\la b_{i,n}, \Omega_{x,\rho}\ra =
\la (\Theta^l(b_{i,n})\otimes \id)x , \rho \ra =
m_{\RR}\bigl(t\mapsto 
\la (\Theta^l(\tau_t(a_{i,n}))\otimes \id)x , \rho \ra \bigr).
\]
Lemmas~\ref{lemma6} and~\ref{lemma8} imply
\[
\la b_{i,n},\Omega_{x,\rho}\ra=
m_{\RR}\bigl(t\mapsto
\la (\Theta^l(a_{i,n})\otimes \id) (\hat{\tau}_{-t}\otimes\id)(x) , \rho\circ (\hat{\tau}_t\otimes\id)\ra 
\bigr)=
\la a_{i,n},\Omega^{\tau}_{x,\rho}\ra .
\]
Both sides of the above equation are continuous with respect to $x,\rho$, hence it holds also for $x\in \mrm{C}(\wh\GGamma)\otimes \mc{K}(\msf{H})$ and $\rho\in \LL^1(\wh\GGamma)\wh{\otimes} \B(\msf{H})_*$. Since $\Omega^{\tau}_{x,\rho}$ is a normal functional, it follows that
\begin{align*}
\lim_{(i,n)\in I\times \NN} \la b_{i,n},\Omega_{x,\rho}\ra
&= \lim_{(i,n)\in I\times \NN} \la a_{i,n},\Omega^{\tau}_{x,\rho}\ra
= \la \I,\Omega^{\tau}_{x,\rho}\ra  \\
&= m_{\mathbb R}\big( t\mapsto \la (\hat\tau_{-t}\otimes\id)(x), \rho\circ(\hat \tau_t\otimes\id) \ra \big)
= \la x, \rho \ra = \la \I, \Omega_{x,\rho} \ra
\end{align*}
and so $b_{i,n}\xrightarrow[(i,n)\in I\times \NN]{w^*}\I$, as the functionals $\Omega_{x,\rho}$ give all of $Q^l(\A(\GGamma))$, by Proposition~\ref{prop15}.  Then using Proposition~\ref{prop12},
\[\begin{split}
&\quad\;
\la R(b_{i,n})^* , \Omega_{x,\rho} \ra=
\la S(b_{i,n}^*) , \Omega_{x,\rho} \ra=
\la ( \Theta^l(b_{i,n})^{\dagger} \otimes \id)x , \rho \ra \\
&=
\la (\Theta^l(b_{i,n})\otimes \id)(x^*)^* , \rho \ra=
\ov{\la (\Theta^l(b_{i,n})\otimes \id)(x^*) , \ov{\rho} \ra}
\end{split}\]
which converges to $\ov{\la x^*,\ov{\rho}\ra }=\la x,\rho\ra=\la \I,\Omega_{x,\rho}\ra $. We can conclude that $c_{i,n}=\tfrac{1}{2}(b_{i,n}+R(b_{i,n})^*)$ also converges weak$^*$ to $\I$. 

We have now shown all the properties required of the net $(c_{i,n})_{(i,n)\in I\times\mathbb N}$ except that each $\Theta^l(c_{i,n})$ is unit preserving. By Lemma~\ref{lem:cent_unit_preserving}, we know that there is a family of scalars $(\alpha_{i,n})_{(i,n)\in I\times\mathbb N}$ with $\Theta^l(c_{i,n})(\I) = \alpha_{i,n}\I$ for each $(i,n)$.  As $\Theta^l(c_{i,n})$ is star-preserving, each $\alpha_{i,n}\in\mathbb R$.  As $c_{i,n} \xrightarrow[(i,n)\in I\times \NN]{} \I$ weak$^*$, $\Theta^l(c_{i,n})(\I) = \alpha_{i,n}\I \xrightarrow[(i,n)\in I\times \NN]{} \I$ weak$^*$ and so $\alpha_{i,n} \xrightarrow[(i,n)\in I\times \NN]{} 1$.  We may hence replace $c_{i,n}$ by $\alpha_{i,n}^{-1} c_{i,n}$.

Finally, when $\GGamma$ has central AP, then we can skip the first step (as $a_i\in \mrm{c}_{00}(\GGamma)\cap \mc{Z}(\ell^{\infty}(\GGamma))$ by assumption), and proceed as above to form $b_i$. It follows from \eqref{eq:scaling_gp_cmpt} that $\hat\tau_t(U^{\alpha}_{i,i}) = U^{\alpha}_{i,i}$, and the equality $(\hat\tau_t\otimes \tau_t)\ww^{\wh\GGamma}=\ww^{\wh\GGamma}$, together with \eqref{eq:hatW_cmpt}, shows $\tau_t(a_i) = a_i$ for each $t,i$.  Thus actually $b_i = a_i$, and the final step of forming $c_i$, and rescaling, will also give central elements.
\end{proof}

In the unimodular case there is no difference between AP and central AP. 

\begin{proposition} \label{prop19}
Let $\GG$ be a unimodular discrete quantum group. Then $\GG$ has AP if and only if it has central AP. 
\end{proposition} 

\begin{proof} 
Since the Haar integral $ h\in \LL^1(\whG) $ is a trace there exists a unique state-preserving normal faithful conditional 
expectation $ E\colon \LL^\infty(\whG) \bar{\otimes} \LL^\infty(\whG) \rightarrow \wh{\Delta}(\LL^\infty(\whG)) \subseteq \LL^\infty(\whG) \bar{\otimes} \LL^\infty(\whG) $.
Explicitly, we have 
\[
E(U^{\alpha}_{i,j} \otimes U^{\beta}_{k,l}) = \frac{\delta_{\alpha\beta} \delta_{jk}}{\dim(\alpha)} \wh{\Delta}(U^{\alpha}_{i,l}),
\]
compare \cite[Section 6.3.2]{Brannan}. 
Set $ \wh{\Delta}^{\sharp} = \wh{\Delta}^{-1} E: \LL^\infty(\whG) \bar{\otimes} \LL^\infty(\whG) \rightarrow \LL^\infty(\whG) $.  Given $ a \in \M_{cb}^l(\A(\GG)) $ define
\[ \Psi(a) = \wh{\Delta}^{\sharp}(\id \otimes \Theta^l(a)) \wh{\Delta} \in \CB^\sigma( \Linfd ). \]
That $\Psi(a)$ is normal and completely bounded is clear, and note that we have $\|\Psi(a)\|_{cb} \leq \|a\|_{cb} $.   Moreover, if $a$ is finitely supported then $\Psi(a)$ has finite-rank, see Remark~\ref{rem:fin_sup_finrank}.

Define also $ A:\ell^\infty(\GG) \rightarrow \mc{Z}\ell^\infty(\GG) $ by 
$$
A(f) = \sum_{\alpha \in \Irr(\whG)} \frac{\Tr_\alpha(f)}{\dim(\alpha)} p_\alpha,
$$
where $p_\alpha\in \ell^{\infty}(\GG)$ is the central projection corresponding to $\B(\msf{H}_\alpha)\subseteq \ell^{\infty}(\GG)$ and $\Tr_\alpha\in\ell^1(\GG)$ is the projection onto $\B(\msf{H}_\alpha)$ composed with the (non-normalised) trace on $\B(\msf{H}_\alpha)$.  Then $ A $ is a contractive linear map.  Given $ a \in \M_{cb}^l(\A(\GG)) $ and $\omega\in h(\Pol(\wh\GG)\cdot)$, we compute
\[\begin{split}
&\quad\;
\lambda_{\wh\GG}( \Psi(a)_*(\omega) )
= \sum_{\alpha\in \Irr(\wh\GG)}\sum_{i,j=1}^{\dim(\alpha)}
\la \Psi(a)(U^{\alpha}_{i,j}) , \omega \ra 
e^{\alpha}_{i,j}\\
&=
\sum_{\alpha\in \Irr(\wh\GG)}\sum_{i,j,k=1}^{\dim(\alpha)}
\la \wh{\Delta}^{\sharp} ( U^{\alpha}_{i,k}\otimes \Theta^l(a)(U^{\alpha}_{k,j})) , \omega \ra 
e^{\alpha}_{i,j}=
\sum_{\alpha\in \Irr(\wh\GG)}\sum_{i,j,k,l=1}^{\dim(\alpha)}
a^{\alpha}_{k,l}
\la \wh{\Delta}^{\sharp} ( U^{\alpha}_{i,k}\otimes U^{\alpha}_{l,j}) , \omega \ra 
e^{\alpha}_{i,j}
\\
&=
\sum_{\alpha\in \Irr(\wh\GG)}\sum_{i,j,k=1}^{\dim(\alpha)}
\tfrac{a^{\alpha}_{k,k}}{\dim(\alpha)}
\la U^{\alpha}_{i,j} , \omega \ra 
e^{\alpha}_{i,j}
= \sum_{\alpha\in \Irr(\wh\GG)}\sum_{i,j=1}^{\dim(\alpha)}
\la U^{\alpha}_{i,j} , \omega \ra \tfrac{\Tr_\alpha(a)}{\dim(\alpha)} e^{\alpha}_{i,j}
=A(a) \lambda_{\wh\GG}(\omega).
\end{split}\]
Here we used Lemma~\ref{lemma4} to compute the action of $\Theta^l(a)$, and notice that by the choice of $\omega$, all the sums involved are finite.  As such $\omega$ are dense in $\LL^1(\wh\GGamma)$, it follows that $ A(a) $ is a left CB multiplier and $ \Theta^l(A(a)) = \Psi(a)$.  In particular, $ A(a) \in \mc{Z}\M^l_{cb}(\A(\GG)) $ and $ \|A(a)\|_{cb} \leq \|a\|_{cb} $. 

Now assume that $ \GG $ has AP. Using Proposition \ref{prop1} we find a net let $ (f_i)_{i \in I}$ in $ \mrm{c}_{00}(\GG) $ which converges weak$^*$ to $ \I $ in $ \M_{cb}^l(\A(\GG)) $.  In order to prove that $ \GG $ has central AP, we shall show that $ A(f_i) \xrightarrow[i\in I]{}\I $ weak$^*$ in $ \M^l_{cb}(\A(\GG)) $. 
Take $ x \in \mrm{C}(\whG) \otimes \mc{K}(\msf{H}), \omega \in \LL^1(\whG) \wh{\otimes} \B(\msf{H})_* $ for a separable Hilbert space $\msf{H}$. Then we have
\begin{align*}
\la A(f_i), \Omega_{x,\omega} \ra  &= \la (\Psi(f_i) \otimes \id)x, \omega \ra  \\
&= \la (\wh{\Delta}^\sharp \otimes \id)(\id\otimes \Theta^l(f_i) \otimes\id)(\wh{\Delta} \otimes \id)x, \omega\ra  \\
&= \la (\id\otimes \Theta^l(f_i) \otimes \id)(\wh{\Delta} \otimes \id)x, \omega \circ (\wh{\Delta}^\sharp \otimes \id)\ra .
\end{align*}
By applying Lemma~\ref{lemma16}, with $\M=\LL^\infty(\wh{\GGamma}), \N = \mathbb C$, it follows that
\[ \lim_{i\in I} \la A(f_i), \Omega_{x,\omega} \ra
= \la (\id\otimes \id \otimes \id)(\wh{\Delta} \otimes \id)x, \omega \circ (\wh{\Delta}^\sharp \otimes \id)\ra
= \la x, \omega \ra
= \la \I, \Omega_{x,\omega} \ra, \]
hence showing that $ A(f_i) \xrightarrow[i\in I]{}\I $ weak$^*$, as required.
\end{proof} 

In the remainder of this section we relate AP to approximation properties of associated operator algebras. Let us start by recalling the appropriate von Neumann algebraic approximation property, see \cite[Section 2]{HaagerupKraus}. 

\begin{definition}
Let $\M$ be a von Neumann algebra. Then $\M$ has the \emph{weak$^*$ operator approximation property (W$^*$OAP)} if there exists a net $(\Theta_i)_{i\in I}$ of finite rank normal CB maps on $\M$ which converges to the identity in the stable point-weak$^*$-topology, i.e.~$(\Theta_i\otimes\id)x\xrightarrow[i\in I]{w^*} x$ for all separable Hilbert spaces $\msf{H}$ and $x\in \M\bar{\otimes} \B(\msf{H})$.
\end{definition}

The following result was established by Kraus and Ruan for Kac algebras, \cite[Theorem 4.15]{KrausRuan}, using the formally stronger definition of AP (which by Theorem~\ref{thm6} 
is equivalent to the definition taken in this paper). A result of this type was also obtained in \cite[Proposition 4.7]{Crann}, again with a formally stronger definition of AP, and under the assumption that $\GG$ is strongly inner amenable. We record a proof for the convenience of a reader.

\begin{proposition}\label{prop5}
Let $\GG$ be a locally compact quantum group. If $\GG$ has AP, then $\Linfd$ has W$^*$OAP.
\end{proposition}

\begin{proof}
Assume that $\GG$ has AP. By Theorem~\ref{thm6} there is a net $(\wh{\lambda}(\omega_i))_{i\in I}$ in $\A(\GG)$ such that $\Theta^l(\wh{\lambda}(\omega_i))\xrightarrow[i\in I]{}\id$ in the stable point-weak$^*$-topology of $\CB^\sigma(\Linfd)$. Extend $\omega_i\in \Ljd$ to $\tilde{\omega}_i\in \B(\LdG)_*$ with the same norm and define
\[\begin{split}
\Psi_i\colon \B(\LdG)\ni T\mapsto
&
( \tilde{\omega}_i\otimes \id)\bigl(\vv^{\whG }(T\otimes \I)\vv^{\whG *} \bigr)\in \B(\LdG).
\end{split}\]
Clearly $\Psi_i$ is a normal CB map on $\B(\LdG)$, and as $\vv^{\whG}\in \LL^{\infty}(\GG)'\bar{\otimes}\Linfd$ the image of $\Psi_i$ lies in $\Linfd$. Note that if $x\in \Linfd$ then
\[\begin{split}
\Psi_i(x)=
(\tilde{\omega}_i\otimes \id)\bigl(\vv^{\whG }(x\otimes \I)\vv^{\whG *} \bigr)=
(\tilde{\omega}_i\otimes\id)\wh\Delta(x)= \Theta^l(\wh{\lambda}(\omega_i))(x).
\end{split}\]
Thus $\Psi_i$ is an extension of $\Theta^l(\wh{\lambda}(\omega_i))$ to all of $\B(\LdG)$.
As $\B(\LdG)$ has W$^*$CPAP, see \cite[Propositions 2.1.4, 2.2.7]{BrownOzawa}, there is a net $(\Upsilon_\lambda)_{\lambda\in \Lambda}$ of finite rank normal unital CP maps on $\B(\LdG)$ which converges to the identity in the point-weak$^*$-topology.  Consider now the maps
\[
\Psi_{i,\lambda}=\Psi_{i}\circ \Upsilon_\lambda|_{\Linfd}\colon \Linfd\rightarrow \Linfd.
\]
These are normal and CB, and, because $\Upsilon_\lambda$ has finite rank, also each $\Psi_{i,\lambda}$ is finite-rank.  Fix a separable Hilbert space $\msf{H}$ and finite sets $F\subseteq  \Linfd\bar{\otimes}\B(\msf{H}),\;G\subseteq \Ljd\wh{\otimes}\B(\msf{H})_*$ and $0<\eps<1$.  Given $x\in F, \rho\in G$, we have
\[
\la (\Psi_i\otimes \id)x,\rho \ra =
\la (\Theta^l(\wh{\lambda}(\omega_i))\otimes \id)\,x,\rho \ra \xrightarrow[i\in I]{}
 \la x,\rho\ra.
\]
Thus there is $i(F,G,\eps)\in I$ such that
\[
|\la (\Psi_{i( F,G,\eps)}\otimes\id)x - x,\rho \ra |\le \tfrac{\eps}{2}\qquad(x\in F,\rho\in G).
\]
Next, since $\Psi_{i( F,G,\eps)}$ is normal, we have
\[
|\la (\Psi_{i( F,G,\eps)}\circ \Upsilon_\lambda\otimes \id)x  - (\Psi_{i( F,G,\eps)}\otimes \id)x,\rho \ra |\xrightarrow[\lambda\in \Lambda]{}0\qquad(x\in F,\rho\in G),
\] 
hence there is $\lambda( F,G,\eps)\in \Lambda$ so that 
\[
|\la (\Psi_{i( F,G,\eps)}\circ \Upsilon_{\lambda( F,G,\eps)}\otimes\id)x - (\Psi_{i(  F,G,\eps)}\otimes \id)x,\rho \ra |\le \tfrac{\eps}{2}\qquad(x\in F,\rho\in G),
\]
and by triangle inequality
\[
|\la (\Psi_{i(F,G,\eps)}\circ \Upsilon_{\lambda(F,G,\eps)}\otimes\id)x-x , \rho \ra |\le \eps\qquad (x\in F,\rho\in G).
\]
Consequently, the net $(\Psi_{i(F,G,\eps),\lambda(F,G,\eps)})_{(F,G,\eps)}$ (indexed by finite subsets of $\Linfd\bar{\otimes}\B(\msf{H})$,\\ $\Ljd\wh{\otimes}\B(\msf{H})_*$ and $\left]0,1\right[$) shows that $\Linfd$ has the W$^*$OAP.
\end{proof}

\begin{remark}
An analogous argument shows that if $\whG$ is coamenable then $\Linfd$ has W$^*$CPAP. In fact, a formally stronger result holds: W$^*$CPAP of $\Linfd$ follows from amenability of $\GG$ by \cite[Theorem 3.3]{BT_Amenability}.
\end{remark}

When the quantum group $\GGamma$ is discrete and has AP, we obtain approximation properties also for the associated \cst-algebra. If $\GGamma$ is furthermore unimodular, the converse implications hold. These results are already known (see e.g.~\cite[Theorem 5.13]{KrausRuan}), hence we skip the proof.
For the definition of OAP and strong OAP, see \cite[Page 204]{EffrosRuan} or \cite[Section~12.4]{BrownOzawa}, for example.

\begin{proposition}\label{prop4}
Let $\GGamma$ be a discrete quantum group. Consider the following conditions:
\begin{enumerate}
\item $\GGamma$ has AP,
\item $\mrm{C}(\wh\GGamma)$ has strong OAP,
\item $\mrm{C}(\wh\GGamma)$ has OAP,
\item $\LL^{\infty}(\wh\GGamma)$ has W$^*$OAP.
\end{enumerate}
Then $(1)\Rightarrow (2)\Rightarrow (3)$ and $(1)\Rightarrow (4)$. If $\GGamma$ is unimodular then all the above conditions are equivalent.
\end{proposition}

\begin{remark}
Combining Proposition~\ref{prop4} and Proposition~\ref{prop19}, we see that when $\GGamma$ is unimodular  and $\LL^{\infty}(\wh\GGamma)$ has W$^*$OAP, then $\GGamma$ has the central AP.
\end{remark}

Following \cite[Definition 1.27]{boundary}, we say that a discrete quantum group $\GGamma$ is exact when the \emph{reduced crossed product functor} $\GGamma\ltimes_r - $, 
preserves short exact sequences.
 
\begin{corollary}
Let $\GGamma$ be a discrete quantum group. If $\GGamma$ has AP then it is exact.
\end{corollary}

\begin{proof}
By Proposition \ref{prop4}, the \cst-algebra $\mrm{C}(\wh\GGamma)$ has strong OAP. It is then exact by a combination of \cite[Corollary~11.3.2]{EffrosRuan} and \cite[Theorem~1.1]{Kirchberg}. The result now follows from \cite[Proposition~1.28]{boundary}.
\end{proof}

We end this section with a result which shows that for a discrete quantum group $\GGamma$, AP is equivalent to a strengthening of W$^*$OAP of $\LL^{\infty}(\wh{\GGamma})$ which takes into 
consideration $\ell^{\infty}(\GGamma)$. Let us introduce this strengthening in a general setting, compare \cite[Definition 6.9]{KrajczokPhD}.

\begin{definition}
Let $(\M,\theta)$ be a von Neumann algebra with a n.s.f.~weight.
\begin{itemize}
\item Let $\Phi\in \CB^{\sigma}(\M)$ be a normal CB map satisfying $\Phi(\mf{N}_{\theta})\subseteq \mf{N}_{\theta}$. We say that $\Phi$ has an $\LL^2$-implementation if there is $T\in \B(\msf{H}_\theta)$ such that $\Lambda_\theta(\Phi(x))=T\Lambda_\theta(x)$ for $x\in \mf{N}_\theta$.
\item Let $\N\subseteq \B(\msf{H}_\theta)$ be a von Neumann algebra. We say that $(\M,\theta)$ has W$^*$OAP relative to $\N$ if there is a net $(\Phi_i)_{i\in I}$ such that:
\begin{itemize} 
\item each $\Phi_i$ is a normal, CB, finite rank map on $\M$,
\item each $\Phi_i$ satisfies 
$\Phi_i(\mf{N}_\theta)\subseteq \mf{N}_\theta$ and has an $\LL^2$-implementation $T_i\in \N$,
\item the net $(\Phi_i)_{i\in I}$ converges to the identity in the stable point-weak$^*$-topology.
\end{itemize}
\end{itemize}
\end{definition}

Note that an $\LL^2$-implementation is unique. If it is clear from the context which weight on $\M$ we choose, we will simply say that $\M$ has W$^*$OAP relative to $\N$. 

\begin{theorem}\label{thm4}
Let $\GGamma$ be a discrete quantum group. Consider the following conditions:
\begin{enumerate}[label=(\arabic*)]
\item\label{thm4:1} $\GGamma$ has AP.
\item\label{thm4:2} $\LL^{\infty}(\wh{\GGamma})$ has W$^*$OAP relative to $\ell^{\infty}(\GGamma)$.
\item\label{thm4:3} $\LL^{\infty}(\wh{\GGamma})$ has W$^*$OAP relative to $\ell^{\infty}(\GGamma)'$.
\item\label{thm4:4} $\GGamma$ has central AP.
\item\label{thm4:5} $\LL^{\infty}(\wh{\GGamma})$ has W$^*$OAP relative to $\mc{Z}(\ell^{\infty}(\GGamma))$. 
\end{enumerate}
Then \ref{thm4:1} $\Leftrightarrow$ \ref{thm4:2} $\Leftrightarrow$ \ref{thm4:3} $\Leftarrow$ \ref{thm4:4} $\Leftrightarrow$ \ref{thm4:5}.
\end{theorem}

\begin{remark}
This result is an analogue of Theorem 6.11 in \cite{KrajczokPhD}, for (co)amenability and relative W$^*$CPAP.  Furthermore, it is similar in spirit to \cite[Theorem 3]{SoltanViselter} which is concerned with amenability and injectivity.
\end{remark}

We start with an auxiliary result (compare \cite[Proposition 6.12]{KrajczokPhD} and \cite[Proposition 2.12]{SkalskiViselterConvolution}). Recall that for a normal CB map $\Phi$ on a von Neumann algebra, we denote by $\Phi^\dagger$ the normal CB map given by $x\mapsto \Phi(x^*)^*$.

\begin{proposition}\label{prop10}
Let $\GG$ be a locally compact quantum group with left Haar integral $\vp$ and let $\Phi\in \CB^\sigma(\Linfd)$ be a normal CB map. Assume that $\Phi^\dagger$ satisfies $\Phi^\dagger(\mf{N}_{\hvp})\subseteq \mf{N}_{\hvp}$ and has an $\LL^2$-implementation $T\colon \LdG\rightarrow\LdG$. We have:
\begin{enumerate}[label=(\arabic*)]
\item\label{prop10:1} $T\in \Linf$ if and only if $\Phi_*(\omega\star\nu)=\Phi_*(\omega)\star \nu$ for all $\omega,\nu\in\Ljd$,
\item\label{prop10:2} $T\in \Linf'$ if and only if $\Phi_*(\omega\star\nu)=\omega\star \Phi_*(\nu)$ for all $\omega,\nu\in\Ljd$,
\item\label{prop10:3} $T\in \mc{Z}(\Linf)$ if and only if $\Phi_*(\omega\star\nu)=\Phi_*(\omega)\star \nu=\omega\star\Phi_*(\nu)$ for all $\omega,\nu\in\Ljd$.
\end{enumerate}
\end{proposition}

\begin{proof}
Using the biduality $\GG=\widehat{\whG}$ and \cite[Definition 4.6]{LCQGDaele} (see also Section \ref{section preliminaries}), we deduce that the subspace
\[
\mc{N}=	\{\widehat{\lambda}(\omega)\,|\, \omega\in \Ljd\colon \exists_{\xi\in \LdG} \forall_{x\in \mf{N}_{\hvp}} \ismaa{\Lhvp(x)}{\xi}=\omega(x^*)\} \subseteq \Linf
\]
is a core for $\Lvp$, and that for $\widehat{\lambda}(\omega)\in \mc{N}$ we have $\Lvp(\widehat{\lambda}(\omega))=\xi$. Before we proceed with the main proof, let us establish some preliminary results:
\begin{itemize}
\item 
For $\widehat{\lambda}(\omega)\in\mc{N}$ we have $\widehat{\lambda}(\Phi_*(\omega))\in \mc{N}$ and
\begin{equation}\label{eq11}
T^*\Lvp(\widehat{\lambda}(\omega))=\Lvp(\widehat{\lambda}(\Phi_*(\omega))).
\end{equation}
\end{itemize}
Indeed, for $x\in \mf{N}_{\hvp}$,
\begin{align*}
\ismaa{\Lhvp(x)}{T^* \Lvp(\widehat{\lambda}(\omega))}
&= \ismaa{T\Lhvp(x)}{ \Lvp(\widehat{\lambda}(\omega))}
= \ismaa{\Lhvp(\Phi^{\dagger}(x))}{\Lvp(\widehat{\lambda}(\omega))} \\
&= \omega( \Phi^{\dagger}(x)^*)
= \omega(\Phi(x^*))
= \Phi_*(\omega)(x^*),
\end{align*}
which proves that $\widehat{\lambda}(\Phi_*(\omega))\in \mc{N}$ and that equation \eqref{eq11} holds.
\begin{itemize}
\item For $\omega\in \Ljd,\widehat{\lambda}(\nu)\in \mc{N}$ we have $\widehat{\lambda}(\omega\star\nu)\in \mc{N}$ and
\begin{equation}\label{eq13}
\Lambda_{\vp}(\widehat{\lambda}(\omega\star\nu))=
\widehat{\lambda}(\omega)\Lambda_{\vp}(\widehat{\lambda}(\nu)).
\end{equation}
\end{itemize}
Indeed, take $x\in \mf{N}_{\hvp}$. Using the definition of $\ww^{\whG}$ (equation \eqref{eq23}) we obtain
\[\begin{split}
&\quad\;
(\omega\star\nu)(x^*)=
\nu( (\omega\otimes\id)\Delta(x^*))=
\nu((\ov{\omega}\otimes\id)\Delta(x)^*)=
\ismaa{\Lhvp( (\ov{\omega}\otimes\id)\Delta(x))}{\Lambda_{\vp}(\widehat{\lambda}(\nu))}\\
&=
\ismaa{ (\ov{\omega}\otimes\id)({\ww^{\whG *}}) \Lhvp(x)}{\Lambda_{\vp}(\widehat{\lambda}(\nu))}
= \ismaa{\Lhvp(x)}{\widehat{\lambda}(\omega)\Lambda_{\vp}(\widehat{\lambda}(\nu))},
\end{split}\]
which proves the claim.

Let us now prove \ref{prop10:1}. If $T\in\Linf$ then \eqref{eq11} implies that $\Lvp(\widehat{\lambda}(\Phi_*(\omega))) = T^*\Lvp(\widehat{\lambda}(\omega)) = \Lvp(T^* \widehat{\lambda}(\omega))$ and so $T^*\widehat{\lambda}(\omega)=\widehat{\lambda}(\Phi_*(\omega))$, for each $\omega\in \Ljd$ such that $\widehat{\lambda}(\omega)\in\mc{N}$, and by density of such $\omega$ (see Lemma~\ref{lemma22}) this equation holds for all $\omega\in \Ljd$.  Consequently
\[
\widehat{\lambda}(\Phi_*(\omega\star \nu))=
T^*\widehat{\lambda}(\omega\star\nu)=T^*\widehat{\lambda}(\omega)\widehat{\lambda}(\nu)=
\widehat{\lambda}(\Phi_*(\omega))\widehat{\lambda}(\nu)=\widehat{\lambda}(\Phi_*(\omega)\star\nu),
\]
and so $\Phi_*(\omega\star\nu)=\Phi_*(\omega)\star\nu$ for all $\omega,\nu\in \Ljd$.

For the converse, assume that $\Phi_*(\omega\star\nu)=\Phi_*(\omega)\star\nu$ for all $\omega,\nu\in \Ljd$.  Let us take $\widehat{\lambda}(\omega),\widehat{\lambda}(\nu)\in \mc{N}$ and assume that the map $\RR\ni t\mapsto (\omega\hat{\delta}^{-it})\circ \hat{\tau}_{-t}\in \Ljd$ extends to an entire map, and denoting by $\rho$ the value of this map at $t=-\tfrac{i}{2}$, we have furthermore $\widehat{\lambda}(\rho)\in \mc{N}$.  We now use ${\ww}^{\whG}=\chi(\ww^{\GG})^*$, together with $(\sigma^{\vp}_t\otimes\id)(\ww^{\GG}) = (\tau_t\otimes\id)(\ww^{\GG}) ( \I\otimes \hat{\delta}^{it})$, see \cite[Proposition 5.15]{LCQGDaele}, and $(\tau_t \otimes \hat\tau_t)(\ww^{\GG}) = \ww^{\GG}$, see \cite[Proposition~8.23]{KustermansVaes}.  It follows that for each $t\in\mathbb R$,
\[\begin{split}
&\quad\;
\sigma^{\vp}_t(\widehat{\lambda}(\omega))=(\omega\otimes\id)((\id\otimes \sigma^{\vp}_{t}){\ww}^{\whG})=
(\id\otimes \ov{\omega}) ( (\sigma^{\vp}_t\otimes\id)(\ww^{\GG}))^*\\
&=
(\id\otimes\ov{\omega})( (\tau_t\otimes\id)(\ww^{\GG}) ( \I\otimes \hat{\delta}^{it}))^*=
(\id\otimes\ov{\omega})( (\id\otimes \hat{\tau}_{-t})(\ww^{\GG}) ( \I\otimes \hat{\delta}^{it}))^*\\
&=
(\id\otimes\omega)( (\I\otimes\hat{\delta}^{-it}) (\id\otimes \hat{\tau}_{-t}) (\ww^{\GG})^*)=
(\omega\otimes\id) ( (\hat{\delta}^{-it}\otimes\I) (\hat{\tau}_{-t}\otimes\id) (\ww^{\whG}))\\
&=
\wh{\lambda}( ( \omega\hat{\delta}^{-it})\circ \hat{\tau}_{-t}).
\end{split}\]
Consequently we obtain $\widehat{\lambda}(\omega)\in \oon{Dom}(\sigma^\vp_{-i/2})$ and $\sigma^\vp_{-i/2}(\widehat{\lambda}(\omega))=\widehat{\lambda}(\rho)$.  From \eqref{eq13}, we know that $\widehat{\lambda}(\nu\star\rho), \widehat{\lambda}(\Phi_*(\nu)\star\rho)\in \mc{N}$. By assumption, $\Phi_*(\nu\star\rho)=\Phi_*(\nu)\star\rho$ hence for $x\in \mf{N}_{\hvp}$ we have
\[\begin{split}
&\quad\;
\ismaa{\Lhvp(x)}{T^* J_\vp \widehat{\lambda}(\omega)^* J_\vp \Lvp(\widehat{\lambda}(\nu))}=
\ismaa{T \Lhvp(x)}{ \Lvp( \widehat{\lambda}(\nu) \sigma^{\vp}_{-i/2}( \widehat{\lambda}(\omega)))}\\
&=
\ismaa{\Lhvp(\Phi^\dagger(x))}{\Lvp( \widehat{\lambda}(\nu \star\rho))}=
(\nu\star\rho)(\Phi^\dagger(x)^*)=\Phi_*(\nu\star\rho)(x^*)=(\Phi_*(\nu)\star\rho)(x^*)\\
&=
\ismaa{\Lhvp(x)}{\Lvp(\widehat{\lambda}(\Phi_*(\nu)\star \rho))}=
\ismaa{\Lhvp(x)}{\Lvp\bigl( \widehat{\lambda}(\Phi_*(\nu)) \sigma^{\vp}_{-i/2}(\widehat{\lambda}(\omega))\bigr)}\\
&=
\ismaa{\Lhvp(x)}{J_{\vp} \widehat{\lambda}(\omega)^* J_{\vp}  T^* \Lvp(\widehat{\lambda}(\nu))}.
\end{split}\]
By Lemma~\ref{lemma22} we know that the collection of functionals $\omega$ with $\wh{\lambda}(\omega)\in\mc{N}$ is dense in $\Ljd$, hence the corresponding collection of operators $J_\vp \widehat{\lambda}(\omega)^* J_\vp$ is weak$^*$-dense in $\Linf'$.  Thus $T^*\in \Linf''=\Linf$, as required.

Next we consider \ref{prop10:2}. Suppose that $T\in \Linf'$. By equations \eqref{eq11} and \eqref{eq13}, given $\widehat{\lambda}(\omega),\widehat{\lambda}(\nu) \in \mc{N}$, we have
\[
\Lvp(\widehat{\lambda}(\Phi_*(\omega\star \nu)))=T^* \Lvp(\widehat{\lambda}(\omega\star\nu))=
T^* \wh{\lambda}(\omega) \Lvp(\wh{\lambda}(\nu))=
\widehat{\lambda}(\omega)T^* \Lvp(\widehat{\lambda}(\nu))=
\Lvp(\widehat{\lambda}(\omega\star\Phi_*(\nu))),
\]
hence $\Phi_*(\omega\star\nu)=\omega\star\Phi_*(\nu)$. As the set of functionals $\omega$ with $\wh{\lambda}(\omega)\in \mc{N}$ is dense in $\Lj$ by Lemma~\ref{lemma22}, the claim follows.

Conversely, suppose that $\Phi_*(\omega\star\nu)=\omega\star\Phi_*(\nu)$ for all $\omega,\nu\in \Ljd$. For $\widehat{\lambda}(\omega),\widehat{\lambda}(\nu) \in \mc{N}$ we then have
\[
\widehat{\lambda}(\omega) T^*\Lvp(\widehat{\lambda}(\nu))=
\Lvp(\widehat{\lambda}(\omega\star\Phi_*(\nu)))=
\Lvp(\widehat{\lambda}(\Phi_*(\omega\star\nu)))=
T^* \widehat{\lambda}(\omega)\Lvp(\widehat{\lambda}(\nu)).
\]
Again by density, it follows that $T^*\in \Linf'$, as required. 

Finally, \ref{prop10:3} follows by combining \ref{prop10:1} and \ref{prop10:2}.
\end{proof}

\begin{proof}[Proof of Theorem \ref{thm4}]
If $\GGamma$ has AP then by Proposition \ref{prop1} we have a net $(a_i)_{i\in I}$ in $\mrm{c}_{00}(\GGamma)$ which converges in $(\M^l_{cb}(\A(\GGamma)),w^*)$ to $\I$, and the associated maps $\Theta^l(a_i)$ satisfy $\Theta^l(a_i)^\dagger=\Theta^l(a_i)$. Proposition \ref{prop2} shows that each $\Theta^l(a_i)$ has an $\LL^2$-implementation equal to $S^{-1}(a_i)\in \ell^{\infty}(\GGamma)$. Proposition \ref{prop17}, for $f=\eps\in \ell^1(\GGamma)$ being the counit of $\GGamma$, shows that $(\Theta^l(a_i))_{i\in I}$ converges to the identity in the stable point-weak$^*$-topology. Indeed, for any Hilbert space $\msf{H}$ and $x\in \LL^{\infty}(\wh\GGamma)\bar{\otimes}\B(\msf{H}), \omega\in \LL^1(\wh\GGamma)\wh{\otimes}\B(\msf{H})_*$ we have 
\[
\la (\Theta^l(a_i)\otimes \id)x,\omega \ra =
\la (\Theta^l(a_i\star \eps)\otimes \id)x,\omega \ra =
\la a_i,\Omega_{x,\omega,\eps}\ra\xrightarrow[i\in I]{}\la \I,\Omega_{x,\omega,\eps}\ra =\la x,\omega\ra.
\]
We conclude that $\LL^{\infty}(\wh\GGamma)$ has W$^*$OAP relative to $\ell^{\infty}(\GGamma)$. This shows \ref{thm4:1} $\Rightarrow$ \ref{thm4:2}. 

If $\GGamma$ has central AP then additionally $S^{-1}(a_i)\in \mc{Z}(\ell^{\infty}(\GGamma))$ and so $\LL^{\infty}(\wh\GGamma)$ has W$^*$OAP relative to $\mc{Z}(\ell^{\infty}(\wh\GGamma))$.  Thus \ref{thm4:4} $\Rightarrow$ \ref{thm4:5}.

Let us now show the equivalence of \ref{thm4:2} and \ref{thm4:3}.  Assume \ref{thm4:2} and let $(\Phi_i)_{i\in I}$ be a net giving W$^*$OAP of $\LL^{\infty}(\wh\GGamma)$ relative to $\ell^{\infty}(\GGamma)$. Define a net $(\Psi_i)_{i\in I}$ by $\Psi_i=\wh{R}\circ \Phi_i^{\dagger}\circ\wh{R}$. Lemma \ref{lemma18} shows that each $\Psi_i$ is a normal, finite rank CB map on $\LL^{\infty}(\wh\GGamma)$ and $\Psi_i\xrightarrow[i\in I]{}\id$ in the stable point-weak$^*$-topology. Let $T_i\in \ell^{\infty}(\GGamma)$ be the $\LL^2$-implementation of $\Phi_i$.  By \cite[Proposition~2.11]{KVVN} we know that $J_\vp \Lambda_h(x)=\Lambda_h(\wh{R}(x)^*)$ for each $x\in \LL^{\infty}(\wh{\GGamma})$, and so
\[\begin{split}
&\quad\;
\Lambda_h(\Psi_i(x))=
\Lambda_h\bigl(\wh{R}\bigl(\Phi_i^{\dagger}(\wh{R}(x))\bigr)\bigr)=
J_\vp\Lambda_h( \Phi_i^\dagger(\wh{R}(x))^*)\\
&=
J_{\vp}\Lambda_h( \Phi_i(\wh{R}(x)^*))=
J_{\vp}T_i\Lambda_h( \wh{R}(x)^*)=
J_\vp T_i J_\vp \Lambda_h(x). 
\end{split}\]
Hence $\Psi_i$ has $\LL^2$-implementation $J_\vp T_i J_\vp \in \ell^{\infty}(\GGamma)'$, showing \ref{thm4:3}.  The converse is analogous.

Now assume \ref{thm4:2}, i.e.~that $\LL^{\infty}(\wh\GGamma)$ has W$^*$OAP relative to $\ell^{\infty}(\GGamma)$. As before, let $(\Phi_i)_{i\in I}$ be the corresponding net of normal, finite rank CB maps with $\LL^2$-implementations $T_i\in \ell^{\infty}(\GGamma)$.  For $i\in I$, let $\Psi = \Phi_i^\dagger$, so that $\Psi$ is a normal CB map such that $\Psi^\dagger = \Phi_i$ has an $\LL^2$-implementation $T_i \in \ell^{\infty}(\GGamma)$.  By Proposition~\ref{prop10}\ref{prop10:1}, $\Psi_*$ is a left centraliser.  Hence also $\Phi_{i,*}$ is a left centraliser, compare with the proof of Proposition~\ref{prop12}, and so there is $a_i \in \M^l_{cb}(\A(\GGamma))$ with $\Theta^l(a_i) = \Phi_i$.  By Lemma~\ref{lemma4}, as $\Phi_i$ is finite-rank, it must be that $a_i \in \mrm{c}_{00}(\GGamma)$.  By definition, $(\Phi_i)_{i\in I} = (\Theta^l(a_i))_{i\in I}$ converges to the identity in the stable point-weak$^*$-topology, and so by Proposition~\ref{prop15}, $a_i\xrightarrow[i\in I]{}\I$ in $(\M^l_{cb}(\A(\GGamma)), w^*)$ and consequently $\GGamma$ has AP. 
Therefore \ref{thm4:1} and \ref{thm4:2} are equivalent.

Finally, suppose that $\LL^{\infty}(\wh\GGamma)$ has W$^*$OAP relative to $\mc{Z}(\ell^{\infty}(\GGamma))$, and proceed as above.  By definition, we have that $T_i \Lambda_h(x) = \Lambda_h(\Phi_i(x)) = \Lambda_h(\Theta^l(a_i)(x))$ for each $x\in \mrm{C}(\wh\GGamma)$.  By Proposition~\ref{prop2}, it follows that $T_i = S(a_i^*)^*$.  As $T_i \in \mc{Z}(\ell^{\infty}(\GGamma))$, it follows that $a_i \in \mc{Z}(\ell^{\infty}(\GGamma))\cap\mrm{c}_{00}(\GGamma)$, which shows that $\GGamma$ has central AP.  This establishes that \ref{thm4:4} and \ref{thm4:5} are equivalent. The implication \ref{thm4:4} $\Rightarrow$ \ref{thm4:1} is trivial.
\end{proof}

\section{Permanence properties} \label{section permanence}

\subsection{Quantum subgroups}

For classical locally compact groups, AP passes to closed subgroups, see \cite[Proposition 1.14]{HaagerupKraus}. We shall show that an analogous property holds also in the quantum case. 

Let us start by recalling the notion of a closed quantum subgroup of a locally compact quantum group, see \cite{DKSS_ClosedSub}. In what follows we will use the universal \cst-algebra $\mrm{C}^u_0(\GG)$ and the reducing map $\Lambda_{\GG}: \mrm{C}^u_0(\GG) \rightarrow \mrm{C}_0(\GG)$, see \cite{KustermansUniversal}, along with the semi-universal and universal multiplicative unitaries, compare \cite[Section~1.2]{DKSS_ClosedSub}. We will also use the notion of a quantum homomorphism as explored in \cite{MeyerRoyWoronowicz}, see also \cite[Section~2.1]{DawsCat}.  

Let $\HH,\GG$ be locally compact quantum groups. Assume that there is a homomorphism $\HH\rightarrow \GG$ exhibited by a \emph{strong quantum homomorphism} in the sense 
of \cite[Section~1.3]{DKSS_ClosedSub}, i.e.~a non-degenerate $\star$-homomorphism $\pi: \mrm{C}_0^u(\GG) \rightarrow \M(\mrm{C}_0^u(\HH))$ such that $\Delta_{\HH}^u\circ \pi=(\pi\otimes\pi)\circ \Delta_{\GG}^u$. 
Then there is a dual strong quantum homomorphism $\widehat{\pi}: \mrm{C}_0^u(\whH) \rightarrow \M(\mrm{C}_0^u(\whG))$ which is related to $\pi$ via $(\pi\otimes \id)\WW^{\GG}=(\id\otimes \widehat{\pi})\WW^{\HH}$, see \cite[Section 1.3]{DKSS_ClosedSub}. In this situation we say that $\HH$ is a \emph{closed quantum subgroup} of $\GG$ in the sense of Vaes if there is a normal unital injective $\star$-homomorphism $\gamma\colon \LL^{\infty}(\whH)\rightarrow\LL^{\infty}(\whG)$ such that $\gamma(\Lambda_{\whH}(x))=\Lambda_{\whG}(\widehat{\pi}(x))$ for all $x\in \mrm{C}_0^u(\whH)$, see \cite[Definition 2.5]{VaesInduction}, \cite[Definition 3.1]{DKSS_ClosedSub}.  Notice that this condition implies that $\Delta_{\whG} \circ \gamma = (\gamma\otimes\gamma)\circ \Delta_{\whH}$.

\begin{theorem}\label{thm2}
Let $\HH,\GG$ be locally compact quantum groups and assume that $\HH$ is a closed quantum subgroup of $\GG$ in the sense of Vaes. If $\GG$ has AP then so does $\HH$.
\end{theorem}

\begin{proof}
Since $\HH$ is a closed quantum subgroup of $\GG$ we obtain maps $\pi,\widehat{\pi},\gamma$ as discussed above. If $\GG$ has AP, then according to Theorem \ref{thm6} we can choose a net $(a_i)_{i\in I}$ in $\A(\GG)\subseteq \mrm{C}_0(\GG)$ such that $(\Theta^l(a_i))_{i\in I}$ converges to the identity in the stable point-weak$^*$-topology of $\CB^\sigma(\Linfd)$.  For each $i\in I$ let $\omega_i\in \Ljd$ be such that $a_i=\lambda_{\whG}(\omega_i)$, and define $b_i=\lambda_{\whH}(\gamma_*(\omega_i))$. As $\gamma_*(\omega_i) \in \LL^1(\whH)$, we see that $b_i\in \A(\HH)$.
Let $\msf{H}$ be a separable Hilbert space and take $x\in \mrm{C}_0(\whH)\otimes \mc{K}(\msf{H}), \omega\in \LL^1(\whH)\wh{\otimes} \B(\msf{H})_*$. We have
\[
\la b_i,\Omega_{x,\omega}\ra =
\la (\Theta^l(b_i)\otimes \id)x,\omega \ra =
\la \bigl((\gamma_*(\omega_i)\otimes\id)\Delta_{\whH}\otimes \id\bigr)x,\omega\ra .
\]
Since $\gamma$ is a complete isometry, $\gamma_*$ is a complete quotient map (\cite[Corollary 4.1.9]{EffrosRuan}). By \cite[Proposition 7.1.7]{EffrosRuan}, $\gamma_*\otimes \id\colon \Ljd\wh{\otimes}\B(\msf{H})_*\rightarrow \LL^1(\whH)\wh{\otimes}\B(\msf{H})_*$ is also a complete quotient map, hence we can find $\omega'\in \LL^1(\whG)\wh{\otimes}\B(\msf{H})_*$ such that $\omega=(\gamma_*\otimes \id)\omega'$.

Since $\gamma$ intertwines the coproducts,
\begin{align*}
\la b_i,\Omega_{x,\omega}\ra
&= \la (\omega_i\otimes\id\otimes\id) 
\bigl((\gamma\otimes \gamma)\Delta_{\whH}\otimes \id\bigr)x,\omega'\ra \\
&= \la ((\omega_i\otimes\id)\Delta_{\whG} \otimes\id) (\gamma\otimes\id) x, \omega' \ra
= \la 
(\Theta^l(a_i)\otimes\id) (\gamma\otimes\id)x , \omega'\ra.
\end{align*}
Using stable point-weak$^*$-convergence we obtain
\[
\la b_i,\Omega_{x,\omega}\ra \xrightarrow[i\in I]{}
\la (\gamma\otimes \id)x,\omega'\ra =\la x, \omega\ra =
\la \I,\Omega_{x,\omega}\ra.
\]
Hence $b_i\xrightarrow[i\in I]{w^*}\I$ showing that $(b_i)_{i\in I}$ witnesses that $\HH$ has the AP.
\end{proof}

This argument also yields the analogous claim for weak amenability, which we record for completeness.

\begin{proposition}
Let $\HH,\GG$ be locally compact quantum groups and assume that $\HH$ is a closed quantum subgroup of $\GG$ in the sense of Vaes. If $\GG$ is weakly amenable, then so is $\HH$, and $\Lambda_{cb}(\HH) \leq \Lambda_{cb}(\GG)$.
\end{proposition}

\begin{proof}
We follow the same proof strategy, using also Proposition~\ref{prop7}.  So, let $(a_i)_{i\in I}$ be a net in $\A(\GG)$ with $\|a_i\|_{cb} \leq \Lambda_{cb}(\GG)$ for each $i$, such that $(\Theta^l(a_i))_{i\in I}$ converges to the identity in the stable point-weak$^*$-topology of $\CB^\sigma(\Linfd)$.  Again let $a_i=\lambda_{\whG}(\omega_i)$, and define $b_i=\lambda_{\whH}(\gamma_*(\omega_i))$.  By the previous proof, and using Proposition~\ref{prop7}, it suffices to show that $\|b_i\|_{cb}  \leq \Lambda_{cb}(\GG)$ for each $i\in I$.

Let $\msf H$ be a Hilbert space, and let $x \in \LL^\infty(\wh\HH) \bar\otimes \B(\msf H)$ and $\omega\in \LL^1(\wh\HH) \wh\otimes \B(\msf H)_*$.  Given $\eps>0$, we can find $\omega'\in \LL^1(\whG)\wh{\otimes}\B(\msf{H})_*$ with $\omega=(\gamma_*\otimes \id)\omega'$ and with $\|\omega'\| \le \|\omega\|+\eps$.  Then, as before,
\begin{align*}
\la (\Theta^l(b_i) \otimes \id)x, \omega \ra
&= \la ((\gamma_*(\omega_i)\otimes\id)\Delta_{\wh\HH}\otimes\id)x, \omega \ra
= \la ((\omega_i\otimes\id)\Delta_{\wh\GG}\otimes\id)(\gamma\otimes\id)x, \omega' \ra \\
&= \la (\Theta^l(a_i)\otimes\id)(\gamma\otimes\id)x, \omega' \ra.
\end{align*}
It follows that
\[ \big| \la (\Theta^l(b_i) \otimes \id)x, \omega \ra \big|
\leq \|a_i\|_{cb} \|x\| \|\omega'\|. \]
As $\eps>0$ was arbitrary, and taking the supremum over functionals $\omega$ with $\|\omega\|=1$, we obtain that $\|(\Theta^l(b_i) \otimes \id)x\| \leq \|a_i\|_{cb} \|x\|$.  As $x, \msf H$ were arbitrary, this shows that $\|b_i\|_{cb} \leq \|a_i\|_{cb} \leq \Lambda_{cb}(\GG)$, as required.
\end{proof}

\subsection{Direct limits} 

In this section we show that AP is preserved by taking direct limits of discrete quantum groups obtained from directed systems with injective connecting maps. 
The corresponding fact for classical groups is certainly known, but we were not able to locate a reference. 

Let us first recall some facts about quantum subgroups of discrete quantum groups, using the notation and terminology from Section~\ref{section discrete}. 
In the discrete setting there is no difference between closed quantum subgroups in the sense of Vaes, closed quantum subgroups in the sense of 
Woronowicz \cite[Theorem 6.2]{DKSS_ClosedSub}, and open quantum subgroups in the sense of \cite{KKS_open}. 
We will therefore simply speak of quantum subgroups of discrete quantum groups in the sequel. 

Let $ \GGamma, \LLambda $ be discrete quantum groups and assume that $ \LLambda $ is a quantum subgroup of $\GGamma $. 
Then one can identify $\Irr(\wh\LLambda)$ with a subset of $\Irr(\wh\GGamma)$, and one 
obtains a corresponding identification of $\LL^2(\wh\LLambda)$ with a subspace of $\LL^2(\wh\GGamma)$. 
Let $p\in \ell^{\infty}(\GGamma)\subseteq \B(\LL^2(\wh\GGamma))$ be the projection onto $\LL^2(\wh\LLambda)$. Then $ p $ is a group-like 
projection (i.e.~$p$ is a central projection in $\ell^{\infty}(\bbGamma)$ satisfying $\Delta_{\GGamma}(p)(\I\otimes p)=p\otimes p$, see \cite[Definition 4.1]{KKS_open}) and the strong quantum homomorphism $\pi\colon \mrm{c}_0(\GGamma)\rightarrow \mrm{c}_{0}(\LLambda)$ associated with the inclusion of $ \LLambda$ into $\GGamma$
is given by $ \pi(f) = fp $. Dually, we have an injective, normal, 
unital $\star$-homomorphism $\iota \colon \LL^{\infty}(\wh\LLambda)\rightarrow \LL^{\infty}(\wh\GGamma)$ which respects the coproducts. 
The map $\iota$ restricts to injective $\star$-homomorphisms $\mrm{C}(\wh\LLambda)\rightarrow\mrm{C}(\wh\GGamma)$ and $\Pol(\wh\LLambda)\rightarrow\Pol(\wh\GGamma)$. 

The following fact is well-known, compare for instance \cite[Section 2]{Vergnioux_kamenability}.  Using $\iota$ we can view $\LL^{\infty}(\wh\LLambda)$ as a subalgebra of $\LL^{\infty}(\wh\GGamma)$, and so in the following, make sense of $\EE$ being a conditional expectation in the usual sense of a contractive projection onto a subalgebra.

\begin{lemma}\label{lem:cond_exp_discrete_subgp}
The formula $ \LL^{\infty}(\wh\GGamma) \ni x \mapsto p x p \in \B(p \LL^2(\wh\GGamma)) = \B(\LL^2(\wh\LLambda)) $ defines a normal conditional 
expectation $\EE\colon \LL^{\infty}(\wh\GGamma)\rightarrow \LL^{\infty}(\wh\LLambda)$ 
satisfying $\EE(U^{\alpha}_{i,j})=0$ for $\alpha\in \Irr(\wh\GGamma)\setminus \Irr(\wh\LLambda),\, 1\le i,j\le \dim(\alpha)$. 
Furthermore, $\EE$ restricts to a conditional expectation $\mrm{C}(\wh\GGamma)\rightarrow \mrm{C}(\wh\LLambda)$.
\end{lemma}

We shall be interested in directed systems of discrete quantum groups in the following sense. 

\begin{definition}\label{def2}
Let $ I $ be a directed set. A \emph{directed system of discrete quantum groups with injective connecting maps} is a family of discrete quantum groups $ (\GGamma_i)_{i \in I} $ together 
with injective unital normal $\star$-homomorphisms 
\[
\iota_{j,i}\colon \LL^{\infty}(\wh{\GGamma}_i)\rightarrow\LL^{\infty}(\wh{\GGamma}_j)\qquad(i,j\in I\colon i\le j),
\]
compatible with coproducts, such that 
\begin{itemize}
\item $ \iota_{i,i} = \id $ for $i\in I$,
\item $ \iota_{k,j} \iota_{j,i} = \iota_{k,i} $ for all $i,j,k\in I$ satisfying $ i \leq j \leq k $. 
\end{itemize}
\end{definition}

If $ (\GGamma_i)_{i \in I} $ is a directed system of discrete quantum groups with injective connecting maps then $\GGamma_i$ is a quantum subgroup of $\GGamma_j$ for $i\leq j$, and we have injective maps $\Pol(\wh\GGamma_i)\rightarrow\Pol(\wh\GGamma_j)$.
The algebraic direct limit $ \varinjlim_{i\in I} \Pol(\wh{\GGamma}_i) $ becomes naturally a unital Hopf $ * $-algebra, equipped with an invariant faithful state 
induced by the Haar integrals of $ \wh{\GGamma}_i $. 
We therefore have $ \varinjlim_{i \in I} \Pol(\wh{\GGamma}_i) = \Pol(\wh{\GGamma}) $ for a uniquely determined discrete quantum group $ \GGamma $, see for example \cite[Chapter 11, Theorem 27]{KS}.  We denote $ \GGamma = \varinjlim_{i\in I} \GGamma_i $ and call this the \emph{direct limit} of the directed system $ (\GGamma_i)_{i \in I} $. 

\begin{proposition}\label{prop8}
Let $ (\GGamma_i)_{i \in I} $ be a directed system of discrete quantum groups with injective connecting maps and let $ \GGamma $ be its associated direct limit. 
If $\GGamma_i $ has (central) AP for all $ i \in I $, then $ \GGamma $ has (central) AP. 
\end{proposition} 

\begin{proof} 
By construction each $\GGamma_i$ is a quantum subgroup of $\GGamma$.  Consequently we obtain injective normal $\star$-homomorphisms $ \iota_i: \LL^\infty(\wh{\GGamma}_i) \rightarrow \LL^\infty(\wh{\GGamma}) $, and normal conditional expectations $\EE_i\colon \LL^{\infty}(\wh\GGamma)\rightarrow \LL^{\infty}(\wh\GGamma_i)$ for all $i\in I$.

Identifying $\Irr(\wh{\GGamma}_i)$ with a subset of $\Irr(\wh\GGamma)$ gives us the extension by zero $*$-homomorphism map $\rho_i\colon\ell^{\infty}(\GGamma_i)\rightarrow\ell^{\infty}(\GGamma)$.  For $\omega\in \LL^1(\wh\GGamma)$ we have $\omega\circ\iota_i \in \LL^1(\wh\GGamma_i)$ and so $\lambda_{\wh\GGamma_i}(\omega\circ\iota_i) \in \ell^\infty(\GGamma_i)$.  We see that $\rho_i\lambda_{\wh\GGamma_i}(\omega\circ\iota_i)$ agrees with $\lambda_{\wh\GGamma}(\omega) \in \ell^\infty(\GGamma)$ restricted to $\ell^\infty(\GGamma_i)$ and set to zero in the remaining matrix blocks.  Similarly, for $\omega\in \LL^1(\wh\GGamma_i)$, by normality, $\omega\circ \EE_i \in \LL^1(\wh\GGamma)$, and as $\EE_i(U^\alpha_{i,j})=0$ for $\alpha\not\in\Irr(\wh\GGamma_i)$, see Lemma~\ref{lem:cond_exp_discrete_subgp}, it follows that $\lambda_{\wh\GGamma}(\omega\circ \EE_i) = \rho_i \lambda_{\wh\GGamma_i}(\omega)$.

We claim that $ \rho_i $ restricts to a contraction $ \M_{cb}^l(\A(\GGamma_i))\rightarrow \M_{cb}^l(\A(\GGamma)) $. Indeed, take $a\in \M^l_{cb}(\A(\GGamma_i))$ and $\omega\in \LL^1(\wh\GGamma)$.  Using the observations from the previous paragraph,
\begin{align*}
\rho_i(a) \lambda_{\wh\GGamma}(\omega)
&= \rho_i \bigl( a \lambda_{\wh\GGamma_i}( \omega\circ \iota_i )\bigr) \\
&= \rho_i \bigl(  \lambda_{\wh\GGamma_i}( \Theta^l(a)_*(\omega\circ \iota_i ))\bigr)
= \rho_i \bigl(  \lambda_{\wh\GGamma_i}( \omega\circ \iota_i\circ \Theta^l(a) )\bigr) \\
&= \lambda_{\wh\GGamma}\bigl(\omega\circ\iota_i \circ \Theta^l(a)\circ \EE_i\bigr).
\end{align*}
It follows that $\rho_i(a)\in \M^l_{cb}(\A(\GGamma))$ and
\[
\Theta^l(\rho_i(a))=\iota_i\circ \Theta^l(a)\circ \EE_i\in \CB^\sigma(\LL^{\infty}(\wh\GGamma)), 
\]
which yields the claim. 
By the definition of $\rho_i$ it is clear that $\rho_i^*(\ell^1(\GGamma)) \subseteq \ell^1(\GGamma_i) \subseteq Q^l(\A(\GGamma_i))$, which shows that the induced map $\rho_i: \M_{cb}^l(\A(\GGamma_i))\rightarrow \M_{cb}^l(\A(\GGamma))$ is weak$^*$-weak$^*$-continuous.

If $ \GGamma_i $ has AP then the identity element $ \I \in \M_{cb}^l(\A(\GGamma_i)) $ is in the weak$^*$-closure of $ \mrm{c}_{00}(\GGamma_i)$ inside $\M_{cb}^l(\A(\GGamma_i))$.  As $\rho_i$ is weak$^*$-weak$^*$-continuous, it follows that $\rho_i(\I)$ is contained in the weak$^*$-closure of $\rho_i( \mrm{c}_{00}(\GGamma_i) ) \subseteq \mrm{c}_{00}(\GGamma)$.  So $ p_i = \rho_i(\I) $, the projection corresponding to $ \Irr(\wh\GGamma_i) \subseteq \Irr(\wh\GGamma) $, is contained in the weak$^*$-closure of $ \mrm{c}_{00}(\GGamma) $ inside $\M_{cb}^l(\A(\GGamma))$. 
Clearly we have $ \la p_i,\omega\ra  \xrightarrow[i\in I]{} \la \I,\omega \ra  $ for all $ \omega \in \ell^1(\GGamma) $.  
Moreover we have $ \|p_i\|_{cb} = 1 $ since $p_i\in \M^l_{cb}(\A(\GGamma))$ with $\Theta^l(p_i)= \iota_i \circ \EE_i$. 
Hence if all $\GGamma_i$ have AP we see that $ \I \in \M_{cb}^l(\A(\GGamma)) $ is contained in the weak$^*$-closure of $ \mrm{c}_{00}(\GGamma) $. This means that $\GGamma$ has AP.

If all $\GGamma_i$ have central AP then we additionally know that $\I\in \M^l_{cb}(\A(\GGamma_i))$ is in the weak$^*$-closure of $\mc{Z}(\ell^{\infty}(\GGamma_i))\cap \mrm{c}_{00}(\GGamma_i)$, hence each $p_i$ is in the weak$^*$-closure of $\mc{Z}(\ell^{\infty}(\GGamma))\cap \mrm{c}_{00}(\GGamma)$. It follows that $\GGamma$ has central AP.
\end{proof}

\subsection{Free products}

In this section we show that AP is preserved by the free product construction for discrete quantum groups. For classical groups this fact is probably known to experts, but 
we could not find a proof in the literature. Our proof is based on results of Ricard and Xu from \cite{RicardXu} which we recall first.

\subsubsection{Ricard-Xu results}\label{sec:RX}
Let $(A_i,\phi_i)_{i\in I}$ be a family of unital $\cst$-algebras with faithful states indexed by some set $I$.  Denote by $\msf{H}_i$ the GNS Hilbert space for $\phi_i$, and by $\msf{H}_i^{\oon{op}}$ the Hilbert space obtained from $A_i$ by completion with respect to the norm given by $a\mapsto \phi_i(aa^*)^{1/2}$. Then $a\mapsto a^*$ extends to an antilinear isometry $\msf{H}_i\rightarrow \msf{H}_i^{\oon{op}}$.

We write $\mc{A}=\star_{i\in I} (A_i,\phi_i)$ for the \emph{reduced unital free product} of the family $(A_i,\phi_i)_{i\in I}$, and $A\subseteq \mc{A}$ for its canonical dense 
unital $\star$-subalgebra (the algebraic unital free product), compare \cite{Avitzour}. Next, for $d\ge 0$, denote by $\Sigma_d\subseteq A$ the subspace of length-$d$ elements. 
That is, we have $\Sigma_0=\CC\I$, and if $\mathring{A}_i$ denotes the set of all $a\in A_i$ with $\phi_i(a)=0$, 
then $\Sigma_d$ for $d\ge 1$ is the subspace of $A$ spanned by all elements of the form $a_{1}\cdots a_{d}$ where $a_j\in \mathring{A}_{i_j}$ for each $j$, and with $i_j\neq i_{j+1}$ for $1\leq j<d$.
Moreover we let $\mc{A}_d \subseteq \mc A$ be the norm closure of $\Sigma_d$.

In the sequel we shall use two results from \cite{RicardXu}, the first one being the following.

\begin{lemma}{\cite[Corollary 3.3]{RicardXu}} \label{lemma1}
For $d\ge 0$, the natural projection $A\rightarrow \Sigma_d$ onto length-$d$ elements extends to a CB map $\mc{P}_d\colon \mc{A}\rightarrow \mc{A}_d$ with $\|\mc{P}_d\|_{cb}\le \max(4d,1)$.
\end{lemma}

The second fact which we need is a minor extension of \cite[Lemma 4.10]{RicardXu}.

\begin{lemma}\label{lemma2}
Fix $d\geq 1$.  For $i\in I$ and $1\le k \le d$, let $T_{i,k}\in \CB(A_i)$ be linear maps which satisfy $\phi_i\circ T_{i,k}=\lambda_{i,k} \phi_i$ for some $\lambda_{i,k}\in \CC$, and which extend to bounded maps $\msf{H}_i\rightarrow \msf{H}_i$ and $\msf{H}_i^{\oon{op}}\rightarrow \msf{H}_i^{\oon{op}}$. If
\[
K = (2d+1)\prod_{k=1}^{d}
\sup_i  \max\bigl(
\|T_{i,k}\|_{cb},\|T_{i,k}\|_{\B(\msf{H}_i)},\|T_{i,k}\|_{\B(\msf{H}_i^{\oon{op}})} 
\bigr) < \infty
\]
then the natural map $\Pi_k T_{i,k}\colon \Sigma_d\rightarrow \Sigma_d$ given by
\[
a_1\cdots a_d\mapsto 
T_{i_1,1}(a_1)\cdots T_{i_d,d}(a_d)\qquad(a_j\in \mathring{A}_{i_j},\; i_j\neq i_{j+1})
\]
extends to a CB map $\mc{A}_d\rightarrow \mc{A}_d$ with CB norm bounded above by $ K $. 
\end{lemma}
\begin{proof}
The only difference between this claim and \cite[Lemma 4.10]{RicardXu} is that \cite[Lemma 4.10]{RicardXu} has the stronger hypothesis that $\phi_i\circ T_{i,k} = \phi_i$ for each $i,k$.  A close examination of the proof of \cite[Lemma 4.10]{RicardXu} shows that this hypothesis is only used to ensure that the map $\Pi_k T_{i,k}$ is well-defined, because each $T_{i,k}$ maps $\mathring{A}_{i}$ to itself.  This condition remains true under our weaker hypothesis, and the rest of the proof of \cite[Lemma 4.10]{RicardXu} carries over without change.
\end{proof}

\subsubsection{AP for free products}

Let $\GGamma_1,\GGamma_2$ be discrete quantum groups and let $\GGamma=\GGamma_1\star\GGamma_2$ be their \emph{free product}. Recall from \cite{SWang} that this means in particular 
that $ \mc{A}=\mrm{C}(\wh\GGamma)$ is the unital reduced free product $\mrm{C}(\wh\GGamma_1)\star\mrm{C}(\wh\GGamma_2)$ with respect to Haar integrals, 
and $h_{\wh\GGamma}$ is the free product state $h_{\wh\GGamma_1}\star h_{\wh\GGamma_2}$. Moreover $\Pol(\wh\GGamma)$ is the algebraic unital free product 
of $\Pol(\wh\GGamma_1)$ and $\Pol(\wh\GGamma_2)$. The irreducible representations of $\wh\GGamma$ are given as follows, see \cite[Theorem~3.10]{SWang}.  Each $\alpha\in \Irr(\wh\GGamma)$ has a well-defined length $\oon{len}(\alpha)\in \ZZ_+$. The trivial representation is the only representation of length $0$, and for $n\ge 1$ we have
\[
\{\alpha\in \Irr(\wh\GGamma)\,|\,\oon{len}(\alpha)=n\}=
\{ \alpha_{i_1}\boxtimes \cdots\boxtimes \alpha_{i_n}\,|\,
\forall_{1\le j\le n}\,\alpha_{i_j}\in \Irr(\wh\GGamma_{i_j})\setminus\{e\},\;
\forall_{1\le j<n}\,i_{j}\neq i_{j+1}\}.
\]
Again, here we denote by $e$ the trivial representation of a compact quantum group.
More explicitly, given $\alpha\in \Irr(\wh\GGamma_k)$ associated to the representation matrix $U^\alpha = [U^\alpha_{i,j}]_{i,j=1}^{\dim(\alpha)} \in \M_{\dim(\alpha)}(\mrm{C}(\wh\GGamma_k))$, by regarding $\mrm{C}(\wh\GGamma_k)$ as a subalgebra of $\mrm{C}(\wh\GGamma)$, we may regard $U^\alpha$ as a representation of $(\mrm{C}(\wh\GGamma),\Delta_{\wh\GGamma})$.  Then $\boxtimes$ is just the usual tensor product of representations.

To ease notation, we will write $h=h_{\wh\GGamma}$ in the sequel. 

\begin{theorem}\label{thm1}
Let $\GGamma_1,\GGamma_2$ be discrete quantum groups and let $\GGamma=\GGamma_1\star\GGamma_2$ be their free product. If $\GGamma_1,\GGamma_2$ have (central) AP, then $\GGamma$ has (central) AP.
\end{theorem}

Before we can prove Theorem \ref{thm1} we need to establish some auxiliary results. For $d\in\NN$ let us define the (non-linear) map
\begin{equation}\label{eq5}
\tilde{\Psi}_d \colon \bigoplus_{k=1}^{d} \ell^{\infty}(\GGamma_1)\oplus_\infty \ell^{\infty}(\GGamma_2)\rightarrow \ell^{\infty}(\GGamma)
\end{equation}
(here $\oplus$ is the $\ell^{\infty}$-direct sum) via
\[\begin{split}
\tilde{\Psi}_d(
(g_{1,k},g_{2,k})_{k=1}^{d} )&=
\bigl(
\tilde{\Psi}_d(
(g_{1,k},g_{2,k})_{k=1}^{d} )_\alpha\bigr)_{\alpha\in\Irr(\wh\GGamma)}, \qquad\text{where} \\
\tilde{\Psi}_d((g_{1,k},g_{2,k})_{k=1}^{d})_\alpha&=
\begin{cases}
0,
& \textnormal{len}(\alpha)\neq d,\\
g_{i_1,1,\alpha_{1}}\otimes\cdots\otimes 
g_{i_d,d,\alpha_{d}},
& \alpha=\alpha_{1}\boxtimes\cdots\boxtimes\alpha_{d}\colon \alpha_{j}\in \Irr(\wh\GGamma_{i_j}).
\end{cases}
\end{split}\]
We consider
\[
\mc{V}= \bigoplus_{k=1}^{d} \M^l_{cb}(\A(\GGamma_1))\oplus _\infty
\M^l_{cb}(\A(\GGamma_2)),
\]
and write $\Psi_d$ for the restriction of $\tilde{\Psi}_d$ to $\mc{V}$. Recall that $\mc{P}_d: \mc{A}\rightarrow \mc{A}_d $ is induced by the projection onto elements of length $d$. 

\begin{lemma}\label{lemma3}
The image of $\Psi_d$ is a subset of $\M^l_{cb}(\A(\GGamma))$, that is, we can regard $\Psi_d$ as a map $\mc{V}\rightarrow \M^l_{cb}(\A(\GGamma))$.
Furthermore, $\mc{P}_d$ extends to a weak$^*$-weak$^*$-continuous CB map $\LL^{\infty}(\wh\GGamma)\rightarrow \ov{\mc{A}_d}^{w^*}$.
\end{lemma}

\begin{proof}
Fix $(g_{1,k},g_{2,k})_{k=1}^{d}\in \mc{V}$. 
Proposition \ref{prop2} shows that each $\Theta^l(g_{i,k})$ extends to a bounded linear map on $\LL^2(\wh\GGamma_i)$ with norm $\|S_{\GGamma_i}^{-1}(g_{i,k})\|$.
Next, using Proposition \ref{prop12} and Proposition \ref{prop2}, for $x\in \LL^{\infty}(\wh\GGamma)$, we have
\[\begin{split}
&\quad\;
\| \Theta^l(g_{i,k})(x)\|_{\LL^2(\wh\GGamma_i)^{\oon{op}}}=
\| \Theta^l(g_{i,k})(x)^*\|_2=
\| \Theta^l(g_{i,k})^{\dagger}(x^*)\|_2=
\|\Theta^l( S_{\GGamma_i}(g_{i,k}^*) )(x^*)\|_2\\
&\le 
\|S_{\GGamma_i}^{-1} (S_{\GGamma_i}(g_{i,k}^*))\|
\|x^*\|_2=
\|g_{i,k}^*\| \,\|x\|_{\LL^2(\wh\GGamma_i)^{\oon{op}}}
=
\|g_{i,k}\| \,\|x\|_{\LL^2(\wh\GGamma_i)^{\oon{op}}}.
\end{split}\]
Hence $\Theta^l(g_{i,k})$ extends to a bounded linear map on $\LL^2(\wh\GGamma_i)^{\oon{op}}$ with norm bounded by $\|g_{i,k}\|$.

Let us consider $T_{i,k}=\Theta^l(g_{i,k}) \in \CB(\mrm{C}(\wh\GGamma_i))$ for $1\le k\le d$. We have $h_{\wh\GGamma_i} \circ T_{i,k}=g_{i,k,e} h_{\wh\GGamma_i}$. Then, according to Lemma \ref{lemma1} and Lemma \ref{lemma2}, we obtain a $\CB$ map $\Upsilon$ on $\mrm{C}(\wh{\GGamma})$ acting by $0$ on elements of length $d'\neq d$, and on elements of length $d$ by
\[
a_{1}\cdots a_{d}\mapsto 
\Theta^l(g_{i_1,1})(a_{1})\cdots  
\Theta^l(g_{i_d,d})(a_{d}),
\]
where $a_{j}\in \mrm{C}(\wh\GGamma_{i_j})$, and the $\CB$ norm of $\Upsilon$ is bounded above by
\[
4d (2d+1) \prod_{k=1}^{d} \max_{i\in\{1,2\}}\max(\|g_{i,k}\|_{cb}, \|S^{-1}_{\GGamma_i}(g_{i,k})\|, \|g_{i,k}\|).
\]
Since $\|\cdot\|\le \|\cdot\|_{cb}$ on $\M^l_{cb}(\A(\GGamma_i))$ we have, using Lemma \ref{lemma18} and Proposition \ref{prop12}, 
\[
\|S_{\GGamma_i}^{-1}(g_{i,k})\|=
\|S_{\GGamma_i}^{-1}(g_{i,k})^*\|\le 
\|S_{\GGamma_i}^{-1}(g_{i,k})^*\|_{cb}=
\|g_{i,k}\|_{cb},
\]
and hence we get in fact 
\begin{equation}\label{eq4}
\|\Upsilon\|_{cb}\le 4d(2d+1) \prod_{k=1}^{d} \max_{i\in \{1,2\}} 
\|g_{i,k}\|_{cb}.
\end{equation}

We claim that $\Upsilon$ extends to a normal map on $\LL^{\infty}(\wh\GGamma)$. For this it suffices to show that $\Upsilon^*$ preserves $\LL^1(\wh\GGamma)\subseteq \mrm{C}(\wh\GGamma)^*$. Indeed, if this is the case, then the extension may be defined as $(\Upsilon^*|_{\LL^1(\wh\GGamma)})^*\in \CB^{\sigma}(\LL^{\infty}(\wh\GGamma))$. 

Thus take $\rho\in \LL^1(\wh\GGamma)$. Since $\Upsilon$ is bounded and $\LL^1(\wh\GGamma)\subseteq \mrm{C}(\wh\GGamma)^*$ is norm-closed, it is enough to consider $\rho=h ( a\cdot )$ for $a\in \LL^{\infty}(\wh\GGamma)$. Take $b\in\LL^{\infty}(\wh\GGamma)$ and denote by $\Upsilon_2$ the extension of $\Upsilon$ to a bounded linear map on $\LL^2(\wh\GGamma)$. 
Note that this extension exists since the GNS Hilbert space for $h $ is
\[
\LL^2(\wh\GGamma)=\CC \Omega\oplus \bigoplus_{d=1}^{\infty} \bigoplus_{i_1\neq \cdots \neq i_d} \LL^2(\wh\GGamma_{i_1})^{\circ}\otimes\cdots\otimes \LL^2(\wh\GGamma_{i_d})^{\circ},
\]
where $\LL^2(\wh\GGamma_i)^{\circ}$ is the subspace of $\LL^2(\wh\GGamma_i)$ orthogonal to $\Lambda_{h_i}(\I)$ (see \cite[Section 2]{Avitzour}), 
and $\Theta^l(g_{i,k})$ has bounded extension to $\LL^2(\wh\GGamma_i)$ by Proposition \ref{prop2}. We have
\[
\Upsilon^*(\rho)(b)=h(a \Upsilon(b) )= \ismaa{\Lambda_h(a^*) }{\Lambda_h(\Upsilon(b))}=
\ismaa{\Lambda_h(a^*)}{\Upsilon_2\Lambda_h(b)}=\ismaa{\Upsilon_2^* \Lambda_h(a^*)}{\Lambda_h(b)}. \]
Hence $\Upsilon^*(\rho) = \omega_{\Upsilon_2^* \Lambda_h(a^*), \Lambda_h(\I)} \in \LL^1(\wh\GGamma)$.
Let us denote the resulting normal extension of $\Upsilon$ to $\LL^{\infty}(\wh\GGamma)$ with the same symbol.

In particular, taking $ g_{i,k} = \I $ for all $1\leq k\leq d$ in the above discussion shows that the projection $\mc{P}_d: \mrm{C}(\wh\GGamma) \rightarrow \mc{A}_d$ extends to a normal CB map $\LL^{\infty}(\wh\GGamma)\rightarrow \ov{\mc{A}_d}^{w^*}$.

We finally identify $\Upsilon$ with the adjoint of a centraliser, namely we claim that $\Psi_d((g_{1,k},g_{2,k})_{k=1}^{d})\in \M^l_{cb}(\A(\GGamma))$ and $\Upsilon=\Theta^l( \Psi_d((g_{1,k},g_{2,k})_{k=1}^{d}) )$.  For $\alpha=\alpha_{1}\boxtimes\cdots\boxtimes \alpha_{d}\in \Irr(\wh\GGamma)$ let us write $i_j\in \{1,2\}$ for indices such that $\alpha_{j}\in \Irr(\wh\GGamma_{i_j})\setminus\{e\},\, i_j\neq i_{j+1}$. Furthermore, write each matrix block as $g_{i,k,\alpha}=[g_{i,k,\alpha,m,n}]_{m,n=1}^{\dim(\alpha)}=\sum_{m,n=1}^{\dim(\alpha)} g_{i,k,\alpha,m,n}e^{\alpha}_{m,n}$, where $\{e^{\alpha}_{m,n}\}_{m,n=1}^{\dim(\alpha)}$ are the matrix units in $\B(\msf{H}_\alpha)$. Choose arbitrary $\omega\in h(\Pol(\wh\GGamma)\cdot)\subseteq \LL^1(\wh\GGamma)$. We can calculate $\Psi_d( (g_{1,k},g_{2,k})_{k=1}^{d}) \lambda_{\wh\GGamma}(\omega)$ as follows:

\begin{equation*}\begin{split}\label{eq21}
&\quad\;
\Psi_d( (g_{1,k},g_{2,k})_{k=1}^{d}) \lambda_{\wh\GGamma}(\omega)\\
&=
\sum_{d'=0}^{\infty}
\sum_{\alpha=\alpha_1\boxtimes\cdots\boxtimes\alpha_{d'}}
\sum_{m_1,n_1=1}^{\dim(\alpha_1)}
\!\!\!\cdots\!\!\!
\sum_{m_{d'},n_{d'}=1}^{\dim(\alpha_{d'})}\!\!
\la U^{\alpha}_{(m_1,\dotsc,m_{d'}),(n_1,\dotsc,n_{d'})},\omega\ra \\
&\qquad\qquad\qquad\qquad
\qquad\qquad\qquad\qquad
\qquad\qquad
\Psi_d( (g_{1,k},g_{2,k})_{k=1}^{d})
(e^{\alpha_1}_{m_1,n_1}\otimes\cdots\otimes 
e^{\alpha_{d'}}_{m_{d'},n_{d'}})\\
&=
\sum_{\alpha=\alpha_1\boxtimes\cdots\boxtimes\alpha_{d}}
\sum_{m_1,n_1=1}^{\dim(\alpha_1)}
\!\!\!\cdots\!\!\!
\sum_{m_{d},n_{d}=1}^{\dim(\alpha_{d})}\!\!
\la U^{\alpha}_{(m_1,\dotsc,m_{d}),(n_1,\dotsc,n_{d})},\omega\ra 
(g_{i_1,1,\alpha_1}e^{\alpha_1}_{m_1,n_1}\otimes\cdots\otimes 
g_{i_d,d,\alpha_d}e^{\alpha_{d}}_{m_{d},n_{d}})\\
&=
\sum_{\alpha=\alpha_1\boxtimes\cdots\boxtimes\alpha_{d}}
\sum_{m_1,k_1,n_1=1}^{\dim(\alpha_1)}
\!\!\!\cdots\!\!\!
\sum_{m_{d},k_d,n_{d}=1}^{\dim(\alpha_{d})}\!\!
\la U^{\alpha_1}_{m_1,n_1}\cdots U^{\alpha_d}_{m_d,n_d},\omega\ra \\
&\qquad\qquad\qquad\qquad
\qquad\qquad\qquad\qquad
\qquad\qquad
(g_{i_1,1,\alpha_1,k_1,m_1}e^{\alpha_1}_{k_1,n_1}\otimes\cdots\otimes 
g_{i_d,d,\alpha_d,k_d,m_d}e^{\alpha_{d}}_{k_{d},n_{d}}),
\end{split}\end{equation*}
note that as $\omega\in h(\Pol(\wh\GGamma)\cdot)$, the sums above are finite. On the other hand, using Lemma \ref{lemma4}, we have
\begin{equation*}\begin{split}\label{eq22}
&\quad\;\lambda_{\wh\GGamma}(\Upsilon_*(\omega))\\
&=
\sum_{d'=0}^{\infty}
\sum_{\alpha=\alpha_1\boxtimes\cdots\boxtimes\alpha_{d'}}
\sum_{m_1,n_1=1}^{\dim(\alpha_1)}
\!\!\!\cdots\!\!\!
\sum_{m_{d'},n_{d'}=1}^{\dim(\alpha_{d'})}\!\!
\la U^{\alpha}_{(m_1,\dotsc,m_{d'}),(n_1,\dotsc,n_{d'})},\Upsilon_*(\omega)\ra 
(e^{\alpha_1}_{m_1,n_1}\otimes\cdots\otimes 
e^{\alpha_{d'}}_{m_{d'},n_{d'}})\\
&=
\sum_{\alpha=\alpha_1\boxtimes\cdots\boxtimes\alpha_{d}}
\sum_{m_1,n_1=1}^{\dim(\alpha_1)}
\!\!\!\cdots\!\!\!
\sum_{m_{d},n_{d}=1}^{\dim(\alpha_{d})}\!\!
\la \Theta^l(g_{i_1,1})( U^{\alpha_1}_{m_1,n_1})\cdots \Theta^l(g_{i_d,d})(U^{\alpha_d}_{m_d,n_d}), \omega\ra 
(e^{\alpha_1}_{m_1,n_1}\otimes\cdots\otimes 
e^{\alpha_{d}}_{m_{d},n_{d}})\\
&=\!\!
\sum_{\alpha=\alpha_1\boxtimes\cdots\boxtimes\alpha_{d}}
\sum_{m_1,k_1,n_1=1}^{\dim(\alpha_1)}
\!\!\!\cdots\!\!\!
\sum_{m_{d},k_d,n_{d}=1}^{\dim(\alpha_{d})}\!\!
\la ( g_{i_1,1,m_1,k_1} U^{\alpha_1}_{k_1,n_1})\!\cdots \!(g_{i_d,d,m_d,k_d} U^{\alpha_d}_{k_d,n_d}), \omega\ra 
(e^{\alpha_1}_{m_1,n_1}\otimes\!\cdots\!\otimes 
e^{\alpha_{d}}_{m_{d},n_{d}})
\end{split}\end{equation*}
These two computations show that $\Psi_d((g_{1,k},g_{2,k})_{k=1}^{d})\lambda_{\wh\GGamma}(\omega)=\lambda_{\wh\GGamma}(\Upsilon_*(\omega))$. As the space of functionals $h(\Pol(\wh\GGamma)\cdot)$ is dense in $\LL^1(\wh\GGamma)$, this proves the claim.
\end{proof}

Define
\[
\mc{V}_0= \M^l_{cb}(\A(\GGamma_1))\oplus_\infty
\M^l_{cb}(\A(\GGamma_2))
\]
so that $\mc{V}=\bigoplus_{k=1}^{d}\mc{V}_0$.

Note that $\mc{V}=\bigoplus_{k=1}^{d} \M^l_{cb}(\A(\GGamma_1))\oplus_\infty \M^l_{cb}(\A(\GGamma_2))$ is a dual Banach space with predual given by 
the $\ell^1$-direct sum $\bigoplus_{k=1}^{d} Q^l(\A(\GGamma_1)) \oplus_1 Q^l(\A(\GGamma_2))$. Similarly, $\mc V_0$ is a dual Banach space with predual $Q^l(\A(\GGamma_1)) \oplus_1 Q^l(\A(\GGamma_2))$.

\begin{lemma}\label{lemma5}
Map $\Psi_d$ is separately weak$^*$-weak$^*$-continuous, i.e.~for any $1\le k \le d$ and fixed elements $(g_{1,k'},g_{2,k'})\in \mc{V}_{0}$ for $1\le k' \le d$ and $k'\neq k$, the map
\[
\mc{V}_0\ni (g_{1,k} , g_{2,k})\mapsto
\Psi_d( (g_{1,k'}, g_{2,k'})_{k'=1}^{d})\in \M^l_{cb}(\A(\GGamma))
\]
is weak$^*$-weak$^*$-continuous.
\end{lemma}

\begin{proof}
Fix $1\le k \le d$ and $(g_{1,k'},g_{2,k'})\in \mc{V}_{0}$ for $1\le k' \le d$ and $k'\neq k$. Assume that 
\[
(g_{1,k}^\lambda , g_{2,k}^\lambda)\xrightarrow[\lambda\in \Lambda]{w^*}
(g_{1,k},g_{2,k})\quad\textnormal{ in }\quad\mc{V}_0=\M^l_{cb}(\A(\GGamma_1))\oplus_\infty \M^l_{cb}(\A(\GGamma_2)),
\]
in particular
 \[
(g_{1,k}^\lambda , g_{2,k}^\lambda)\xrightarrow[\lambda\in \Lambda]{w^*}
(g_{1,k},g_{2,k})\quad\textnormal{ in }\quad \ell^{\infty}(\GGamma_1)\oplus_\infty\ell^{\infty}(\GGamma_2).
\]
Take $\Omega\in Q^l(\A(\GGamma))$.
Since $Q^l(\A(\GGamma))$ is the norm closure of $\ell^1(\GGamma)$ in $\M^l_{cb}(\A(\GGamma))^*$, we can find a sequence $(\Omega_n)_{n\in\NN}$ in $\ell^1(\GGamma)\subseteq Q^l(\A(\GGamma))$ which converges in norm to $\Omega$. By the description \eqref{eq5} of $\tilde{\Psi}_d$, we see that the map
  \[
\ell^{\infty}(\GGamma_1)\oplus_\infty\ell^{\infty}(\GGamma_2)\ni 
(g_{1,k}',g_{2,k}')\mapsto
\tilde{\Psi}_d \bigl(
(g_{1,k'},g_{2,k'})_{k'=1}^{k-1} , (g'_{1,k}, g_{2,k}'),
(g_{1,k'},g_{2,k'})_{k'=k+1}^{d}\bigr) \in \ell^{\infty}(\GGamma)  
  \]
is weak$^*$-weak$^*$-continuous, hence the linear functional
\[
 \Omega_n\circ 
\tilde{\Psi}_d \bigl(
(g_{1,k'},g_{2,k'})_{k'=1}^{k-1} ,\; \cdot\; ,
(g_{1,k'},g_{2,k'})_{k'=k+1}^{d}\bigr)
 \]
is contained in $\ell^1(\GGamma_1)\oplus_1 \ell^1(\GGamma_2)$ for each $n\in\NN$. Consider the difference
 \[\begin{split}
 \Omega_n&\circ 
\Psi_d \bigl(
(g_{1,k'},g_{2,k'})_{k'=1}^{k-1} ,\; \cdot\; ,
(g_{1,k'},g_{2,k'})_{k'=k+1}^{d}\bigr)\\
-
 \Omega&\circ 
\Psi_d \bigl(
(g_{1,k'},g_{2,k'})_{k'=1}^{k-1} ,\; \cdot\; ,
(g_{1,k'},g_{2,k'})_{k'=k+1}^{d}\bigr),
 \end{split}\]
living in $\M^l_{cb}(\A(\GGamma_1))^*\oplus_1 \M^l_{cb}(\A(\GGamma_2))^*$. Because of the bound \eqref{eq4} we can estimate 
 \[\begin{split}
&\quad\; \bigl\|
 (\Omega-\Omega_n)\circ \Psi_d \bigl(
(g_{1,k'},g_{2,k'})_{k'=1}^{k-1} ,\; \cdot\; ,
(g_{1,k'},g_{2,k'})_{k'=k+1}^{d}\bigr)
\bigr\|\\
&=
\sup_{(g'_{1,k},g'_{2,k})\in (\M^l_{cb}(\A(\GGamma_1))\oplus_\infty \M^l_{cb}(\A(\GGamma_2)))_1}\\
&\quad\quad\quad
\bigl| (\Omega-\Omega_n)\circ \Psi_d \bigl(
(g_{1,k'},g_{2,k'})_{k'=1}^{k-1} ,
(g'_{1,k},g'_{2,k}),
(g_{1,k'},g_{2,k'})_{k'=k+1}^{d}\bigr) \bigr|\\
&\le 
\|\Omega-\Omega_n\|\;
4d (2d+1) \prod_{k'=1, k'\neq k}^{d} \max_{i\in\{1,2\}}
\|g_{i,k'}\|_{cb}\xrightarrow[n\to\infty]{}0.
 \end{split}\]
This shows 
\[
 \Omega \circ \Psi_d \bigl(
(g_{1,k'},g_{2,k'})_{k'=1}^{k-1} ,\; \cdot\; ,
(g_{1,k'},g_{2,k'})_{k'=k+1}^{d}\bigr)\in Q^l(\A(\GGamma_1)) \oplus_1 Q^l(\A(\GGamma_2)),
\]
and hence
\[
\Omega\circ \Psi_d\bigl((g_{1,k'},g_{2,k'})_{k'=1}^{k-1} , (g_{1,k}^\lambda,g_{2,k}^\lambda), 
(g_{1,k'},g_{2,k'})_{k'=k+1}^{d}\bigr)\xrightarrow[\lambda\in \Lambda]{}
\Omega\circ \Psi_d\bigl((g_{1,k'},g_{2,k'})_{k'=1}^{d}\bigr)
\]
as desired. This proves that $\Psi_d$ is separately weak$^*$-weak$^*$-continuous.
\end{proof}

For $d\ge 1$ consider $p_d\in \ell^{\infty}(\GGamma)$ defined via $p_d=(p_{d,\alpha})_{\alpha\in\Irr(\wh\GGamma)}$ where
\[
p_{d,\alpha}=
\begin{cases}
0 ,& \textnormal{length of }\alpha \neq d,\\
\I, &
\textnormal{length of }\alpha = d.
\end{cases}
\]

\begin{lemma}
Projection $p_d$ belongs to $ \M^l_{cb}(\A(\GGamma))$ and $\Theta^l(p_d)=\mc{P}_d$. 
\end{lemma}

\begin{proof}
We already know that $\mc{P}_d$ is a weak$^*$-continuous map on $\LL^{\infty}(\wh\GGamma)$. Take a linear functional $\omega\in h(\Pol(\wh\GGamma)\,\cdot )\subseteq \LL^1(\wh\GGamma)$. Then 
we get 
\[\begin{split}
&\quad\;
(\omega\otimes\id)\bigl( (\I\otimes p_d)\ww^{\wh\GGamma}\bigr)=
\sum_{d'=1}^{\infty} \sum_{\alpha\in \Irr(\wh\GGamma):\; \oon{len}(\alpha)=d'}
\sum_{i,j=1}^{\dim(\alpha)}
(\omega\otimes\id)\bigl(
U^{\alpha}_{i,j}\otimes p_d\, e^{\alpha}_{i,j}\bigr)\\
&=
\sum_{\alpha\in \Irr(\wh\GGamma):\; \oon{len}(\alpha)=d}
\sum_{i,j=1}^{\dim(\alpha)}
(\omega\otimes\id)\bigl(
U^{\alpha}_{i,j}\otimes e^{\alpha}_{i,j}\bigr)\\
&=
\sum_{d'=1}^{\infty} \sum_{\alpha\in \Irr(\wh\GGamma):\; \oon{len}(\alpha)=d'}
\sum_{i,j=1}^{\dim(\alpha)}
(\omega\otimes\id)\bigl(
\mc{P}_d(U^{\alpha}_{i,j})\otimes e^{\alpha}_{i,j}\bigr)=
(\omega\circ \mc{P}_d\otimes\id) \ww^{\wh\GGamma}, 
\end{split}\]
noting that all sums in this calculation are finite, because of the form of $\omega$.  As such $\omega$ are dense, this yields $(\I\otimes p_d)\ww^{\wh\GGamma}=(\mc{P}_d\otimes\id) \ww^{\wh\GGamma}$ as required. 
\end{proof}

Finally, we are ready to prove that AP is preserved by taking free products of discrete quantum groups.

\begin{proof}[Proof of Theorem \ref{thm1}]
Assume that $\GGamma_1, \GGamma_2$ have AP and for $i\in\{1,2\}$ choose families $(f_{i,\lambda})_{\lambda\in \Lambda_i}$ in $\mrm{c}_{00}(\GGamma_i)$ converging to $\I$ in $(\M^l_{cb}(\A(\GGamma_i)),w^*)$. 
Due to Proposition \ref{prop1} we may assume without loss of generality that each $\Theta^l(f_{i,\lambda})$ is unit preserving.  As in Remark~\ref{rem:cent_unit_pres_means}, it then follows that $\Theta^l(f_{i,\lambda})$ preserves the Haar integral on $\wh\GGamma_i$, and that $f_{i,\lambda,e}=1$ for all $i\in \{1,2\}, \lambda\in \Lambda_i$.

Fix $d\in \NN$.  We shall first show that $p_d \in \ov{\mrm{c}_{00}(\GGamma)}^{w^*}\subseteq \M^l_{cb}(\A(\GGamma))$.  To do this, we will consider a net of the form
\[
( f_{1,\lambda_{1,k}} , f_{2,\lambda_{2,k}})_{k=1}^{d}\in \mc{V}.
\]
where each $\lambda_{i,k}\in \Lambda_i$ for $i\in\{1,2\}$ and $1\le k\le d$.
Lemma \ref{lemma5} gives us
\[
\Psi_d\bigl( 
( f_{1,\lambda_{1,k}} , f_{2,\lambda_{2,k}})_{k=1}^{d}\bigr)\in \M^l_{cb}(\A(\GGamma)).
\]
In fact, using the definition of $\tilde\Psi_d$, we see that these multipliers are in $\mrm{c}_{00}(\GGamma)$ since all of the $f_{i,\lambda}$ are finitely supported.

We first consider the case when we keep $\lambda_{i,k}$ fixed, for $k\ge 2$.  Since $\Psi_d$ is separately weak$^*$-weak$^*$-continuous, by Lemma~\ref{lemma5}, we have
\[
\Psi_d\bigl( (f_{1,\lambda_{1,1} } , f_{2, \lambda_{2,1}}) , 
( f_{1,\lambda_{1,k}} , f_{2,\lambda_{2,k}})_{k=2}^{d}\bigr)\xrightarrow[(\lambda_{1,1} , \lambda_{2,1} ) \in \Lambda_1\times \Lambda_2]{w^*}
\Psi_d\bigl( (\I,\I) , 
( f_{1,\lambda_{1,k}} , f_{2,\lambda_{2,k}})_{k=2}^{d}\bigr).
\]
We now repeat this argument in the second variable, and so forth, and using that $(\ov{ \mrm{c}_{00}(\GGamma)}^{w^*})^{- w^*}=\ov{\mrm{c}_{00}(\GGamma)}^{w^*}$, we obtain
\[
p_d=\Psi_d ( (\I,\I)_{k=1}^{d})\in \ov{\mrm{c}_{00}(\GGamma)}^{w^*}\subseteq \M^l_{cb}(\A(\GGamma)).
\]
Clearly we also have $p_0=(\delta_{\alpha,e}\I)_{\alpha\in\Irr(\wh\GGamma)}\in \mrm{c}_{00}(\GGamma) \subseteq \ov{ \mrm{c}_{00}(\GGamma)}^{w^*}$.

We now show that $\I\in \ov{\mrm{c}_{00}(\GGamma)}^{w^*}$.  Consider
\[
T_n=\sum_{d=0}^{n} (1-\tfrac{1}{\sqrt{n}})^d p_d\in 
\ov{ \mrm{c}_{00}(\GGamma)}^{w^*}\subseteq \M^l_{cb}(\A(\GGamma))\qquad(n\in\NN).
\] 
According to \cite[Proposition 3.5]{RicardXu}, $\lim_{n\to\infty} \|T_n\|_{cb}=1$ and $(\Theta^l(T_n))_{n\in\NN}$ converges pointwise to the identity on $\mrm{C}(\wh\GGamma)$. Take $x\in \mrm{C}(\wh\GGamma)\odot \mc{K}(\msf{H}), \omega \in \LL^1(\wh\GGamma)\odot \B(\msf{H})_*$ for a separable Hilbert space $\msf{H}$. Then we obtain
\[
\la T_n-\I , \Omega_{x,\omega} \ra =
\la 
\bigl((\Theta^l(T_n)-\id)\otimes \id\bigr)x,\omega\ra \xrightarrow[n\to\infty]{} 0,
\]
and since $(T_n)_{n\in\NN}$ is uniformly bounded in CB norm, the same holds for general $x\in \mrm{C}(\wh\GGamma)\otimes \mc{K}(\msf{H}), \omega \in \LL^1(\wh\GGamma)\wh\otimes \B(\msf{H})_*$.  By Proposition~\ref{prop15}, this shows that $T_n \xrightarrow[n\to\infty]{} \I$ weak$^*$ in $\M^l_{cb}(\A(\GGamma))$.  As each $p_d$ is in the weak$^*$-closure of $\mrm{c}_{00}(\GGamma)$, the same is true of each $T_n$, and hence we conclude that $\I$ is in the weak$^*$-closure of $\mrm{c}_{00}(\GGamma)$, showing that $\GGamma$ has AP.

If $f_{i,\lambda}\in \mc{Z}(\ell^{\infty}(\GGamma_i))\cap \mrm{c}_{00}(\GGamma_i)$ for each $i,\lambda$, then $\Psi_d( (f_{1,\lambda_{1,k}} , f_{2,\lambda_{2,k}} )_{k=1}^{d})$ is also central.  Consequently, if $\GGamma_1,\GGamma_2$ have central AP then so does $\GGamma=\GGamma_1\star\GGamma_2$.
\end{proof}

\begin{corollary}
Let $(\GGamma_i)_{i\in I}$ be a family of discrete quantum groups with (central) AP. Then the free product $\GGamma=\star_{i\in I}\GGamma_i$ has (central) AP.
\end{corollary}

\begin{proof}
If $I$ is finite the claim follows from Theorem \ref{thm1} by induction. In the general case, for any finite (nonempty) set $F\subseteq I$, the free product $\star_{i\in F}\GGamma_i$ is a quantum subgroup of $\star_{i\in I}\GGamma_i$ in a natural way. Moreover $(\star_{i\in F}\GGamma_i)_{F\subseteq I}$ forms a directed system of discrete quantum groups with injective connecting maps over the directed set of finite subsets of $I$, compare Definition \ref{def2}, and $\star_{i\in I}\GGamma_i=
\varinjlim_{F\subseteq I} \star_{i\in F}\GGamma_i$. Since $\star_{i\in F}\GGamma_i$ has (central) AP, the claim follows from Proposition \ref{prop8}.
\end{proof}

\subsection{Double crossed products}

In this section we study how the approximation property behaves with respect to the double crossed product construction. This contains the Drinfeld double of 
a locally compact quantum group as a special case. 

\subsubsection{Preliminaries}\label{sec:dbl_cp:prelim}
We start by recalling some definitions, following the conventions in \cite{Doublecrossed}.

A \emph{matching} between two locally compact quantum groups $\GG_1,\GG_2$ is a faithful normal $\star$-homomorphism $\oon{m}\colon \LL^{\infty}(\GG_1)\bar{\otimes}\LL^{\infty}(\GG_2)\rightarrow \LL^{\infty}(\GG_1)\bar{\otimes}\LL^{\infty}(\GG_2)$ satisfying
\[
(\Delta_1\otimes\id)\oon{m}=\oon{m}_{23} \oon{m}_{13} (\Delta_1\otimes\id)\quad\textnormal{ and }\quad
(\id\otimes\Delta_2)\oon{m}=\oon{m}_{13}\oon{m}_{12}(\id\otimes\Delta_2).
\]
Given this data one defines the \emph{double crossed product} $\GG_{\oon{m}}$ of $\GG_1, \GG_2$ as follows. The Hilbert space of square integrable functions on $\GG_{\oon{m}}$, the von Neumann algebra of functions on $\GG_{\oon{m}}$ and its comultiplication are given by
\[
\LL^2(\GG_{\oon{m}})=\LL^2(\GG_1)\otimes \LL^2(\GG_2),\;\, \LL^{\infty}(\GG_{\oon{m}})=\LL^{\infty}(\GG_1)\bar{\otimes}\LL^{\infty}(\GG_2),\;\,
\Delta_{\oon{m}}=(\id\otimes \chi\oon{m}\otimes\id)(\Delta_1^{\oon{op}}\otimes\Delta_2).
\]
To ease notation, we will decorate objects related to $\GG_1$ (resp. $\GG_2,\GG_{\oon{m}})$ with $1$ (resp. $2,\oon{m}$) in the sequel, e.g.~$\ww_1=\ww^{\GG_1}$. We will also denote the unit in $\B(\LL^{2}(\GG_{\oon{m}}))$ by $\I_{\oon{m}}$.

Let $J$ (resp.~$\hat{J}$) be the modular conjugation of the left Haar integral on the bicrossed product of $\GG_1,\GG_2$ (resp.~its dual), see \cite[Section 2.4]{Doublecrossed}. 
Define a unitary
\[
Z=J \hat{J}(\hat{J}_1 J_1\otimes \hat{J}_2 J_2).
\]
It implements $\oon{m}$ in the sense that $\oon{m}(z)=Z z Z^*$ for all $z\in \LL^{\infty}(\GG_1)\bar{\otimes} \LL^{\infty}(\GG_2)$. The Kac-Takesaki operator of $\GG_{\oon{m}}$ is given by
\[
\ww_{\oon{m}}=(\Sigma \vv_1^* \Sigma)_{13} Z_{34}^* \ww_{2,24} Z_{34}.
\]
One can describe structure of $\GG_{\oon{m}}$ and its dual, see in particular \cite[Theorem 5.3]{Doublecrossed}. For example, we have
\[\begin{split}
\LL^{\infty}(\wh{\GG_{\oon{m}}})&=\bigl((\LL^{\infty}(\wh{\GG_1})'\otimes\I)\cup Z^* (\I\otimes \LL^{\infty}(\wh{\GG_2}))Z\bigr)'',\\
\LL^{\infty}(\wh{\GG_{\oon{m}}})'&=\bigl(Z^* (\LL^{\infty}(\wh{\GG_1})\otimes\I)Z \cup (\I\otimes \LL^{\infty}(\wh{\GG_2})')\bigr)''.
\end{split}\]

A special case of this construction is the (generalised) Drinfeld double. Let $\GG_1,\GG_2$ be locally compact quantum groups 
and assume that $\mc{Z}\in \LL^{\infty}(\GG_1)\bar{\otimes} \LL^{\infty}(\GG_2)$ is a bicharacter. That is, $\mc{Z} $ is a unitary satisfying
\[
(\Delta_1\otimes\id)\mc{Z}=\mc{Z}_{23} \mc{Z}_{13} \quad 
\textnormal{and}\quad
(\id\otimes\Delta_2)\mc{Z}=\mc{Z}_{13}\mc{Z}_{12}.
\]
Then one obtains an inner $\star$-automorphism
\[
\oon{m}\colon \LL^{\infty}(\GG_1)\bar{\otimes}\LL^{\infty}(\GG_2)\ni x \mapsto \mc{Z}x \mc{Z}^*\in
\LL^{\infty}(\GG_1)\bar{\otimes}\LL^{\infty}(\GG_2),
\]
and it is easy to check that this defines a matching between $\GG_1$ and $\GG_2$. 
Consequently, one can form the double crossed product $\GG_{\oon{m}}$, and this is called the \emph{generalised Drinfeld double} of $\GG_1,\GG_2$ with respect to $\mc{Z}$. 

In particular, if $\HH$ is a locally compact quantum group then we can consider $\GG_1=\HH^{\oon{op}}, \GG_2=\whH$ together with the bicharacter $\mc{Z}=\ww^{\HH}$. 
The corresponding double crossed product $\GG_{\oon{m}}$ is called the \emph{Drinfeld double} of $\HH$.

\subsubsection{$\GG_1^{\oon{op}},\GG_2$ are quantum subgroups of $\GG_{\oon{m}}$}

Let us return to the general situation of locally compact quantum groups $\GG_1,\GG_2$ with a matching $\oon{m}$. It is stated in \cite[Theorem~5.3]{Doublecrossed}, see also the introduction to \cite[Section~6]{Doublecrossed}, that $\GG_1^{\oon{op}}$ and $\GG_2$ are closed quantum subgroups of $\GG_{\oon{m}}$. We give a quick argument for the convenience of the reader. 

\begin{lemma}\label{lemma11}
$\GG_1^{\oon{op}}$ and $\GG_2$ are closed quantum subgroups of $\GG_{\oon{m}}$ in the sense of Vaes.
\end{lemma}

\begin{proof}
Note first that $\wh{\GG_1^{\oon{op}}}=\wh{\GG_1}'$, compare \cite[Proposition 5.4]{KVVN}. We have natural normal, injective $\star$-homomorphisms
\[
\begin{split}
\gamma_1&\colon \LL^{\infty}(\wh{\GG_1})'\rightarrow\LL^{\infty}(\wh{\GG_{\oon{m}}})\colon \wh x'\mapsto \wh{x}'\otimes \I,\\
\gamma_2&\colon \LL^{\infty}(\wh{\GG_2})\rightarrow\LL^{\infty}(\wh{\GG_{\oon{m}}})\colon \wh x\mapsto Z^*(\I\otimes \wh x)Z,
\end{split}\]
hence by \cite[Theorem 3.3]{DKSS_ClosedSub} it is enough to show that both maps respect coproducts.

First, take $\wh x' \in \LL^{\infty}(\wh{\GG_1})'$. Then $\wh x'\otimes \I \in \LL^{\infty}(\wh{\GG_{\oon{m}}})$ and 
\[
\wh\Delta_{\oon{m}}(\wh x'\otimes \I) = \Sigma  \ww_{\oon{m}} ((\wh x'\otimes \I)\otimes \I_{\oon{m}}) \ww_{\oon{m}}^* \Sigma
= \Sigma (\Sigma \vv_1^* \Sigma)_{13} (\wh x'\otimes \I\otimes \I\otimes \I) (\Sigma \vv_1 \Sigma)_{13} \Sigma,
\]
noting that all the other parts of $\ww_{\oon m}$ cancel out. By slight abuse of notation, we write here $\Sigma$ both for the swap map on $\LL^2(\GG_{\oon{m}}) \otimes \LL^2(\GG_{\oon{m}})$, 
which is identified with $\Sigma_{13} \Sigma_{24}$, and for the swap map on $\LL^2(\GG_1)\otimes \LL^2(\GG_1)$. 
From the proof of \cite[Proposition~4.2]{KVVN} we find that the coproduct on $\LL^{\infty}(\wh{\GG_1})'$ is given by $\Delta_{\wh{\GG_1}'}(\wh x') = \vv_1^*(\I\otimes\wh x')\vv_1$, and so
\[
\wh\Delta_{\oon{m}}(\wh x'\otimes \I) = 
\Sigma ( \Sigma \vv_1^* (\I\otimes \wh x') \vv_1 \Sigma )_{13} \Sigma
= 
\Sigma_{24} \Delta_{\wh{\GG_1}'}(\wh x')_{13} \Sigma_{24} = \Delta_{\wh{\GG_1}'}(\wh x')_{13}.
\]
Since the inclusion $\LL^{\infty}(\wh{\GG_1})' \bar\otimes\LL^{\infty}(\wh{\GG_1})' \rightarrow 
\LL^{\infty}(\wh{\GG_{\oon{m}}}) \bar\otimes \LL^{\infty}(\wh{\GG_{\oon{m}}})$ is given by $a\mapsto a_{13}$, this concludes the proof that $\GG_1^{\oon{op}}$ is a 
closed quantum subgroup of $\GG_{\oon{m}}$.

Take now $\wh x \in \LL^{\infty}(\wh{\GG_2})$ so that $Z^*(\I\otimes \wh x)Z \in \LL^{\infty}(\wh{\GG_{\oon{m}}})$.  Then, following exactly the proof of \cite[Proposition~3.5]{Doublecrossed},
\begin{align*}
\wh\Delta_{\oon{m}}(Z^*(\I\otimes \wh x)Z)
&= \Sigma  \ww_{\oon{m}}( Z^*(\I\otimes \wh x)Z\otimes \I_{\oon{m}}) \ww_{\oon{m}}^* \Sigma = (Z^*\otimes Z^*)\wh\Delta_2(\wh x) _{24}(Z\otimes Z),
\end{align*}
which is exactly the embedding of $\Delta_{\wh{\GG_2}}(\wh x) \in \LL^{\infty}(\wh{\GG_2}) \bar\otimes \LL^{\infty}(\wh{\GG_2})$ into $\LL^{\infty}(\wh{\GG_{\oon{m}}}) \bar\otimes \LL^{\infty}(\wh{\GG_{\oon{m}}})$.
\end{proof}

As a consequence of Theorem \ref{thm2} and Proposition \ref{prop11} we therefore obtain the following fact. 

\begin{corollary}\label{cor1}
Suppose that $\GG_{\oon{m}}$ has AP. Then both $\GG_1$ and $\GG_2$ have AP.
\end{corollary}

\begin{remark}\label{rem:suq3}
In view of the close analogy between Drinfeld doubles of $q$-deformations of compact semisimple Lie groups and the corresponding complex Lie groups \cite{Aranospherical}, \cite{VoigtYuncken}, 
it is natural to speculate that the converse to Corollary \ref{cor1} does not hold, see also Remark 7.4 in \cite{ARANO_DELAAT_WAHL_fourier}. 
Specifically, the Drinfeld double of $\SU_q(3)$ might be an example of a locally compact quantum group which does not have AP.
\end{remark}

\subsubsection{AP for $\wh{\GG_{\oon{m}}}$}

We now aim to prove the following result.

\begin{theorem}\label{thm3}
Let $\GG_1,\GG_2$ be locally compact quantum groups with a matching $\oon{m}$. 
\begin{itemize}
\item If $\wh{\GG_1}$ and $\wh{\GG_2}$ have AP then so does $\wh{\GG_{\oon{m}}}$.
\item If $\wh{\GG_1}$ and $\wh{\GG_2}$ are weakly amenable with Cowling--Haagerup constants $\Lambda_{cb}(\wh{\GG_1}), \Lambda_{cb}(\wh{\GG_2})$ then $\wh{\GG_{\oon{m}}}$ is weakly amenable with $\Lambda_{cb}(\wh{\GG_{\oon{m}}})\le \Lambda_{cb}(\wh{\GG_1})\, \Lambda_{cb}(\wh{\GG_2})$.
\item If $\GG_1,\GG_2$ are coamenable then so is $\GG_{\oon{m}}$.
\end{itemize}
\end{theorem}

An analogous result for the Haagerup property was obtained in \cite{RoyHaagerup}, but we note that the terminology used in \cite{RoyHaagerup} is different.
We also note that for generalised Drinfeld doubles the statement on coamenability in Theorem~\ref{thm3} can be shown quite easily using standard properties of bicharacters, 
as discussed in a preprint version of \cite{RoyHaagerup}. 

Before we prove Theorem \ref{thm3} we need to establish a number of auxiliary results. Recall from Section \ref{section isomorphism} the 
construction of a normal CB map $\Theta^l(a)$ on $\LL^{\infty}(\whH)$ and its extension $\Phi(a)\in \CB^\sigma(\B(\LL^2(\HH)))$ for any locally compact 
quantum group $\HH$ and $a\in \M^l_{cb}(\A(\HH))$. 
Moreover, in the proof of Lemma \ref{lemma11} we introduced injective, normal $\star$-homomorphisms $\gamma_1 \colon \LL^{\infty}(\wh{\GG_1})'\rightarrow\LL^{\infty}(\wh{\GG_{\oon{m}}})$ 
and $\gamma_2 \colon \LL^{\infty}(\wh{\GG_2})\rightarrow\LL^{\infty}(\wh{\GG_{\oon{m}}})$. We will now use these maps to transport elements of 
the Fourier algebras $\A(\wh{\GG_1^{\oon{op}}}),\A(\wh{\GG_2})$ to left CB multipliers of $\A(\wh{\GG_{\oon{m}}})$.

\begin{lemma}\label{lemma10}
For $\omega\in \LL^1(\GG_1)$ we have $ \gamma_1(\lambda_1^{\oon{op}}(\omega)) \in \M_{cb}^l(\A(\wh{\GG_{\oon{m}}}))$. The maps associated with $a=\gamma_1(\lambda_1^{\oon{op}}(\omega))$ are
\[
\Theta^l(a)=\Theta^l(\lambda_1^{\oon{op}}(\omega))\otimes \id \in \CB^\sigma(\LL^{\infty}(\GG_{\oon{m}}))=\CB^\sigma(\LL^{\infty}(\GG_1)\bar{\otimes}\LL^\infty(\GG_2))
\]
and
\[
\Phi(a)=\Phi(\lambda_1^{\oon{op}}(\omega))\otimes \id \in \CB^\sigma(\B(\LL^2(\GG_{\oon m})))=
\CB^\sigma(\B(\LL^2(\GG_1))\bar{\otimes} \B(\LL^2(\GG_2))).
\]
\end{lemma}

\begin{proof}
Take $\omega_1\otimes\omega_2\in \LL^1(\GG_1^{\oon{op}})\wh{\otimes}\LL^1(\GG_2) = \LL^1(\GG_{\oon{m}})$. Using $\ww_1^{\oon{op}}=\Sigma \vv_1^* \Sigma$, see \cite[Section 4]{KVVN}, we get 
\begin{equation}\begin{split}\label{eq10}
\lambda_{\oon{m}}(\omega_1\otimes \omega_2)&=
(\omega_1\otimes \omega_2\otimes\id\otimes\id)\bigl(
(\Sigma \vv_1^*\Sigma)_{13} Z_{34}^* \ww_{2,24} Z_{34}\bigr)\\
&=
\bigl(  (\omega_1\otimes \id)(\Sigma\vv_1^*\Sigma)\otimes \I\bigr) 
Z^*\bigl(\I\otimes  (\omega_2\otimes\id)(\ww_2)\bigr)Z\\
&=
\gamma_1(\lambda_1^{\oon{op}}(\omega_1))\gamma_2(\lambda_2(\omega_2)),
\end{split}\end{equation}
and consequently, writing $\star$ for the product on $\LL^1(\GG_1^{\oon{op}})$,
\begin{align*}
a \lambda_{\oon{m}}(\omega_1\otimes\omega_2)
&= \gamma_1(\lambda_1^{\oon{op}}(\omega)) \lambda_{\oon{m}}(\omega_1\otimes\omega_2)
= \gamma_1\big( \lambda_1^{\oon{op}}(\omega) \lambda_1^{\oon{op}}(\omega_1) \big)\gamma_2(\lambda_2(\omega_2)) \\
&= \gamma_1\big( \lambda_1^{\oon{op}}(\omega \star \omega_1) \big)\gamma_2(\lambda_2(\omega_2))=\lambda_{\oon{m}}\bigl((\omega\star\omega_1)\otimes \omega_2\bigr).
\end{align*}
By linearity and continuity, $a$ maps $\A(\wh{\GG_{\oon{m}}})$ into itself, and $\Theta^l(a)$ has the given form. 

The second assertion is verified using a direct calculation. Indeed, if $x\in \B(\LL^2(\GG_{\oon{m}}))=\B(\LL^2(\GG_1))\bar{\otimes}\B(\LL^2(\GG_2))$ then 
\[\begin{split}
\I_{\oon{m}}\otimes\Phi(a)(x) &=
\ww_{\oon{m}}\bigl( (((\omega\otimes\id)\Delta^{\oon{op}}_1\otimes\id)\otimes\id)(\ww_{\oon{m}}^*(\I_{\oon{m}}\otimes x)\ww_{\oon{m}}) \bigr) \ww_{\oon{m}}^* \\
&= \ww_{\oon{m}}
\bigl( ((\omega\otimes\id)\Delta^{\oon{op}}_1\otimes\id^{\otimes 3})( Z_{34}^* \ww_{2,24}^* Z_{34} \ww^{\oon{op}\,*}_{1,13} x_{34} \ww^{\oon{op}}_{1,13} Z_{34}^* \ww_{2,24} Z_{34})
   \bigr) \ww_{\oon{m}}^* \\
&= \ww_{\oon{m}} Z_{34}^* \ww_{2,24}^* Z_{34} \bigl( ((\omega\otimes\id)\Delta^{\oon{op}}_1\otimes\id^{\otimes 3}) \ww^{\oon{op}\,*}_{1,13} x_{34} \ww^{\oon{op}}_{1,13} \bigr) Z_{34}^* \ww_{2,24} Z_{34} \ww_{\oon{m}}^* \\
&= \ww_{\oon{m}} Z_{34}^* \ww_{2,24}^* Z_{34} \bigl( (\omega\otimes\id^{\otimes 4}) (\ww^{\oon{op}\,*}_{1,24} \ww^{\oon{op}\,*}_{1,14} x_{45} \ww^{\oon{op}}_{1,14} \ww^{\oon{op}}_{1,24}) \bigr) Z_{34}^* \ww_{2,24} Z_{34} \ww_{\oon{m}}^* \\
&= \ww_{\oon{m}} Z_{34}^* \ww_{2,24}^* Z_{34} \ww^{\oon{op}\,*}_{1,13} \big(
  (\omega\otimes\id\otimes\id) (\ww^{\oon{op}\,*}_{1,12} x_{23} \ww^{\oon{op}}_{1,12})
  \big)_{34} \ww^{\oon{op}}_{1,13} Z_{34}^* \ww_{2,24} Z_{34} \ww_{\oon{m}}^* \\
&= \I_{\oon{m}} \otimes (\omega\otimes\id\otimes\id) ( \ww^{\oon{op}\,*}_{1,12} x_{23} \ww^{\oon{op}}_{1,12}).
\end{split}\]
Here we use that $(\Delta^{\oon{op}}_1\otimes\id)\ww^{\oon{op}}_1 = \ww^{\oon{op}}_{1,13} \ww^{\oon{op}}_{1,23}$.  The claim follows from Lemma~\ref{lemma9}.
\end{proof}

\begin{lemma}\label{lemma12}
For $\omega\in \LL^1(\GG_2)$ we have $\gamma_2(\lambda_2(\omega)) \in \M^l_{cb}(\A(\wh{\GG_{\oon{m}}}))$. The associated maps with $b=\gamma_2(\lambda_2(\omega))$ are
\[
\Theta^l(b) = \oon{m}^{-1}(\id\otimes \Theta^l( \lambda_2(\omega))) \oon{m}
\]
and
\[
\Phi(b)\colon \B(\LL^2(\GG_{\oon m}))\ni x \mapsto   Z^* (\id\otimes\Phi(\lambda_2(\omega)))(ZxZ^*) Z\in \B(\LL^2(\GG_{\oon m})).
\]
\end{lemma}

\begin{proof}
Using equation \eqref{eq10}, we get for $\omega_0 \in \LL^1(\GG_2)$ and $\omega_1\otimes\omega_2 \in \LL^1(\GG_1^{\oon{op}})\wh{\otimes}\LL^1(\GG_2)= \LL^1(\GG_{\oon{m}})$ the relation
\[\begin{split}
&\quad\;
\lambda_{\oon{m}}(\omega_1\otimes\omega_2) \gamma_2(\lambda_2(\omega_0)) = 
\gamma_1(\lambda_1^{\oon{op}}(\omega_1)) \gamma_2(\lambda_2(\omega_2)) \gamma_2(\lambda_2(\omega_0))
\\
&= \gamma_1(\lambda_1^{\oon{op}}(\omega_1)) \gamma_2(\lambda_2(\omega_2 \star \omega_0))=
\lambda_{\oon{m}}(\omega_1\otimes (\omega_2\star\omega_0)),
\end{split}\]
here with $\star$ the product on $\LL^1(\GG_2)$.
Thus, if $T_0\colon\LL^\infty(\GG_{\oon m})\rightarrow \LL^\infty(\GG_{\oon m})$ is the map given by $T_0 = \id\otimes(\id\otimes\omega_0)\Delta_2$, then $T_0$ is normal, and the pre-adjoint $(T_0)_*$ satisfies
\[ \lambda_{\oon m}(\omega_1\otimes\omega_2) \gamma_2(\lambda_2(\omega_0)) = \lambda_{\oon m}( (T_0)_*(\omega_1\otimes\omega_2) ).
\]
As $\gamma_2$ intertwines the coproducts, it automatically intertwines the unitary antipodes (\cite[Proposition 3.10]{MeyerRoyWoronowicz}). 
The same is true for $\lambda_{\oon m}$ and $\lambda_2$, and hence we get 
\[\begin{split}
&\quad\;
\lambda_{\oon{m}}\bigl( (R_{\oon{m}}\circ T_0\circ R_{\oon{m}})_* (R_{\oon{m} *}(\omega_1\otimes \omega_2))\bigr)=
\lambda_{\oon{m}}( (T_0\circ R_{\oon{m}})_*(\omega_1\otimes \omega_2))= \wh{R}_{\oon{m}}(\lambda_{\oon{m}}((T_0)_*(\omega_1\otimes \omega_2)))\\
&=\wh{R}_{\oon{m}}\bigl(\gamma_2(\lambda_2(\omega_0))\bigr) \wh{R}_{\oon{m}}(\lambda_{\oon{m}}(\omega_1\otimes \omega_2))=
\gamma_2(\lambda_2( R_{2 *}(\omega_0) ))
\lambda_{\oon{m}}( R_{\oon{m} *}(\omega_1\otimes \omega_2)).
\end{split}\] 
Now set $\omega_0 = \omega \circ R_2$ for our given $\omega$.
As the set of functionals of the form $R_{\oon{m} *}(\omega_1\otimes \,\omega_2)$ is linearly dense in $\LL^1(\GG_{\oon{m}})$, we obtain from Lemma \ref{lemma18} that $b=\gamma_2(\lambda_2(\omega))\in \M^l_{cb}(\A(\wh{\GG_{\oon{m}}}))$ and $\Theta^l(b)=R_{\oon{m}}\circ T_0\circ R_{\oon{m}}$. 
By \cite[Theorem~5.3]{Doublecrossed} we know that $R_{\oon m} = {\oon m}^{-1}(R_1\otimes R_2) = (R_1\otimes R_2)\oon{m}$. Therefore
\begin{align*}
R_{\oon m} \circ T_0 \circ R_{\oon m} &= \oon{m}^{-1}(R_1\otimes R_2) (\id\otimes(\id\otimes\omega\circ R_2)\Delta_2)  (R_1\otimes R_2){\oon m} \\
&= {\oon m}^{-1}(\id\otimes R_2\circ (\id\otimes\omega\circ R_2)\Delta_2\circ R_2) \oon{m} \\
&= \oon{m}^{-1}(\id\otimes (\id\otimes\omega)\Delta_2^{\oon{op}}) \oon{m} \\
&= \oon{m}^{-1}(\id\otimes (\omega\otimes\id)\Delta_2) \oon{m},
\end{align*}
and this yields the stated formula for $\Theta^l(b)$.

In order to verify the formula for $\Phi(b)$, recall that the unitary operator $Z$ implements $\oon{m}$ by $\oon{m}(\cdot) = Z\cdot Z^*$, and hence $\oon{m}^{-1}(\cdot) = Z^*\cdot Z$. 
Moreover, from \cite[Proposition~3.5]{Doublecrossed} we know that $(\oon{m}\otimes\id)(\ww_{\oon m}) = Z_{34}^* \ww_{2,24} Z_{34} \ww^{\oon{op}}_{1,13}$.
For $x\in \B(\LL^2(\GG_{\oon m}))$, by applying the expression for $\Theta^l(b)$ just obtained, we get
\begin{align*} 
&\quad\;(\Theta^l(b)\otimes\id)(\ww_{\oon m}^*(\I_{\oon{m}}\otimes x)\ww_{\oon m}) \\
&= (\oon{m}^{-1}\otimes\id\otimes\id) (\id\otimes(\omega\otimes\id)\Delta_2\otimes\id\otimes\id)\\
&\quad
\quad\quad\quad\quad\quad\quad\quad\quad\quad\quad\quad\quad 
\big( \ww^{\oon{op}\,*}_{1,13} Z_{34}^* \ww_{2,24}^* Z_{34} (\I\otimes \I\otimes x) Z_{34}^* \ww_{2,24} Z_{34} \ww^{\oon{op}}_{1,13} \big).
\end{align*}
Now using $(\Delta_2\otimes\id)\ww_2 = \ww_{2,13} \ww_{2,23}$, this expression becomes
\begin{align*}
&\quad\;(\Theta^l(b)\otimes\id)(\ww_{\oon m}^*(\I_{\oon{m}}\otimes x)\ww_{\oon m})\\
&= (\oon{m}^{-1}\otimes\id\otimes\id) \big( \ww^{\oon{op}\,*}_{1,13} Z_{34}^* (\id\otimes\omega\otimes\id^{\otimes 3}) 
    (\ww_{2,35}^* \ww_{2,25}^* (ZxZ^*)_{45} \ww_{2,25} \ww_{2,35})
  Z_{34} \ww^{\oon{op}}_{1,13} \big) \\
&= ({\oon m}^{-1}\otimes\id\otimes\id) \big( \ww^{\oon{op}\,*}_{1,13} Z_{34}^* \ww_{2,24}^*
   (\omega\otimes\id\otimes\id)(\ww_{2,13}^* (ZxZ^*)_{23} \ww_{2,13})_{34}
   \ww_{2,24} Z_{34} \ww^{\oon{op}}_{1,13} \big).
\end{align*}
Finally, we use the form of $\Phi(\lambda_2(\omega))\in \CB^{\sigma}(\B(\LL^2(\GG_2)))$, as in Lemma~\ref{lemma9} to get
\begin{align*}
&\quad\;(\Theta^l(b)\otimes\id)(\ww_{\oon m}^*(\I_{\oon{m}}\otimes x)\ww_{\oon m})\\
&= ({\oon m}^{-1}\otimes\id\otimes\id) \big( \ww^{\oon{op}\,*}_{1,13} Z_{34}^* \ww_{2,24}^*
   (\id\otimes\Phi(\lambda_2(\omega)))(ZxZ^*)_{34}
   \ww_{2,24} Z_{34} \ww^{\oon{op}}_{1,13} \big).
\end{align*}
Using again $(\oon{m}\otimes\id)(\ww_{\oon m}) = Z_{34}^* \ww_{2,24} Z_{34} \ww^{\oon{op}}_{1,13}$, we continue the calculation as
\begin{align*}
&\quad\;(\Theta^l(b)\otimes\id)(\ww_{\oon m}^*(\I_{\oon{m}}\otimes x)\ww_{\oon m})\\
&= (\oon{m}^{-1}\otimes\id\otimes\id) \big( (\oon{m}\otimes\id)(\ww_{\oon m}^*) Z_{34}^*
   (\id\otimes\Phi(\lambda_2(\omega)))(ZxZ^*)_{34}  Z_{34}  (\oon{m}\otimes\id)(\ww_{\oon m}) \big) \\
&= \ww_{\oon m}^* (\oon{m}^{-1}\otimes\id\otimes\id)( Z_{34}^* (\id\otimes\Phi(\lambda_2(\omega)))(ZxZ^*)_{34}  Z_{34} ) \ww_{\oon{m}} \\
&= \ww_{\oon m}^* Z_{12}^* Z_{34}^* (\id\otimes\Phi(\lambda_2(\omega)))(ZxZ^*)_{34}  Z_{34} Z_{12} \ww_{\oon m},
\end{align*}
where at the end we used that $\oon{m}^{-1}(\cdot) = Z^*\cdot Z$.  It follows that
\begin{align*}
\I_{\oon{m}}\otimes\Phi(b)(x) &=
\ww_{\oon m} (\Theta^l(b)\otimes \id)(\ww^*_{\oon m} (\I_{\oon{m}}\otimes x) \ww_{\oon m})\ww^*_{\oon m}
\\
&=
\ww_{\oon m} \big( \ww_{\oon m}^* Z_{12}^* Z_{34}^* (\id\otimes\Phi(\lambda_2(\omega)))(ZxZ^*)_{34}  Z_{34} Z_{12} \ww_{\oon m} \big) \ww_{\oon m}^* \\
&= (Z\otimes Z)^* (\I_{\oon{m}}\otimes (\id\otimes\Phi(\lambda_2(\omega)))(ZxZ^*)) (Z\otimes Z),
\end{align*}
and hence $\Phi(b)(x) = Z^* (\id\otimes\Phi(\lambda_2(\omega)))(ZxZ^*) Z$ as claimed.
\end{proof}

In the next step, we shall establish continuity properties of the maps  $\A(\wh{\GG_1^{\oon {op}}})\rightarrow \M^l_{cb}(\A(\wh{\GG_{\oon m}}))$ and $\A(\wh{\GG_2})\rightarrow \M^l_{cb}(\A(\wh{\GG_{\oon m}}))$ described in Lemmas \ref{lemma10}, \ref{lemma12}. For this we need the following general fact.

\begin{lemma}\label{lemma14}$ $
\begin{itemize}
\item Let $E,F$ be operator spaces.  The map $(E\check\otimes F)\wh\otimes F^* \rightarrow E$ given on simple tensors by $(x\otimes y)\otimes f \mapsto \la f , y \ra  x$ is completely contractive.
\item Let $A,B$ be $C^*$-algebras. The map $(A\otimes B)\wh\otimes(A^*\wh\otimes B^*)\rightarrow A\wh\otimes A^*$ given on simple tensors by $(a\otimes b)\otimes(\mu\otimes\omega) \mapsto \la \omega , b \ra a\otimes \mu$ is completely contractive.
\end{itemize}\end{lemma}

\begin{proof}
Due to \cite[Theorem~8.1.10]{EffrosRuan} the ``tensor interchange'' map $F^* \wh\otimes (F\check\otimes E) \rightarrow (F^*\wh\otimes F) \check\otimes E$, which is the formal identity on simple tensors, is a complete contraction. Since $(F^*\wh\otimes F)^* \cong \CB(F^*)$, the identity map $\id_{F^*}$ induces the (completely) contractive linear 
functional $F^*\wh\otimes F\rightarrow\mathbb C\colon f\otimes y\mapsto \la f , y \ra$. Composing these complete contractions shows that $F^* \wh\otimes (F\check\otimes E) \rightarrow E\colon f\otimes (y\otimes x) \mapsto \la f,y\ra x$ is a complete contraction. By commutativity of the projective and the injective tensor products of operators spaces, respectively, 
the first claim follows.

For the second part recall that the injective operator space tensor product agrees with the spatial tensor product on $C^*$-algebras. Using the re-bracketing isomorphism
\[
(A\otimes B)\wh\otimes(A^*\wh\otimes B^*)\cong (((A\otimes B)\wh\otimes B^*)\wh\otimes A^*, 
\]
the assertion hence follows by applying the first part to $ E = A, F = B $ and tensoring with $ A^* $.
\end{proof}

\begin{lemma}\label{lemma13}
Let $(\omega_i)_{i\in I}$ be a net in $\LL^1(\GG_1^{\oon{op}})$ such that $\lambda_1^{\oon{op}}(\omega_i) \xrightarrow[i\in I]{} \I$ weak$^*$ in $\M_{cb}^l(\A(\wh{\GG_1^{\oon{op}}}))$. Consider $ \gamma_1(\lambda_1^{\oon{op}}(\omega_i)) \in \M_{cb}^l(\A(\wh{\GG_{\oon m}}))$ and $\Phi(\gamma_1(\lambda_1^{\oon{op}}(\omega_i)))\in \CB^{\sigma}(\B(\LL^2(\GG_{\oon m}))) $.  Then $\Phi(\gamma_1(\lambda_1^{\oon{op}}(\omega_i)))$ $\xrightarrow[i\in I]{}\id$ weak$^*$, and thus $\gamma_1(\lambda_1^{\oon{op}}(\omega_i))\xrightarrow[i\in I]{} \I_{\oon{m}}$ weak$^*$ in $\M_{cb}^l(\A(\wh{\GG_{\oon m}}))$.
\end{lemma}

\begin{proof}
By Lemma \ref{lemma10} we have $\Phi(\gamma_1(\lambda_1^{\oon{op}}(\omega_i))) = \Phi(\lambda_1^{\oon {op}}(\omega_i))\otimes\id$ for each $i\in I$. 
Applying Lemma \ref{lemma14} with $A=\mc{K}(\LL^2(\GG_1))$ and $B=\mc{K}(\LL^2(\GG_2))$ we obtain a completely contractive map
\[
T \colon \bigl(\mc{K}(\LL^2(\GG_1))\otimes\mc{K}(\LL^2(\GG_2)) \bigr)\wh\otimes \bigl( \B(\LL^2(\GG_1))_* \wh\otimes \B(\LL^2(\GG_2))_* \bigr) \rightarrow \mc{K}(\LL^2(\GG_1)) \wh\otimes \B(\LL^2(\GG_1))_*
\]
given by $T((a\otimes b)\otimes(\omega_1\otimes\omega_2)) = \langle \omega_2, b\rangle a\otimes\omega_1$ on simple tensors. Recall that for any Hilbert space $\msf{H}$, the space of normal CB maps $\CB^\sigma(\B(\msf{H}))$ has predual $\mc{K}(\msf{H})\wh{\otimes} \B(\msf{H})_*$ (see Section \ref{section isomorphism}). Using this identification
\[\begin{split}
\langle \Phi(\gamma_1(\lambda_1^{\oon{op}}(\omega_i))), (a\otimes b)\otimes(\omega_1\otimes\omega_2) \rangle
&= \langle \Phi(\lambda_1^{\oon {op}}(\omega_i))(a), \omega_1 \rangle \langle b, \omega_2 \rangle\\
&= \langle \Phi(\lambda_1^{\oon {op}}(\omega_i)), T((a\otimes b)\otimes(\omega_1\otimes\omega_2)) \rangle, 
\end{split}\]
and hence $\Phi(\gamma_1(\lambda_1^{\oon{op}}(\omega_i))) = T^*(\Phi(\lambda_1^{\oon{op}}(\omega_i)))$.  As $\lambda_1^{\oon{op}}(\omega_i) \xrightarrow[i\in I]{}\I$ weak$^*$, we know that $\Phi(\lambda_1^{\oon{op}}(\omega_i))$ $\xrightarrow[i\in I]{}\id$ weak$^*$, and so $\Phi(\gamma_1(\lambda_1^{\oon{op}}(\omega_i)))\xrightarrow[i\in I]{} T^*(\id) = \id$ weak$^*$, as required.
\end{proof}

\begin{lemma}\label{lemma15}
Let $(\omega_i)_{i\in I}$ be a net in $\LL^1(\GG_2)$ with $\lambda_2(\omega_i)\xrightarrow[i\in I]{}\I $ weak$^*$ in $\M_{cb}^l(\A(\wh{\GG_2}))$.  
Then $\Phi(\gamma_2(\lambda_2(\omega_i)))\xrightarrow[i\in I]{}\id$ weak$^*$ in $\CB^\sigma(\B(\LL^2(\GG_{\oon m})))$, and consequently $\gamma_2(\lambda_2(\omega_i))\xrightarrow[i\in I]{} \I_{\oon{m}}$ weak$^*$ in $\M^l_{cb}(\A(\wh{\GG_{\oon m}}))$.
\end{lemma}

\begin{proof}
According to Lemma \ref{lemma12} we have $\Phi(\gamma_2(\lambda_2(\omega_i)))(x) = Z^* (\id\otimes\Phi(\lambda_2(\omega_i)))(ZxZ^*) Z$. We can now argue exactly as in the proof of 
Lemma~\ref{lemma13}. Explicitly, for $x\otimes u$ in $ \mc{K}(\LL^2(\GG_{\oon m})) \wh\otimes \B(\LL^2(\GG_{\oon m}))_*$, the predual of $\CB^\sigma(\B(\LL^2(\GG_{\oon m})))$, consider
\begin{align*}
&\quad\;
\langle \Phi(\gamma_2(\lambda_2(\omega_i))), x\otimes u \rangle
= \langle Z^*(\id\otimes\Phi(\lambda_2(\omega_i)))(ZxZ^*)Z, u \rangle \\
&= \langle (\id\otimes\Phi(\lambda_2(\omega_i)))(ZxZ^*), ZuZ^* \rangle
= \langle (\id\otimes\Phi(\lambda_2(\omega_i))), ZxZ^* \otimes ZuZ^* \rangle.
\end{align*}
Since $x\mapsto ZxZ^*$ is a complete isometry on $\mc{K}(\LL^2(\GG_{\oon m}))$ and $u\mapsto ZuZ^*$ is a complete isometry on $\B(\LL^2(\GG_{\oon m}))_*$, we obtain 
a complete isometry $T$ on $\mc{K}(\LL^2(\GG_{\oon m})) \wh\otimes \B(\LL^2(\GG_{\oon m}))_*$ given on simple tensors by $T(x\otimes u) = ZxZ^* \otimes ZuZ^*$. Thus
\[
\Phi(\gamma_2(\lambda_2(\omega_i))) = T^*(\id\otimes\Phi(\lambda_2(\omega_i)))\xrightarrow[i\in I]{} T^*(\id) = \id
\]
weak$^*$, as required.
\end{proof}

\begin{proof}[Proof of Theorem \ref{thm3}]
Assume that $\wh{\GG_1}$ and $\wh{\GG_2}$ have AP. Due to Proposition \ref{prop11} it follows that $(\GG_1^{\oon {op}})^{\wedge}=(\wh{\GG_1})'$ also has AP. 
Choose nets $(\omega^{(1)}_i)_{i\in I}$ in $\LL^1(\GG_1)$ with $a_i = \lambda^{\oon{op}}_1(\omega^{(1)}_i)\xrightarrow[i\in I]{} \I$ weak$^*$ in $\M^l_{cb}(\A(\wh{\GG_1^{\oon{op}}}))$, and similarly $(\omega^{(2)}_{j})_{j\in J}$ in $\LL^1(\GG_2)$ with $b_j = \lambda_2(\omega^{(2)}_j)\xrightarrow[j\in J]{}\I$ weak$^*$ in $\M^l_{cb}(\A(\wh{\GG_2}))$. Then, by Lemmas \ref{lemma10}, \ref{lemma12}, we have $\gamma_1(a_i),\gamma_2(b_j)\in \M^l_{cb}(\A(\wh{\GG_{\oon m}}))$, and \eqref{eq10} gives
\[
c_{i,j} = \gamma_1(a_i) \gamma_2(b_j) = \lambda_{\oon m}(\omega^{(1)}_i \otimes\omega^{(2)}_j) \in \A(\wh{\GG_{\oon m}}). 
\]
Since $\M^l_{cb}(\A(\wh{\GG_{\oon m}}))$ is a dual Banach algebra, see Proposition \ref{prop18}, it follows from Lemma \ref{lemma13} that $\lim_{i\in I} c_{i,j} = \gamma_2(b_j)$ weak$^*$ 
for each $j\in J$. 
Hence $\gamma_2(b_j)$ is contained in the weak$^*$-closure of $\A(\wh{\GG_{\oon m}})$ in $\M^l_{cb}(\A(\wh{\GG_{\oon m}})$. Taking the limit in $j$ and using Lemma \ref{lemma15} 
we see that $\I_{\oon{m}}$ is contained in the weak$^*$-closure of $\A(\wh{\GG_{\oon m}})$ in $\M^l_{cb}(\A(\wh{\GG_{\oon m}}))$, as required.

The remaining statements regarding weak amenability and coamenability are verified in a similar way: if $\wh{\GG_1},\wh{\GG_2}$ are weakly amenable with 
Cowling-Haagerup constants $\Lambda_{cb}(\wh{\GG_1}), \Lambda_{cb}(\wh{\GG_2})$ then we additionally know that
\[
\|a_i\|_{cb}=\|\Phi(a_i)\|_{cb}\le \Lambda_{cb}(\wh{\GG_1}),\quad
\|b_j\|_{cb}=\|\Phi(b_j)\|_{cb}\le \Lambda_{cb}(\wh{\GG_2}),
\]
and consequently $\|c_{i,j}\|_{cb}\le \Lambda_{cb}(\wh{\GG_1})\,\Lambda_{cb}(\wh{\GG_2})$. The result then follows from Proposition \ref{prop7}.

If $\GG_1,\GG_2$ are coamenable then we can choose $a_i,b_j$ such that $\sup_{i\in I,j\in J}\|c_{i,j}\|_{\A(\wh{\GG_{\oon m}})}=$\\ $\sup_{i\in I, j\in J}\|\omega_i\otimes \omega_j\|<+\infty$. The result then follows from Proposition \ref{prop3}.
\end{proof}

\subsection{Direct products}

The double crossed product construction contains as a special case the direct product of locally compact quantum groups. More precisely, assume that $\HH_1,\HH_2$ are locally compact quantum groups and let $\oon{m}=\id$ be the trivial matching between $\GG_1=\HH_1^{\oon{op}}$ and $\GG_2=\HH_2$. In this case we write $\GG_{\oon{m}}=\HH_1\times \HH_2$ and call this 
the \emph{direct product} of $\HH_1$ and $\HH_2$. Note that this definition agrees with the usual one since
\[\begin{split}
\LL^{\infty}(\HH_1\times \HH_2)&=\LL^{\infty}(\GG_1^{\oon{op}})\bar{\otimes}
\LL^{\infty}(\GG_2)=\LL^{\infty}(\HH_1)\bar{\otimes}\LL^{\infty}(\HH_2),\\
\Delta_{\HH_1\times\HH_2}&=(\id\otimes \chi\otimes\id)(\Delta_{\GG_1^{\oon{op}}}\otimes\Delta_{\GG_2})=
(\id\otimes \chi\otimes\id)(\Delta_{\HH_1 }\otimes\Delta_{\HH_2}), 
\end{split}\]
and we have $\wh{\HH_1\times \HH_2}=\wh{\HH_1}\times \wh{\HH_2}$. Consequently, Corollary~\ref{cor1} and Theorem \ref{thm3} immediately give the following result. 

\begin{proposition}\label{prop9}
The direct product $\HH_1\times \HH_2$ of two locally compact quantum groups $\HH_1,\HH_2$ has AP if and only if $\HH_1$ and $\HH_2$ have AP.
\end{proposition}

\section{Categorical AP} \label{section categorical}

In this section we discuss the approximation property in the setting of rigid \cst-tensor categories, building 
on \cite{ARANO_DELAAT_WAHL_fourier}, \cite{ARANO_DELAAT_WAHL_howemoore}, \cite{AVhecke}, \cite{PVcstartensor}. 
As an application we show in particular that the central approximation property for discrete quantum groups is invariant under monoidal equivalence. 

Let us first fix some notation and terminology regarding \cst-tensor categories, referring to \cite{NeshveyevTuset} for more details and background. 
If $ \BT $ is a \cst-category and $ X, Y \in \BT $ are objects we write $ \BT(X,Y) $ for the space of morphisms from $ X $ to $ Y $. 
We denote by $ \id_X $ or $ \id $ the identity morphism in $ \BT(X,X) $.  
By definition, a \cst-tensor category is a \cst-category $ \BT $ together with a bilinear $ * $-functor $ \otimes: \BT \times \BT \rightarrow \BT $, 
a distinguished object $ \one \in \BT $ and unitary natural isomorphisms 
\begin{align*}
\one \otimes X \cong X \cong X \otimes \one, \qquad (X \otimes Y) \otimes Z \cong X \otimes (Y \otimes Z)  
\end{align*}
satisfying certain compatibility conditions. For simplicity we shall always assume that $ \BT $ is \emph{strict}, which means that these unitary natural isomorphisms are 
identities, and we also assume that the tensor unit $ \one $ is simple. 

A \cst-tensor category $ \BT $ is called \emph{rigid} if all objects of $ \BT $ are dualisable. This means that for every object $ X \in \BT $ there exists an 
object $ X^\vee \in \BT $ and morphisms $ s_X \in \BT(X \otimes X^\vee, \one), t_X \in \BT(X^\vee \otimes X, \one) $ which form a standard solution of the conjugate equations. 
That is, we have 
\begin{align*}
(t_X \otimes \id_{X^\vee})(\id_{X^\vee} \otimes s_X^*) = \id_{X^\vee}, \qquad (s_X \otimes \id_X)(\id_X \otimes t_X^*) = \id_X,
\end{align*}
and $ s_X(f \otimes \id) s_X^* = t_X(\id \otimes f) t_X^* $ for all $ f \in \BT(X,X) $. 
The \emph{quantum trace} $ \Tr_q: \BT(X,X) \rightarrow \mathbb{C} $ of $ X $ is defined by $ \Tr_q(f) = s_X(f \otimes \id)s_X^* = t_X(\id \otimes f) t_X^* $, 
and the \emph{quantum dimension} of $ X $ is $ \dim_q(X) = \Tr_q(\id_X) $. 
Every rigid \cst-tensor category $ \BT $ is semisimple, that is, every object of $ \BT $ is isomorphic to a finite direct sum of simple objects. 
We write $ \Irr(\BT) $ for the set of isomorphism classes of simple objects in $ \BT $, and choose representatives $ X_i \in \BT $ for elements $ i = [X_i] \in \Irr(\BT) $, 
with the convention that we also write $ i = 0 $ for the class $ [\one] $. 

The \emph{fusion algebra} $ \mathbb{C}[\BT] $ is the vector space with basis $ \Irr(\BT) $ equipped with the fusion product 
$$
[X_i] \cdot [X_j] = \sum_{k \in \Irr(\BT)} N_{ij}^k [X_k], 
$$
where $ N^k_{ij} = \dim(\BT(X_i \otimes X_j, X_k)) $, and the $ * $-structure determined by $ [X_i]^* = [X_i^\vee] $. We will follow the usual abuse of notation and identify $X\in \Irr(\BT) $ 
with its class $[X]$. The regular representation $ \lambda: \mathbb{C}[\BT] \rightarrow \B(\ell^2(\Irr(\BT))) $ is defined by $ \lambda(X)(Y) = X \cdot Y $. 
The \emph{tube algebra} $ \Tub(\BT) $ is 
\begin{equation}
\Tub(\BT) = \bigoplus_{i, j, k \in \Irr(\BT)} \BT(X_k \otimes X_i, X_j \otimes X_k) 
\label{eq:tub_alg_defn}
\end{equation}
equipped with a suitable multiplication and $ \star $-structure, see \cite[Definition~3.3]{GHOSH_JONES_annularreptheory}. 
Let us only note that $ \Tub(\BT) $ is a non-unital $ \star $-algebra with local units, which are given by the projections $ p_i = \id \in \BT(\one \otimes X_i, X_i \otimes \one) = \BT(X_i, X_i) \subseteq \Tub(\BT) $ for $ i \in \Irr(\BT) $. 
Moreover, the corner $ p_0 \Tub(\BT) p_0 $ corresponding to $ [\one] \in \Irr(\BT) $ is canonically isomorphic to the fusion algebra $ \mathbb{C}[\BT] $, see \cite[Proposition~3.1]{GHOSH_JONES_annularreptheory}.

There is a natural faithful positive trace $ \Tr\colon \Tub(\BT) \rightarrow \mathbb{C} $ which vanishes on $ \BT(X_k \otimes X_i, X_j \otimes X_k) $ 
for $ i \neq j $ or $ X_k \neq \one $, and is given by the quantum trace $ \Tr_q $ on $ \BT(\one \otimes X_i, X_i \otimes \one) $ for all $ i \in \Irr(\BT) $. 
We write $ \LL^2(\Tub(\BT)) $ for the associated GNS-Hilbert space. 
This yields the regular representation $ \lambda\colon \Tub(\BT) \rightarrow \B(\LL^2(\Tub(\BT))) $ of $ \Tub(\BT) $. 
The reduced \cst-algebra $ \mrm{C}^*_\red(\Tub(\BT)) $ and the von Neumann algebra $ \mc{L}(\Tub(\BT)) $ are defined as the closure of $ \lambda(\Tub(\BT)) $ 
in the operator norm and the weak operator topology, respectively. The map $ \Tr $ extends to a n.s.f trace on $ \mc{L}(\Tub(\BT)) $ by \cite[Proposition 3.10]{POPA_SHLYAKHTENKO_VAES_l2betti}. 

Let us next recall the definition of multipliers on rigid \cst-tensor categories from the work of Popa-Vaes \cite[Section~3]{PVcstartensor}. 

\begin{definition} 
Let $ \BT $ be a rigid \cst-tensor category. 
A \emph{multiplier} on $ \BT $ is a family $ \theta = (\theta_{X,Y}) $ of linear maps $ \theta_{X,Y}: \BT(X \otimes Y, X \otimes Y) \rightarrow \BT(X \otimes Y, X \otimes Y) $ 
for $ X, Y \in \BT $ such that 
\begin{align*}
\theta_{X_2,Y_2}(gfh^*) &= g \theta_{X_1, Y_1}(f) h^*, \\
\theta_{X_2 \otimes X_1, Y_1 \otimes Y_2}(\id_{X_2} \otimes f \otimes \id_{Y_2}) &= \id_{X_2} \otimes \theta_{X_1,Y_1}(f) \otimes \id_{Y_2},
\end{align*}  
for all $ X_i, Y_i \in \BT $, $ f \in \BT(X_1 \otimes Y_1, X_1 \otimes Y_1) $ and $ g, h \in \BT(X_1, X_2) \otimes \BT(Y_1, Y_2) \subseteq \BT(X_1 \otimes Y_1, X_2 \otimes Y_2) $. \\
A multiplier $ \theta = (\theta_{X,Y}) $ on $ \BT $ is said to be completely positive (or a \emph{CP multiplier}) if all the maps $ \theta_{X,Y} $ are completely positive. 
A multiplier $ \theta = (\theta_{X,Y}) $ on $ \BT $ is said to be completely bounded (or a \emph{CB multiplier}) if all the maps $ \theta_{X,Y} $ are completely bounded 
and $ \| \theta \|_{cb} = \sup_{X,Y \in \BT} \|\theta_{X,Y}\|_{cb} < \infty $. 
\end{definition} 

It is shown in \cite[Proposition 3.6]{PVcstartensor} that multipliers on $ \BT $ are in canonical bijection with functions $ \Irr(\BT) \rightarrow \mathbb{C} $.  We will often identify a multiplier $\theta = (\theta_{X,Y})$ with its associated function $\theta = (\theta(k))_{k\in\Irr(\BT)}$.  Note that we have $\|(\theta(k))_{k\in\Irr(\BT)}\|_\infty \leq \|\theta\|_{cb}$.

Let us write $ \M_{cb}(\BT) $ for the space of CB multipliers on $ \BT $. Via composition of maps and the CB norm this becomes naturally a Banach algebra.  From the definition of the correspondence between functions on $ \Irr(\BT) $ and multipliers, compare \cite[Formula~(3.5)]{PVcstartensor}, it follows that the product on $ \M_{cb}(\BT) $ corresponds to pointwise multiplication of functions.  It is shown in \cite[Corollary 5.3]{ARANO_DELAAT_WAHL_fourier} that $ \M_{cb}(\BT) $ is a dual Banach algebra, whose predual $ Q(\BT) $ can be constructed using 
the tube algebra of $ \BT $.  More specifically, if $ \theta \in \M_{cb}(\BT) $ is a CB multiplier, define $ M_\theta\colon \Tub(\BT) \rightarrow \Tub(\BT) $ by 
$$
M_\theta(f) = \theta(k) f, \qquad f \in \BT(X_k \otimes X_i, X_j \otimes X_k) \subseteq \Tub(\BT). 
$$
Due to \cite[Proposition 5.1]{AVhecke} the map $ M $ gives an isometric embedding of $ \M_{cb}(\BT) $ into the space $ \CB^\sigma(\mc{L}(\Tub(\BT))) $ 
of normal CB maps on $ \mc{L}(\Tub(\BT)) $, 
and also an isometric embedding with weak$^*$-closed image $ \M_{cb}(\BT) \rightarrow \CB(\mrm{C}^*_{\red}(\Tub(\BT)), \mc{L}(\Tub(\BT))) $.  
Then the predual $ Q(\BT) $ of $ \M_{cb}(\BT) $ obtained in \cite{ARANO_DELAAT_WAHL_fourier} is constructed as the resulting quotient of the 
predual $ \mrm{C}^*_\red(\Tub(\BT)) \wh{\otimes} \mc{L}(\Tub(\BT))_* $ of $ \CB(\mrm{C}^*_\red(\Tub(\BT)), \mc{L}(\Tub(\BT))) $ -- compare Theorem \ref{thm5}. 

We can approximate elements of $ Q(\BT) $ by taking tensor products of elements in $ \Tub(\BT) $ and vector functionals associated to vectors from $ \LL^2(\Tub(\BT)) $. 
Noting that $ \Tub(\BT) $ is dense in $ \LL^2(\Tub(\BT)) $, a linearly dense collection of functionals in $\mc{L}(\Tub(\BT))_*$ is given by $ \mc{L}(\Tub(\BT)) \ni T \mapsto \ismaa{f}{Tg} = \Tr(f^*Tg) = \Tr(Tgf^*)$ for $f,g\in \Tub(\BT) $.  As $\Tub(\BT)$ has local units, we have $\{gf^*\,|\,f,g\in \Tub(\BT)\}=\Tub(\BT)$ and it suffices to look at functionals of the form $T \mapsto \Tr(Tf)$ for $f\in \Tub(\BT)$.  Under this identification, the canonical pairing of $ f \otimes g \in \Tub(\BT) \odot \Tub(\BT) \subseteq \mrm{C}^*_\red(\Tub(\BT)) \wh{\otimes} \mc{L}(\Tub(\BT))_* $ with $ \theta \in \M_{cb}(\BT) $ becomes 
\begin{equation}
\la \theta , f  \otimes g\ra  = \Tr(g M_\theta(f)).  \label{eq:QT_dense}
\end{equation}

Let us define a weighted $ \ell^1 $-norm on $ \mrm{c}_{00}(\Irr(\BT)) $ by 
$$
\|f\|_1 = \sum_{k \in \Irr(\BT)} \dim_q(X_k) |f(k)|, 
$$
and denote by $ \ell^1(\Irr(\BT)) $ the corresponding completion, compare \cite[Remark 10.4]{PVcstartensor}.  
The weighting by quantum dimensions ensures that admissible $ * $-representations of $ \mathbb{C}[\BT] $ in the sense of \cite[Definition 4.1]{PVcstartensor} 
extend to contractive $ * $-representations of $ \ell^1(\Irr(\BT)) $. 
Note that there is a contractive embedding $ \iota\colon \ell^1(\Irr(\BT)) \rightarrow \M_{cb}(\BT)^* $ given by 
$$ 
\iota(\omega)(\theta) = \sum_{k \in \Irr(\BT)} \dim_q(X_k) \omega(k) \theta(k)
\qquad (\omega\in \ell^1(\Irr(\BT)), \theta\in \M_{cb}(\BT)).
$$ 

\begin{lemma} \label{l1density}
Let $ \BT $ be a rigid \cst-tensor category. Then the Banach space $ Q(\BT) $ can be identified with the closure of $ \ell^1(\Irr(\BT)) $ in $ \M_{cb}(\BT)^* $ 
under the embedding $\iota$.
\end{lemma} 

\begin{proof} 
It follows from the explicit formulas that the image of the subspace $ \Tub(\BT) \odot \Tub(\BT) \subseteq\mrm{C}^*_\red(\Tub(\BT)) \wh{\otimes} \mc{L}(\Tub(\BT))_* $ 
in $ \M_{cb}(\BT)^* $ agrees with the image of $ \mrm{c}_{00}(\Irr(\BT)) $ under the map $ \iota $.  
Indeed, that we obtain all of $\mrm{c}_{00}(\Irr(\BT))$ in this way can be seen by considering $f,g\in p_0\Tub(\BT)p_0 \cong \mathbb C[\BT]$ in \eqref{eq:QT_dense}.
The claim therefore follows from density of the former space inside $ Q(\BT) $. 
\end{proof} 

Remembering that the weak$^*$-topology on $ \M_{cb}(\BT) $ means the one induced by the predual $Q(\BT)$, we shall now give the following definition. 

\begin{definition} 
Let $ \BT $ be a rigid \cst-tensor category. We say that $ \BT $ has the \emph{approximation property (AP)} if there exists a net of finitely supported CB multipliers of $ \BT $ 
converging to $ \I $ in the weak$^*$-topology of $ \M_{cb}(\BT) $. 
\end{definition}

Comparing with \cite[Definition 5.1]{PVcstartensor} we see that every weakly amenable rigid \cst-tensor category has AP. 
Indeed, a uniformly bounded net of finitely supported CB multipliers converging pointwise to $ \I $ 
converges also in the weak$^*$-topology since $ \mrm{c}_{00}(\Irr(\BT)) $ is dense in $ Q(\BT) $ by Lemma \ref{l1density}.

Next recall the notion of central approximation property for discrete quantum groups from Definition \ref{defcentralAP}. 
We aim to show that the central approximation property for a discrete quantum group $\GGamma$ is equivalent to the approximation property for the rigid \cst-tensor 
category $ \Corep(\GGamma) $ of finite dimensional unitary corepresentations of $\GGamma$ (i.e.~representations of $\wh\GGamma$). To make the notation more coherent, we will write in this section $\CC[\GGamma]=\Pol(\wh\GGamma)$ and $\mrm{C}^*_{\red}(\GGamma)=\mrm{C}(\wh\GGamma)$, $\mc{L}(\GGamma)=\LL^{\infty}(\wh\GGamma)$, and use the same conventions for the Drinfeld double $D(\GGamma)$.

We shall first discuss the relation between categorical AP for $ \Corep(\GGamma) $ and AP for the Drinfeld double $ D(\GGamma) $ of $\GGamma$. 
Recall from Section~\ref{sec:dbl_cp:prelim} that $ \LL^\infty(D(\GGamma)) = \ell^\infty(\GGamma) \bar{\otimes} \mc{L}(\GGamma) $ with the coproduct 
$$
\Delta_{D(\GGamma)} = (\id \otimes \chi \otimes \id)(\id \otimes \ad(\ww) \otimes \id)(\Delta \otimes \wh{\Delta}), 
$$ 
where $ \ww \in \ell^\infty(\GGamma) \bar{\otimes} \mc{L}(\GGamma) $ is the Kac-Takesaki operator of $\GGamma$.  
Note also that we have a canonical identification of the center $ \mc{Z}\ell^\infty(\GGamma) $ of $ \ell^\infty(\GGamma) $ 
with $ \ell^\infty(\Irr(\Corep(\GGamma))) = \ell^\infty(\Irr(\wh \GGamma))$. In particular, 
every multiplier $ \theta \in \M_{cb}(\Corep(\GGamma)) $ can be viewed as a (central) element of $ \ell^\infty(\GGamma) $. 

In what follows we will use the notion of an \emph{algebraic quantum group}, see \cite{VanDaeleAlgebraic}, \cite[Section 3.2]{VoigtYuncken}. By definition, an algebraic quantum group is described via a multiplier Hopf $\star$-algebra for which there exists a positive left invariant functional and a positive right invariant functional.  For example, if $\GGamma$ is a discrete quantum group then $\mrm{c}_{00}(\GGamma)$ 
and $\CC[\GGamma]$ equipped with their respective comultiplications and Haar integrals are examples of algebraic quantum groups. Every algebraic quantum group gives rise to a locally compact quantum group in the sense of Kustermans-Vaes via an appropriate completion procedure, see \cite{KustermansAnalytic}. Moreover, one finds that all elements of the underlying 
multiplier Hopf $\star$-algebra are contained in the Fourier algebra of the locally compact quantum group, see e.g.~the end of Section $1$ in \cite{KustermansAnalytic}.

The Drinfeld double $D(\GGamma)$ of a discrete quantum group $\GGamma$ and its dual $\wh{D(\GGamma)}$ are also algebraic quantum 
groups. The corresponding multiplier Hopf $ * $-algebras are $\mrm{c}_{00}(\GGamma)\odot \CC[\GGamma]\subseteq \LL^{\infty}(D(\GGamma))$ and 
\[
\mc{D}(D(\GGamma))=\CC[\GGamma]\bowtie \mrm{c}_{00}(\GGamma)=\lin\{\gamma_1(\wh x)\gamma_2(x)\,|\,\wh x\in \CC[\GGamma],x\in \mrm{c}_{00}(\GGamma)\}\subseteq 
\mc{L}(D(\GGamma)),
\]
where
\[\begin{split}
\gamma_1&\colon \mc{L}(\GGamma)\rightarrow \mc{L}( D(\GGamma) )\colon \wh x\mapsto \wh x\otimes \I , \\ 
\gamma_2&\colon \ell^{\infty}(\GGamma)\rightarrow\mc{L}( D(\GGamma) )\colon x\mapsto Z^*(\I\otimes x)Z
\end{split}\]
are the maps introduced in Lemma \ref{lemma11}. We will write $\gamma_1(\wh x)\gamma_2(x)=\wh{x}\bowtie x$ for $\wh x\in \CC[\GGamma],x\in \mrm{c}_{00}(\GGamma)$.

\begin{lemma} \label{drinfeldmultiplier}
Let $\GGamma$ be a discrete quantum group and let $ D(\GGamma) $ be its Drinfeld double. There is an isometric embedding 
$$ 
N\colon \M_{cb}(\Corep(\GGamma)) \rightarrow \CB^\sigma(\mc{L}(D(\GGamma)))\colon \theta\mapsto N_\theta
$$ 
given by 
$$ 
N_\theta(U^\alpha_{i,j} \bowtie x_\beta) = \theta(\alpha) U^\alpha_{i,j} \bowtie x_\beta
$$ 
for $ U^\alpha_{i,j}\in \mathbb{C}[\GGamma], x_\beta \in \B(\msf{H}_\beta) \subseteq \mrm{c}_{00}(\GGamma) $. 
If $ \theta \in \M_{cb}(\Corep(\GGamma)) $ then $ \theta \otimes \I \in \M^l_{cb}(\A(D(\GGamma)))$ $ \subseteq \LL^\infty(D(\GGamma)) $
and $ N_\theta = \Theta^l(\theta \otimes \I) $.
\end{lemma} 

\begin{proof}
We wish to apply the results of \cite[Section~3]{NESHVEYEV_YAMASHITA_tube}. For this we need to work with the annular algebra
\[ \Tub(\GGamma) = \bigoplus_{\alpha,\beta\in\Irr(\wh \GGamma)} 
\Big( \bigoplus_{\gamma\in\Irr(\wh \GGamma)} \Mor(\gamma\otimes\alpha, \beta\otimes\gamma) \Big) \otimes \B(\msf H_{\overline{\alpha}}, \msf H_{\overline{\beta}}), \]
which is equipped with the multiplication given by the product from $\Tub(\Corep(\GGamma))$ in \eqref{eq:tub_alg_defn} and the composition of operators between 
the Hilbert spaces $\msf H_\gamma$. 
We refer to \cite[Section~3]{GHOSH_JONES_annularreptheory} for the general definition of annular algebras associated with full weight sets. 
 
We again obtain a trace on $\Tub(\GGamma)$ and so can perform the GNS construction, and construct the associated von Neumann algebra $\mc L(\Tub(\GGamma))$. 
Furthermore, \cite[Proposition 5.1]{AVhecke} applies, and we obtain a map $ \wt{M}\colon \M_{cb}(\Corep(\GGamma)) \rightarrow \CB^\sigma(\mc{L}(\Tub(\GGamma))) $ 
given by $ \wt{M}_\theta(f) = \theta(\gamma) f $ 
for $ f = f' \otimes T \in \Mor(\gamma \otimes \alpha, \beta \otimes \gamma) \otimes 
\B(\msf H_{\overline{\alpha}}, \msf H_{\overline{\beta}}) \subseteq \Tub(\GGamma) \subseteq \mc L(\Tub(\GGamma)) $, which is a well-defined isometric embedding. 

It is shown in \cite[Theorem 3.5]{NESHVEYEV_YAMASHITA_tube} that there is a $\star$-isomorphism between $ \Tub(\GGamma) $ 
and the algebraic convolution algebra $ \D(D(\GGamma)) = \mathbb{C}[\GGamma] \bowtie \mrm{c}_{00}(\GGamma) $ of the Drinfeld double of $\GGamma$. 
Under this isomorphism, the trace $ \Tr $ on $ \Tub(\GGamma) $ does not correspond to the left invariant functional on $ \D(D(\GGamma)) $ on the nose, 
but both functionals can be obtained from one another by multiplication with a positive invertible element in the algebraic multiplier algebra of $ \Tub(\GGamma) \cong \D(D(\GGamma)) $.
It follows that the regular representations of $ \Tub(\GGamma) \cong \D(D(\GGamma))) $ on $ \LL^2(\Tub(\GGamma))$ and $\LL^2(D(\GGamma))$ are unitarily equivalent, which means that the 
isomorphism in \cite[Theorem 3.5]{NESHVEYEV_YAMASHITA_tube} induces a normal $\star$-isomorphism $ \mc{L}(\Tub(\GGamma)) \cong \mc{L}(D(\GGamma)) $, 
which restricts to a $\star$-isomorphism $\mrm{C}^*_{\red}(\Tub(\BT)) \cong \mrm{C}^*_{\red}(D(\GGamma)) $. 

Inspecting the formulas in \cite{NESHVEYEV_YAMASHITA_tube} one checks that $ \wt{M}_\theta\colon \Tub(\GGamma) \rightarrow \Tub(\GGamma) $ identifies under 
the isomorphism $\Tub(\GGamma) \cong \D(D(\GGamma)))$ with the map $ N_\theta\colon \D(D(\GGamma)) \rightarrow \D(D(\GGamma)) $ in the statement of the lemma. 
Consequently, we see that $N_\theta$ extends to a normal CB map on $\mc{L}(D(\GGamma))$.

An explicit formula for the multiplication in $\mc D(D(\GGamma))$ is given in \cite[page~219]{VoigtYuncken}, though be aware that \cite{VoigtYuncken} uses a different convention to us, with the factors $\mathbb C[\GGamma]$ and $\mrm{c}_{00}(\GGamma)$ swapped around; in what follows, we make the necessarily changes to follow our conventions.  It follows that, for $x\in \mrm{c}_{00}(\GGamma)$, we have that $\gamma_2(x) \gamma_1(U^\alpha_{i,j}) \in \lin\{ \gamma_1(U^\alpha_{k,l}) \gamma_2(y) \,|\, 1\leq k,l\leq \dim(\alpha), y\in \mrm{c}_{00}(\GGamma) \}$ and so $N_\theta( \gamma_2(x) \gamma_1(U^\alpha_{i,j})) = \theta(\alpha) \gamma_2(x) \gamma_1(U^\alpha_{i,j})$.

From Section~\ref{sec:dbl_cp:prelim}, we find that
\[ \ww^{D(\GGamma)*} = Z_{34}^* \wh\ww_{24}^* Z_{34} \ww_{13}^* = (\id\otimes\gamma_2)(\wh\ww^*)_{234}(\id\otimes\gamma_1)(\ww^*)_{134}, \]
where $\ww = \ww^\GGamma$.  Given $\alpha,i,j$, there is $\omega\in\ell^1(\GGamma)$ with $(\omega\otimes\id)(\ww^*) = U^\alpha_{i,j}$, because $\ww^* = \chi(\wh\ww)$ and $\wh\ww$ is given by \eqref{eq:hatW_cmpt}.  Then for $\wh\omega\in \LL^1(\wh\GGamma)$,
\begin{align*}
(\omega\otimes\wh\omega\otimes N_\theta)(\ww^{D(\GGamma)*})
&= N_\theta\big( \gamma_2( (\wh\omega\otimes\id)(\wh\ww^*) ) \gamma_1(U^\alpha_{i,j})) \big)
= \theta(\alpha) \gamma_2( (\wh\omega\otimes\id)(\wh\ww^*) ) \gamma_1(U^\alpha_{i,j}))
\end{align*}
using the previous observation.  Given $x_\beta \in \B(\msf H_\beta) \subseteq \mrm{c}_{00}(\GGamma)$, notice that $\theta x_\beta = \theta(\beta) x_\beta$, and so $\omega \theta = \theta(\alpha) \omega$, as $\omega \in \B(\msf{H}_\alpha)^*$.  Thus
\[ (\omega\otimes\wh\omega\otimes N_\theta)(\ww^{D(\GGamma)*})
= \theta(\alpha) (\omega\otimes\wh\omega\otimes\id)(\ww^{D(\GGamma)*})
= (\omega\otimes\wh\omega\otimes\id)( ((\theta\otimes\I)\otimes\I)\ww^{D(\GGamma)*} ). \]
As such $\omega$ are linearly dense, it follows that $(\id\otimes N_\theta)(\ww^{D(\GGamma)*}) = ((\theta\otimes\I)\otimes\I)\ww^{D(\GGamma)*}$ or equivalently, $(N_\theta\otimes\id)(\ww^{\wh{D(\GGamma)}}) = (\I \otimes (\theta\otimes\I) ) \ww^{\wh{D(\GGamma)}}$.

We conclude that $ \theta \otimes \I \in \LL^\infty(D(\GGamma)) = \ell^\infty(\GGamma) \bar{\otimes} \mc{L}(\GGamma) $ satisfies $ \theta\otimes\I \in \M^l_{cb}(D(\GGamma)) $ 
and $ \Theta^l(\theta \otimes \I) = N_\theta $ as claimed. 
\end{proof} 

Similar to Proposition~\ref{prop:new_wstar_cent}, we can restrict $N_\theta$ to a map in $\CB\bigl(\mrm{C}^*_{\red}(D(\GGamma)) , \mc{L}(D(\GGamma))\bigr)$.  As $\mrm{C}^*_{\red}(D(\GGamma))$ is weak$^*$-dense in $\mc{L}(D(\GGamma))$, Kaplansky density shows that this restriction map is a (complete) isometry.

\begin{lemma}\label{lemma23}
Let $\GGamma$ be a discrete quantum group.  The map
\[
N\colon \M_{cb}(\Corep(\GGamma))\rightarrow \CB\bigl( \mrm{C}^*_{\red}(D(\GGamma)),\mc{L}(D(\GGamma))\bigr)
\]
of Lemma \ref{drinfeldmultiplier} is weak$^*$-weak$^*$-continuous.
\end{lemma}

\begin{proof}
As in the proof of Lemma~\ref{drinfeldmultiplier}, we identify $\mrm{C}^*_{\red}(D(\GGamma))$ with $\mrm{C}^*_{\red}( \Tub(\GGamma) )$ and $\mc L(D(\GGamma))$ with $\mc L(\Tub(\GGamma))$. 
Then $N$ identifies with $ \widetilde{M}: \M_{cb}(\Corep(\GGamma)) \rightarrow \CB(\mrm{C}^*_{\red}( \Tub(\GGamma) ),\mc{L}( \Tub(\GGamma)))$, 
and we note that $\widetilde{M}$ is again isometric.

It hence suffices to show that $\widetilde{M}$ is weak$^*$-weak$^*$-continuous, for which we shall apply Lemma~\ref{biduallemma} with $\alpha = \widetilde{M}, E = Q(\Corep(\GGamma))$ 
and $ F = \mrm{C}^*_{\red}( \Tub(\GGamma) ) \wh\otimes \mc{L}( \Tub(\GGamma) )_*$. From the definition of $\Tub(\GGamma)$ in the proof of Lemma~\ref{drinfeldmultiplier}, 
we see that the elements of the form $\omega = (f\otimes T) \otimes \Tr(\cdot(g\otimes S))$ with 
$f\otimes T \in \Mor(\gamma \otimes \alpha, \beta \otimes \gamma) \otimes \B(\msf H_{\overline{\alpha}}, \msf H_{\overline{\beta}}) \subseteq \Tub(\GGamma)$ 
and $g\otimes S \in \Mor(\gamma' \otimes \alpha', \beta' \otimes \gamma') \otimes \B(\msf H_{\overline{\alpha'}}, \msf H_{\overline{\beta'}})$ form a linearly dense subset $ D \subseteq F $.

For such an $\omega$, given $\theta\in \M_{cb}(\Corep(\GGamma))$ we calculate that
\[ \la \widetilde{M}_\theta, \omega \ra 
= \la \widetilde{M}_\theta , (f\otimes T) \otimes \Tr(\cdot(g\otimes S)) \ra
= \Tr\big( \widetilde{M}_\theta(f\otimes T) (g\otimes S) \big)
= \theta(\gamma) \Tr( fg \otimes TS). \]
Hence the function $\theta \mapsto \la \widetilde{M}_\theta, \omega \ra$ lies in the image of $\mrm{c}_{00}(\Irr(\wh\GGamma))$ inside $\M_{cb}(\Corep(\GGamma))^*$ -- 
in particular in $Q(\Corep(\GGamma))$ by Lemma~\ref{l1density}. This verifies the condition of Lemma~\ref{biduallemma}, and the claim follows.
\end{proof}

Write $ \mc{Z}\M_{cb}^l(\A(\GGamma)) $ for the center of the Banach algebra $ \M_{cb}^l(\A(\GGamma)) $ and note that $\mc{Z} \M^l_{cb}(\A(\GGamma))=\M^l_{cb}(\A(\GGamma))\cap \mc{Z} \ell^{\infty}(\GGamma)$. Furthermore, observe that $ \mc{Z}\M_{cb}^l(\A(\GGamma)) \subseteq \M_{cb}^l(\A(\GGamma)) $ is weak$^*$-closed, hence $\mc{Z} \M^l_{cb}(\A(\GGamma))$ is a dual space, with distinguished predual being a quotient of $Q^l(\A(\GGamma))$.

\begin{lemma} \label{multiplierlemma}
Let $\GGamma$ be a discrete quantum group. Then there is a canonical isometric isomorphism $\M_{cb}(\Corep(\GGamma)) \cong \mc{Z}\M^l_{cb}(\A(\GGamma)) $, 
and this isomorphism is a weak$^*$-weak$^*$-homeomorphism. 
\end{lemma} 

\begin{proof} 
Let $ \theta\colon \Irr(\wh \GGamma) \rightarrow \mathbb{C} $ be a bounded function, identified with a central element of $\ell^{\infty}(\GGamma)$. We shall first verify that $ \theta $ is contained in $ \M_{cb}(\Corep(\GGamma)) $ if and only if it is contained in $ \M_{cb}^l(\A(\GGamma)) $, and that the corresponding CB norms agree. 

Firstly, let $\theta \in \mc Z \M^l_{cb}(\A(\GGamma))$, considered as a map $\Irr(\wh \GGamma)\rightarrow\mathbb C$. In \cite[Section 6]{PVcstartensor}, the associated centraliser is denoted by $\Psi_\theta$, see \cite[Equation~(6.1)]{PVcstartensor}, and then \cite[Proposition 6.1]{PVcstartensor} shows that if $\Psi_\theta$ is completely bounded, which under our assumption it is, then the multiplier on $\Corep(\GGamma)$ given by $\theta$ is also completely bounded, with $ \|\theta\|_{\M_{cb}(\Corep(\GGamma))} \leq \|\theta\|_{\M_{cb}^l(\A(\GGamma))} $.

Conversely, when $\theta \in \M_{cb}(\Corep(\GGamma))$, then by Lemma~\ref{drinfeldmultiplier}, we know that $\theta\otimes\I \in \M^l_{cb}(\A(D(\GGamma)))$.  We again use the normal injective $\star$-homomorphism $\gamma_1 : \mc L(\GGamma) \rightarrow \mc L(D(\GGamma))$ which identifies $\GGamma$ as a closed quantum subgroup of $D(\GGamma)$, see Lemma~\ref{lemma11}.  
As $\gamma_1(\wh x) = \wh x \otimes\I$, it follows from Lemma~\ref{drinfeldmultiplier} that $\Theta^l(\theta\otimes\I) = N_\theta$ leaves the image of $\gamma_1$ invariant, and so we obtain a map $L_\theta \in \CB^\sigma( \mc L(\GGamma) )$ with $\gamma_1 L_\theta = N_\theta \gamma_1$ and $\|L_\theta\|_{cb} \leq \|N_\theta\|_{cb} = \|\theta\|_{cb}$.  In particular, 
$L_\theta(U^\alpha_{i,j}) = \theta(\alpha) U^\alpha_{i,j}$ for each $\alpha,i,j$. 
Thus $\theta \in \M^l_{cb}(\A(\GGamma)) $, with $ \|\theta\|_{\M_{cb}(\Corep(\GGamma))} \geq \|\theta\|_{\M_{cb}^l(\A(\GGamma))} $.

As these identifications are mutual inverses, we have shown that $\M_{cb}(\Corep(\GGamma)) \cong \mc{Z}\M^l_{cb}(\A(\GGamma)) $ isometrically. Let $ \gamma : \M_{cb}(\Corep(\GGamma)) \rightarrow \mc{Z}\M^l_{cb}(\A(\GGamma)) $ be the resulting isometric isomorphism, which we claim is weak$^*$-weak$^*$-continuous.  We again use Lemma~\ref{biduallemma}, with $E = Q(\Corep(\GGamma))$ and $F$ the predual of $\mc{Z}\M^l_{cb}(\A(\GGamma))$, with $D\subseteq F$ to be constructed.  As the predual of $\mc{Z}\M^l_{cb}(\A(\GGamma))$ is a quotient of $Q^l(\A(\GGamma))$, it suffices to take $D$ to be the image under the quotient map of a linearly dense subset $D'$ of $Q^l(\A(\GGamma))$.  We take $D' \subseteq \ell^1(\GGamma) \subseteq Q^l(\A(\GGamma))$ to consist of all linear functionals $\omega$ constructed by choosing $x\in\mrm{c}_{00}(\GGamma)$ and defining 
$\la y, \omega \ra = \sum_{\alpha\in\Irr(\wh\GGamma)} \dim_q(\alpha) \Tr_\alpha(y_\alpha x_\alpha)$ for $y\in\ell^\infty(\GGamma)$. Given $\theta \in \M_{cb}(\Corep(\GGamma))$ and $\omega\in D$ induced by $x\in \mrm{c}_{00}(\GGamma)$, we see that
\[\la \gamma(\theta), \omega \ra = \sum_{\alpha\in\Irr(\wh\GGamma)} \dim_q(\alpha) \theta(\alpha) \Tr_\alpha(x_\alpha), \]
where the sum is finite. Hence if we set $z = (\Tr_\alpha(x_\alpha))_{\alpha\in\Irr(\wh\GGamma)} \in\mrm{c}_{00}(\Irr(\wh\GGamma))$ 
then $\la \gamma(\theta), \omega \ra = \la \theta, \iota(z) \ra$, where $\iota$ is the embedding of $ \ell^1(\Irr(\wh \GGamma)) $ into $ Q(\Corep(\GGamma)) $ as in Lemma~\ref{l1density}. It follows that $\omega\circ\gamma \in Q(\Corep(\GGamma))$, and hence $\gamma^*\kappa_{F}(D) \subseteq \kappa_{E}(E)$. Now Lemma~\ref{biduallemma} yields the claim.
\end{proof} 

Let us now compare the categorical approximation property of $ \Corep(\GGamma) $ with the central approximation property of $ \GGamma $. 

\begin{proposition} \label{centralap}
A discrete quantum group $ \GGamma $ has central AP if and only if $ \Corep(\GGamma) $ has AP. 
\end{proposition} 

\begin{proof} 
The claim follows from Lemma \ref{multiplierlemma}, as the isomorphism $\M_{cb}(\Corep(\GGamma))\simeq \mc{Z} \M^l_{cb}(\A(\GGamma))$ is a unital weak$^*$-weak$^*$-homeomorphism which restricts 
to $\mrm{c}_{00}(\Irr(\wh \GGamma))\simeq \mc{Z} \mrm{c}_{00}(\GGamma)$.
\end{proof} 

As a consequence of Proposition \ref{centralap} we see that central AP is invariant under monoidal equivalence. 

\begin{corollary} 
Let $\GGamma$ and $\LLambda$ be discrete quantum groups such that $\wh\GGamma$ and $\wh{\LLambda}$ are monoidally equivalent. Then $\GGamma$ has central AP if and only if $\LLambda$ has central AP. 
\end{corollary} 

\begin{proof} 
According to the definition of monoidal equivalence \cite{BdRV}, the \cst-tensor categories $\Corep(\GGamma)$ and $\Corep(\LLambda)$ are unitarily monoidally equivalent. 
This means that $\Corep(\GGamma)$ has AP if and only if $\Corep(\LLambda)$ has AP. Due to Proposition~\ref{centralap} this yields the claim. 
\end{proof}

Finally, let us relate AP of $\Corep(\GGamma)$ and $D(\GGamma)$. 

\begin{proposition} \label{drinfelddoubleap}
Let $\GGamma$ be a discrete quantum group such that $\Corep(\GGamma)$ has AP. Then the Drinfeld double $D(\GGamma)$ has AP. If $\GGamma$ is unimodular, then the converse also holds: AP of $D(\GGamma)$ implies AP of $\Corep(\GGamma)$.
\end{proposition} 

\begin{proof} 
Due to Lemma~\ref{drinfeldmultiplier} we have an isometric embedding $\M_{cb}(\Corep(\GGamma)) \rightarrow \M_{cb}^l(\A(D(\GGamma))) $ given by $\theta \mapsto \theta \otimes \I$.  As $\Corep(\GGamma)$ has AP, there is a net $(\theta_i)_{i\in I}$ of finitely supported elements in $\M_{cb}(\Corep(\GGamma))$ with $ \theta_i \xrightarrow[i\in I]{}\I $ in the weak$^*$-topology.  By Lemma~\ref{lemma23}, it follows that the net $(N_{\theta_i})_{i\in I}$ converges weak$^*$ to the inclusion map in $\CB\bigl( \mrm{C}^*_{\red}(D(\GGamma)),\mc{L}(D(\GGamma))\bigr)$.  As $N_{\theta_i} = \Theta^l(\theta_i\otimes\I)$ for each $i$, by Theorem~\ref{thm5} this means that $\theta_i\otimes \I \xrightarrow[i\in I]{} \I\otimes \I$ weak$^*$ in $\M_{cb}^l(\A(D(\GGamma)))$. Since the elements $\theta_i\otimes \I$ belong to the multiplier Hopf $\star$-algebra, they also belong to the Fourier algebra $\A(D(\GGamma))$ and we conclude that $D(\GGamma)$ has AP. 

If $D(\GGamma)$ has AP then by Theorem~\ref{thm2} the same is true for $\GGamma$ since $\GGamma$ is a closed quantum subgroup of $D(\GGamma)$. If $\GGamma$ is in addition unimodular, 
then AP of $\GGamma$ implies central AP of $\GGamma$ by Proposition \ref{prop19}, and consequently AP of $\Corep(\GGamma)$ due to Proposition \ref{centralap}.
\end{proof}

\bibliographystyle{plain}
\bibliography{bibliografia}

\end{document}